\documentclass[reqno, 11pt]{amsart}

\UseRawInputEncoding
\usepackage{mathrsfs}
\usepackage{amscd}
\usepackage{amsmath}
\usepackage{latexsym}
\usepackage{amsfonts}
\usepackage{amssymb}
\usepackage{amsthm}
\usepackage{graphicx}
\usepackage{hyperref}
\usepackage{makecell}
\usepackage{array}
\usepackage{booktabs}
\usepackage{multirow}
\usepackage{color,xcolor}

\newtheorem{theorem}{Theorem}[section]
\newtheorem{proposition}[theorem]{Proposition}
\newtheorem{corollary}[theorem]{Corollary}
\newtheorem{lemma}[theorem]{Lemma}
\newtheorem{definition}[theorem]{Definition}

\newtheorem{remark}[theorem]{Remark}

\numberwithin{equation}{section}

\title[Decay estimates for bi-Schr\"odinger operators  in dimension one ]
{Decay estimates for bi-Schr\"odinger operators  in dimension one }

\author{Avy Soffer,\ Zhao Wu \ and Xiaohua Yao  \textsuperscript{\dag } }

\address{Avy Soffer, Mathematics Department, Rutgers University, New Brunswick, NJ, 08903, USA}
\email{soffer@math.rutgers.edu}
\address{Zhao Wu, School of Mathematics and Statistics, Central China Normal University, Wuhan, 430079, P.R. China}
\email{wuzhao218@yahoo.com}
\address{Xiaohua Yao, Department of Mathematics and  Hubei Province Key Laboratory of Mathematical Physics, Central China Normal University, Wuhan, 430079, P.R. China}
\email{yaoxiaohua@mail.ccnu.edu.cn}

\thanks{\textsuperscript{\dag} Corresponding author}

\date{\today}

\begin{document}

\begin{abstract}\baselineskip=12pt
		This paper is devoted to study the time decay estimates for bi-Schr\"odinger operators  $H=\Delta^{2}+V(x)$ in dimension one with decaying potentials $V(x)$. We first deduce the asymptotic expansions of resolvent of $H$ at zero energy threshold without/with the presence of resonances, and then characterize these resonance spaces corresponding to different types of zero resonance in suitable weighted spaces $L_s^2({\mathbf{R}})$. Next we use them to establish the sharp $L^1-L^\infty$ decay estimates of Schr\"odinger groups $e^{-itH}$ generated by bi-Schr\"odinger operators also with zero resonances.  As a consequence, Strichartz estimates are obtained for the solution of fourth-order Schr\"odinger equations with potentials for initial data in $L^2({\mathbf{R}})$. In particular, it should be emphasized that the presence of zero resonances does not change the optimal time decay rate of  $e^{-itH}$ in dimension one, except at requiring faster decay rate of the potential.
\end{abstract}
		\maketitle

	\baselineskip=14pt
\section{Introduction}
\subsection{Backgrounds and problems}

The higher-order elliptic operators $P(D)+V$ have been extensively considered as general Hamiltonian operator by many people in different contexts. We just mention some references, not all. For instance, see Schechter \cite{Schechter} for spectral theory, Kuroda \cite{Kur}, Agmon \cite {Agmon}, H\"ormander \cite{H2} for scattering theory, Davies \cite{Davies}, Davies and Hinz \cite{DaHi}, Deng et al \cite{DDY} for the Gaussian estimates of semigroup,  and  as well \cite{Herbst-Skibsted-Adv-2015, Herbst-Skibsted-Adv-2017, BS,  Mourre, SYY} for other interesting topics.

In this paper, we will consider the time decay estimates for the bi-Schr\"odinger operator:
\begin{equation}\label{eq-schrodinger operator}
H=\Delta^{2}+V(x),\,\, H_{0}=\Delta^{2},
\end{equation} where $V(x)$ is a real-valued function satisfying $|V(x)|\lesssim (1+|x|)^{-\beta}$ for some $\beta>0$.  Such investigations  are motivated in large part by classical Schr\"odinger operator $-\Delta+V$ and their fundamental applications.   In the last thirty years, the dispersive estimates of Schr\"odinger operator have been one of the key topics, which were applied broadly to  nonlinear Schr\"odinger equations, see e.g. \cite{Bou,  Cazenave, TT2, JSS, KeelTao,NS11, Schlag, Schlag21, Simon-Review-1, Simon-Review-2} and  references therein.

For the free propagators $e^{-itH_0}$ and $H_{0}=\Delta^{2}$, Ben-Artzi, Koch and Saut in \cite{BKS} have proved the following pointwise estimates for the kernel of $e^{-itH_0}$ on $\mathbf{R}^n$:
\begin{equation}\label{eq-Ben-Artzi-1}
|\partial_x^{\alpha}I_0(x,t)|\lesssim t^{-\frac{n+|\alpha|}{4}}\Big(1+\frac{|x|}{t^{1/4}}\Big)^{-\frac{(n-|\alpha|)}{3}},\ \ x\in\mathbf{R}^n,
\end{equation}
where $I_0(x,t)$ be the kernel of $e^{-itH_0}$ and $\alpha\in\mathbf{N}^n, $ which immediately implies the following $L^1- L^\infty$ estimates:
\begin{equation}\label{eq-Ben-Artzi-2}
\big\|\partial^\alpha e^{-itH_0}\big\|_{L^1(\mathbf{R}^n)\rightarrow L^\infty(\mathbf{R}^n)} \lesssim |t|^{-\frac{n+|\alpha|}{4}}, \ \ |\alpha|\le n.
\end{equation}
Moreover, we notice that the regular term $\partial^\alpha$ of the inequality \eqref{eq-Ben-Artzi-2} can be replaced by fractional power  $H_0^{\frac{s}{4}}$ for any $0\le s\le n$ ( see e.g. \cite{CMY} ).

For  the fourth order Schr\"odinger  operators $H=\Delta^{2}+V(x)$,  such time decay estimates were much less understood than Schr\"odinger operators. Recently, Feng, Soffer and Yao \cite{FSY} deduced the asymptotic expansion of the resolvent $R_V(z)=(H-z)^{-1}$ around zero threshold,  assuming that zero is a regular point of $H$ for $n\ge 5$ and $n=3$, and showed that in the regular case, the Kato-Jensen time decay estimate of $e^{-itH}$ is $(1+|t|)^{-n/4}$ for $n\geq5$ and the $L^1-L^\infty$ decay estimates is $O(|t|^{-1/2})$ for $n=3$ (not optimal).

 Based on the higher energy estimates of resolvent  in \cite{FSY},  Erdog$\breve{a}$n, Green and Toprak in \cite{Erdogan-Green-Toprak} for $n=3$ and Green-Toprak in \cite{Green-Toprak} for $n=4$ further derived the asymptotic expansion of $R_V(z)$ near zero with the presence of resonance or eigenvalue, and obtained $L^1-L^\infty$ estimates for each kind of zero resonances. They proved the $L^1- L^\infty$ time decay estimate is $|t|^{-n/4}$ for $n=3,4$ if zero is a regular point, and zero energy resonance will change the time decay rate.

Moreover, we mention that  Kato-Jensen estimates and  asymptotic expansion of resolvent have been established by \cite{FSWY} for poly-Schr\"odinger operators $(-\Delta)^m+V$ with $n\ge 2m\ge 4$.  In  previous works,   Murata in \cite{MM}  have generalized Kato and Jensen's  work \cite{JK} to a certain class of $P(D)+V$ assume that $P(D)$ satisfies
	\begin{equation}\label{nondegenerate}
	\big(\nabla P\big)(\xi_0)=0, \ \ \ \  \det\big[\partial_{i}\partial_{j}P(\xi)\big]\Big|_{\xi_0}\neq0.
	\end{equation}	
However, the poly-harmonic operators $H_0=(-\Delta)^m$ do not satisfy the  nondegenerate condition \eqref{nondegenerate} at zero but the case $m=1$. Hence the results in \cite{FSWY} were not covered by \cite{MM}. In particular,
the degeneracy of $(-\Delta)^m$ at zero energy certainly leads to more complicated zero resonances classifications and asymptotic expansions of the resolvent for $(-\Delta)^m+V$ as $m\ge2$. This is actually  one of the main difficult points for the higher order estimates.

 In this paper, we are devoted to establishing the same $L^1-L^\infty$ estimates of $H=\Delta^2+V$ (with regularity terms)  as the free estimates \eqref{eq-Ben-Artzi-2} in dimension one. We showed that the presence of zero resonance  of $H$ does not affect the time decay rate of $e^{-itH} $  but requires higher decay rate of potential $V$.
 This seems to be distinct phenomena than  other situations ( see e.g.  \cite{FSWY,Erdogan-Green-Toprak,Green-Toprak,Schlag21}).

  As a consequence, Strichartz estimates were obtained for the fourth-order Schrodinger equations with potentials for initial data in $L^2({\mathbf{R}})$.  Furthermore, by using similar methods, Li, Soffer and Yao \cite{LSY} have derived the $L^1-L^\infty$ estimates for $H$ with zero resonances in two dimensional case. In particular, we remark that arguments used here can be applied to the higher-order elliptic operators $(-\Delta)^m+V$.

 For Schr\"odinger operator $-\Delta+V$, the $L^{1}- L^{\infty}$ decay estimates have been widely investigated by many people.  For $n\geq3$, Journ\'e, Soffer and Sogge in \cite{JSS} first  established
 the $L^{1}- L^{\infty}$ dispersive estimate in the regular case. For $n\le 3$, one can see Weder \cite{Weder1},  Rodnianski and Schlag \cite{RodSchl}, Schlag and Goldberg \cite{Goldberg-Schlag}, Schlag \cite{Schlag-CMP} and so on.
 Yajima in \cite{Yajima-JMSJ-95} has also established the $L^{1}- L^{\infty}$  estimates for $-\Delta+V$ by wave operator method.
For many further studies, one can refer to \cite{Weder, ES1, ES2,  Goldberg-Green-1, Goldberg-Green-2, Schlag, Schlag21} and therein references.
\subsection{Some notations}
In this subsection, we collect some notations used throughout the paper.

\begin{itemize}
  \item For $a\in\mathbf{R}$, we write $a\pm$ to denote $a\pm\epsilon$ for any small $\epsilon>0$. $[a]$ denotes the largest integer less than or equal to $a$.
  For $a,b\in\mathbf{R^+},$ we denote $A\lesssim B$ (or $B\gtrsim A$) which means that there exists some constant $c>0$ such that $A\le cB$ (or $B\ge  cA$). Moreover, if $A\lesssim B$ and $B\lesssim A,$ then we write $A\sim B.$
  \item For $s, s'\in\mathbf{R}$, $B(s, s')$ denote the space of the bounded operators from $L^{2}_{s}(\mathbf{R})$ to $L^{2}_{s'}(\mathbf{R})$,  where $L^{2}_{s}(\mathbf{R})$ is the weighted $L^2$ space:
	$$L^{2}_{s}(\mathbf{R})=\Big\{f\in L^2_{loc}(\mathbf{R}): (1+|\cdot|)^{s}f\in L^{2}(\mathbf{R})\Big\}.$$
\end{itemize}

\subsection{Main Results}
In the subsection, we  will state our main results. Recall that,  if there exists a nonzero solution $\phi(x)\in L^2(\mathbf R)$ satisfying the equation $H\phi(x)=0$ in distributional sense, then we say that zero is an eigenvalue of $H=\Delta^{2}+V(x).$  In general, there may exist a nonzero resonant solution $\phi(x)$ in $L^2_{-s}(\mathbf{R})$ with some $s>0$ such that $H\phi(x)=0$,  which usually results in the obstructions to dispersive estimates at low energy. In the present case, we first introduce the  zero resonance definitions of $\Delta^{2}+V(x)$ in dimension one.
For $\sigma\in\mathbf{R}$,  let $W_{\sigma}(\mathbf{R})$ denotes the intersection  space
	$ W_{\sigma}(\mathbf{R})= \bigcap\limits_{s>\sigma} L^{2}_{-s}(\mathbf{R}).$
  Note that, $W_{\sigma_1}(\mathbf{R})\supset W_{\sigma_2}(\mathbf{R})$ if $\sigma_1>\sigma_2$ and  $W_{0}(\mathbf{R})\supset L^{2}(\mathbf{R})$.
\begin{definition}\label{defi-resonance-1}
Let $H=\Delta^2+V$ on $L^{2}(\mathbf{R})$ and $|V(x)|\lesssim(1+|x|)^{-\beta}$ for some $\beta>0$. Then we say that

(i)~ zero is  a first kind resonance  of $H$
if there exists some nonzero $\phi(x)\in W_{\frac{3}{2}}(\mathbf{R})$ but no any nonzero $ \phi\in W_{\frac{1}{2}}(\mathbf{R}) $ such that $H\phi(x)=0$ in the distributional sense;

(ii)~zero is  a second kind resonance  of $H$ if there exists some nonzero  $\phi(x)\in W_{\frac{1}{2}}(\mathbf{R})$ but no any nonzero $ \phi\in  L^2(\mathbf{R}) $ such that $H\phi(x)=0$ in the distributional sense;

(iii)~zero is  an eigenvalue  of $H$ if there exists some nonzero $\phi(x)\in L^2(\mathbf{R}) $ such that $H\phi(x)=0$ in the distributional sense;

(iv)~ zero is a regular point of $H$ if zero is neither a resonance nor an eigenvalue of $H$.
\end{definition}

\begin{remark}\label{remark-eigenvalue} Some examples of zero resonance and zero eigenvalue:

(i)~(\textbf{Zero resonance}) Assume that $\phi\in C^\infty(\mathbf{R})$ is a positive function which is equal to $c|x|+d$ for $|x|>1.$ Then $H\phi(x)=0$ when $\displaystyle V(x)=-(\Delta^2\phi)/\phi(x).$ One can easily check that $V(x)\in C_0^\infty(\mathbf{R})$ and $\phi(x)\in W_{\frac{3}{2}}(\mathbf{R}).$  Hence by Definition \ref{defi-resonance-1}, we know that the zero is a resonance point of $H$ even with a compactly supported smooth potential.

(ii)~(\textbf{Zero eigenvalue}) Unlike zero resonance, the zero eigenvalue problem of $H$ is subtler.
For instance, let ~$\phi(x)=(1+x^2)^{-s}$ with $s>\frac{1}{4} $ and  $\displaystyle V(x)=-(\Delta^2\phi)/\phi(x).$  Then one has that $$V(x)=O((1+|x|)^{-4}), \ \ |x|\rightarrow \infty,$$ and  $\phi\in L^2(\mathbf{R})$ satisfying $H \phi=0. $  This means that zero is an eigenvalue of $H$ with such slowly decaying potential.

On the other hand, it was pointed out to us by Prof. Goldberg \cite{Goldberg} that if a potential $V(x)$ has faster decay,
then zero eigenvalue of the one-dimensional bi-Schr\"odinger operator may not occur, see e.g.\cite{DT} for Schr\"{o}dinger operator case. In fact, if the potential $V$ satisfies the decay condition of Theorem \ref{thm-main results} below,  then a proof on the absence of zero eigenvalue has been stated in Lemma \ref{lemma-S_3} of Section \ref{section-M}.

\end{remark}

Now we state the first main result under the decay conditions depending on the types of zero energy defined above. In this paper due to the faster decay rate of $V$ and the absence of zero eigenvalue, we do not necessarily discuss zero eigenvalue case in Theorem \ref{thm-main results} below. Even so, it should be noticed that compared with the Schr\"{o}dinger operator $-\Delta+V$ on $L^2(\mathbf{R})$ ( see e.g. Jensen and Nenciu \cite{JN, JN2} ),
the classification of the zero resonance for $H=\Delta^2+V$ in one dimensional case is more complicated due to the low dimensionality ( i.e. $n=1$ ) and  the degeneracy of $\Delta^2$ at zero energy ( i.e. zero is the degenerate critical point of the symbol $ \xi ^4$ ); it leads to the more singular terms in the zero asymptotic expansion of resolvent $R_V^\pm(\mu)$ ( see Theorem \ref{thm-M-inverse} below or Section \ref{section-M} ).

\begin{theorem}\footnote{After the submission of our paper, we learned that by a different method, Thomas Hill has proved that the estimates \eqref{eq-main result} and \eqref{eq-regularity estimate} hold, assuming that zero is a regular point for  $H=\Delta^2+V$ with a compactly supported potential $V(x)\in C_c^1(\mathbf{R}),$ see his dissertation \cite{Hill}.}\label{thm-main results}
Let $|V(x)|\lesssim(1+|x|)^{-\beta}$ with some $\beta>0$ depending on the following types:
\begin{equation}\label{eq-beta}
\beta>\begin{cases}
13, & \text{zero is the regular point},\\
17, & \text{zero is the first kind resonance},\\
25, & \text{zero is the second kind resonance}.
\end{cases}
\end{equation}
Assume that $H=\Delta^2+V$ has no positive embedded eigenvalues and $P_{ac}(H)$ denotes the projection onto the absolutely continuous spectrum space of $H$.
Then we have the following decay estimate:
\begin{equation}\label{eq-main result}
\big\|e^{-itH}P_{ac}(H)\big\|_{L^1(\mathbf{R})\rightarrow L^\infty(\mathbf{R})}\lesssim |t|^{-\frac{1}{4}}, \ \ t\neq0.
\end{equation}
Furthermore, for any $0\le \alpha<1,$  the following improved estimate holds:
\begin{equation}\label{eq-regularity estimate}
\big\|H^{\frac{\alpha}{4}} e^{-itH}P_{ac}(H)\big\|_{L^1(\mathbf{R})\rightarrow L^\infty(\mathbf{R})}\lesssim |t|^{-\frac{1+\alpha}{4}}, \ \ t\neq0.
\end{equation}
\end{theorem}

As we will see later, the  needed decay rates of the potential $V$ mainly depend on the lower energy expansions of resolvent $R_V(z)$ for different zero resonance types (see Theorem \ref{thm-M-inverse} below). According to Remark \ref{remark-M inverse-expansions} after Theorem \ref{thm-M-inverse}, we still can obtain the decay estimates \eqref{eq-main result} and \eqref{eq-regularity estimate} by using the weak expansions of $R_V(z)$ under the assumption of the following less decay rates:
\[\beta>\begin{cases}
7, & \text{zero is the regular point},\\
9, & \text{zero is the first kind resonance},\\
13, & \text{zero is the second kind resonance}.
\end{cases}\]
In the paper, we do not follow the way along with these interesting improvements. Instead, we focus on the simplicity of the proof based on the expansions of resolvent $R_V(z)$ in Theorem \ref{thm-M-inverse}.  Noting that the estimates \eqref{eq-main result} and \eqref{eq-regularity estimate} above are optimal in view of the free estimate \eqref{eq-Ben-Artzi-2} with $V=0$.

In particular, the optimal time decay rate ( e.g. $|t|^{-\frac{1}{4}}$ in the \eqref{eq-main result}) does not change as the zero resonance arises, but only requires the faster decay rate $\beta$ of potential $V$ as the compensations of zero singularities. Moreover, note that for Schr\"{o}dinger operator $-\Delta+V$, Egorova, Kopylova, Marchenko and Teschl in \cite{EKMT} have proved that the resonant case does not need additional decay of the potential, hence it is an interesting question whether we can show that the sharp decay estimates \eqref{eq-main result} still hold on the same decay condition on $V(x)$ for $H=\Delta^2+V$.

Moreover, in the above Theorem \ref{thm-main results}, we also assume that  $H=\Delta^2+V$ has no  positive eigenvalues embedded in the absolutely continuous spectrum, which has been the indispensable condition in dispersive estiamtes.
For Schr\"odinger operator $-\Delta+V$, Kato in \cite{K} has shown the absence of positive eigenvalues for $H=-\Delta+V$ with the decay potentials $V=o(|x|^{-1})$ as $|x|\rightarrow\infty$. For the bi-Schr\"odinger operator $H=\Delta^2+V$, the situation seems to be more complicated than the second order operator, since there exist examples even with compactly supported smooth potentials such that positive eigenvalues appear ( see e.g. \cite{FSWY} ).

  However, interestingly,
 a simple criterion has been proved in \cite{FSWY} that $H=\Delta^2+V$ has no any positive eigenvalues under the assumption that the potential $V$ is bounded and satisfies the repulsive condition (i.e. $(x\cdot\nabla)V\le 0$). In particular, such criterion also works for the general higher order elliptic operator $P(D)+V$.
 We also notice that for a general selfadjoint operator $\mathcal{H}$ on $L^{2}(\mathbf{R}^n)$, even if $\mathcal{H}$ has a simple embedded eigenvalue, Costin and Soffer in \cite{CoSo} have proved that $\mathcal{H}+\epsilon W$ can kick off the eigenvalue located in a small interval under certain small perturbation of potential.

%

As an interesting consequence  of $L^1-L^\infty$ estimates above,  by Keel-Tao \cite{KeelTao} method, we immediately obtain Strichartz estimate to the fourth-order Schr\"{o}dinger equation on $\mathbf{R}$:
\begin{equation}\label{eq-nonlinear equation}
\left\{ \begin{gathered}
   i\partial_t \phi = \big(\Delta^2+V\big)\phi+F(t,x),\ \ (t,x)\in\mathbf{R}^{1+1}, \hfill \\
   \phi(0,\cdot)=\phi_0(x).\hfill \\
\end{gathered}  \right.\ \ \
\end{equation}
Recall that, a pair $(q,r)$ is said to be $\sigma-$admissible if
\begin{equation}\label{admissblepair}
  \frac{1}{q}+\frac{\sigma}{r}=\frac{\sigma}{2}, \ \  2\leq q\leq\infty,\ \ (q,r,\sigma)\neq (2,\infty,1).
\end{equation}

\begin{corollary}\label{cor-Strichartz}
Let $V$ and $H=\Delta^2+V$ satisfy the same conditions as Theorem \ref{thm-main results} above.
 Then for any $\displaystyle\frac{1}{4}-$admissible pairs $(q,r)$ and $(\tilde{q},\tilde{r}),$  we have the following estimates:\\

(i)~Homogeneous Strichartz Estimate
\begin{equation}\label{eq-homogeneous strichartz}
\big\|e^{-itH}P_{ac}(H)\phi_0\big\|_{L_t^qL_x^r(\mathbf{R}\times\mathbf{R})}\lesssim \|\phi_0\|_{L^2(\mathbf{R})}.
\end{equation}

(ii)~Dual Homogeneous Strichartz Estimate
\begin{equation}\label{eq-dual homogeneous strichartz}
\Big\|\int_\mathbf{R}e^{-is H}P_{ac}(H)F(s,\cdot)ds\Big\|_{L_x^2(\mathbf{R})}\lesssim \|F\|_{L_t^{\tilde{q}'}L_x^{\tilde{r}'}(\mathbf{R}\times\mathbf{R})}.
\end{equation}

(iii)~The solution $\phi(t,x)$ of the equation \eqref{eq-nonlinear equation} satisfies
\begin{equation}\label{eq-solution strichartz}
\big\|P_{ac}(H)\phi\big\|_{L_t^qL_x^r(\mathbf{R}\times\mathbf{R})}\lesssim \|\phi_0\|_{L^2(\mathbf{R})}+\|F\|_{L_t^{\tilde{q}'}L_x^{\tilde{r}'}(\mathbf{R}\times\mathbf{R})}.
\end{equation}
\end{corollary}

In order to derive the decay estimates \eqref{eq-main result}, it follows from Stone's formula  that
\begin{equation}\label{eq-Stone formula}
e^{-itH}P_{ac}(H)f=\frac{2}{\pi i}\int_0^\infty e^{-it\mu^4}\mu^3[R_V^+(\mu^4)-R_V^-(\mu^4)]f(x)d\mu.
\end{equation}
 Hence we need to study the detailed zero expansions of the resolvent $R_V(z)=(H-z)^{-1}$ by the perturbations of the free resolvent $R_0(z)$, for which the following identity holds:
\begin{equation}\label{eq-resolvent identity}
R_0(z):=\big(\Delta^2-z\big)^{-1}=\frac{1}{2z^{\frac{1}{2}}}\Big(R(-\Delta;z^{\frac{1}{2}})-R(-\Delta;-z^{\frac{1}{2}})\Big),\ \ \ z\in\mathbf{C}\setminus [0,\infty).
\end{equation}
Here the resolvent $R(-\Delta;z^{\frac{1}{2}}):=(-\Delta-z^{\frac{1}{2}})^{-1}.$
For $\lambda\in\mathbf{R}^{+},$ we define the limiting  resolvent operators by
\begin{equation}\label{eq-free resolvent-LAP}
R_0^\pm(\lambda):=R_0^\pm(\lambda\pm i0)=\lim_{\epsilon\rightarrow 0^+}(\Delta^2-(\lambda\pm i\epsilon))^{-1},
\end{equation}
\begin{equation}\label{eq-perturbation resolvent-LAP-1}
R_V^\pm(\lambda):=R_V^\pm(\lambda\pm i0)=\lim_{\epsilon\rightarrow 0^+}(H-(\lambda\pm i\epsilon))^{-1},
\end{equation}
By using the representation \eqref{eq-resolvent identity} for $R_0(z)$ with $z=\omega^4$ for $\omega$ in the first quadrant of the complex plane, and taking limits as $\omega\rightarrow\mu$ and $\omega\rightarrow i\mu$, we obtain
\begin{equation}\label{eq-free resolvent-LAP-2}
R_0^\pm(\mu^4)=\frac{1}{2\mu^2}\Big(R^\pm(-\Delta;\mu^2)-R(-\Delta;-\mu^2)\Big),\ \ \mu>0.
\end{equation}  	
It is well-known that by the limiting absorption principle ( see e.g. \cite{Agmon} ), $R^\pm(-\Delta;\mu^2)$ are well-defined as the bounded operators of $B(L_s^2, L^2_{-s})$ for any $s>1/2$, and therefore $R_0^\pm(\mu^4)$ are also well-defined between these weighted spaces. This limiting absorption property can be extended to $R_V^\pm(\mu^4)$ for $\mu>0$ for certain decaying and  bounded potentials ( see e.g. \cite{FSY} ).

The main challenge is to derive the asymptotic expansions of $R_V^\pm(\mu^4)$ as $\mu$ near zero. Recall that the Green's kernel of the free Laplace resolvent in dimension one( see \cite{ES1,JN} for example)
\begin{equation}\label{eq-resolvent of Laplacian}
R^\pm(-\Delta;\mu^2)(x,y)=\frac{\pm i}{2\mu}e^{\pm i\mu|x-y|}.
\end{equation}
Therefore, by \eqref{eq-free resolvent-LAP-2}, we have for $\mu>0$,
\begin{equation}\label{eq-free resolvent-LAP-3}
\begin{split}
R_0^\pm(\mu^4)(x,y)=&\frac{1}{4\mu^3}\Big(\pm ie^{\pm i\mu|x-y|}-e^{-\mu|x-y|}\Big).
\end{split}
\end{equation}
Denote $M^\pm(\mu)=U+vR_0^\pm(\mu^4)v,$ $v(x)=|V(x)|^{1/2}$ and $U=\text{sgn}\big(V(x)\big),$
by the following symmetric identity:
\begin{equation}\label{eq-symmetric identity}
R_V^\pm(\mu^4)=R_0^\pm(\mu^4)-R_0^\pm(\mu^4)v\Big(M^\pm(\mu)\Big)^{-1}vR_0^\pm(\mu^4).
\end{equation}
It suffices to establish the asymptotic expansions of $\Big(M^\pm(\mu)\Big)^{-1}$ with the presence of zero resonance. For this end, let's introduce some spaces  to derive the asymptotic expansions of $\Big(M^\pm(\mu)\Big)^{-1}$.

Let $T_0:= U+vG_0v$ with $G_0=\big(\Delta^2\big)^{-1}$ and  $P=\|V\|_{L^1(\mathbf{R})}^{-1}v\langle v,\cdot \rangle$ be the orthogonal projection onto the span of $v,$ i.e. $PL^2(\mathbf{R})=\text{span}\{v\}.$
\begin{definition}\label{defi-projection space}
Let $Q=I-P.$ Then  we define that

(i)~$S_0L^2(\mathbf{R})=\{f\in L^2~|~\langle  f,x^kv\rangle=0,\ k=0,1 \}.$

(ii)~$S_1L^2(\mathbf{R})=\{f\in L^2~|~\langle f,x^kv\rangle=0, \ k=0,1,\ \text{and}\ S_0T_0f=0 \}.$

(iii)~$S_2L^2(\mathbf{R})=\{f\in L^2~|~\langle f,x^kv\rangle=0, \ k=0,1,2,\ \text{and}\ QT_0f=0 \}.$

(iv)~$S_3L^2(\mathbf{R})=\{f\in L^2~|~\langle f,x^kv\rangle=0, \ k=0,1,2,3,\  \text{and}\ T_0f=0 \}.$
 \end{definition}

\begin{remark}\label{remark-equivalence}

(i) The projection spaces clearly satisfy the following equivalent relationships:
\begin{align*}
 f\in S_0L^2(\mathbf{R}) &\Longleftrightarrow f \in \text{span}\{v,xv\}^{\perp};\\
 f\in S_1L^2(\mathbf{R}) &\Longleftrightarrow f \in \text{span}\{v,xv\}^{\perp}\ \text{and}\ T_0f\in\text{span}\{v,xv\};\\
 f\in S_2L^2(\mathbf{R}) &\Longleftrightarrow f \in \text{span}\{v,xv,x^2v\}^{\perp}\ \text{and}\ T_0f\in\text{span}\{v\};\\
 f\in S_3L^2(\mathbf{R}) &\Longleftrightarrow f \in \text{span}\{v,xv,x^2v,x^3v\}^{\perp}\ \text{and}\ T_0f=0.
\end{align*}
Hence, $S_0L^2(\mathbf{R})\supseteq S_1L^2(\mathbf{R})\supseteq S_2L^2(\mathbf{R})\supseteq S_3L^2(\mathbf{R}).$

(ii) Denote by $S_j\ (j=0,1,2,3)$ the orthogonal projections onto the related to projection spaces $S_jL^2(\mathbf{R})$,  then their orthogonal relationships can be listed in the following {\bf{Table 1}}.
\begin{table}[h]\label{table-1}
\renewcommand\arraystretch{2}
\centering
\begin{tabular}{|p{4.5cm}<{\centering}|p{10cm}<{\centering}|}
  \specialrule{0.03em}{0pt}{0pt}
  {\bf Projections on $L^2(\mathbf{R})$}      & {\bf Orthogonal  Relationships}   \\
  \specialrule{0.03em}{0pt}{0pt}
  \makecell{$S_0$}      &  \makecell{$S_0(v(x))=S_0(xv(x))=0$} \\
  \specialrule{0.03em}{0pt}{0pt}
 \makecell{$S_1$}     &  \makecell{$S_1T_0S_0=S_0T_0S_1=0$ \\ $S_1(v(x))=S_1(xv(x))=0$   }\\
  \specialrule{0.03em}{0pt}{0pt}
  \makecell{$S_2$}      &  \makecell{$S_2T_0Q=QT_0S_2=0$ \\ $S_2(v(x))=S_2(xv(x))=S_2(x^2v(x))=0$    } \\
  \specialrule{0.03em}{0pt}{0pt}
 \makecell{$S_3$}     &  \makecell{$S_3T_0Q=QT_0S_3=0$\\
 $S_3T_0P=PT_0S_3=0$ \\ $S_3(v(x))=S_3(xv(x))=0$ \\ $S_3(x^2v(x))=S_3(x^3v(x))=0$}\\
  \specialrule{0.03em}{0pt}{0pt}
\end{tabular}
\vspace{0.2cm}
\caption{Orthogonal Relationships of Projections}
\end{table}

(iii) From Lemma \ref{lemma-S_3} later, we will further prove that $S_3L^2(\mathbf{R})=\{0\}$, which is important in the proof of Theorem \ref{thm-M-inverse} (iii), see Section \ref{section-M} below.
\end{remark}

These projection spaces in Definition \ref{defi-projection space} will  be used to characterize the zero energy resonance types of $H$ ( see  Theorem \ref{thm-resonance space} for the proof).
\begin{theorem}\label{thm-equivalence}
Assume that $H=\Delta^2+V$ with $|V(x)|\lesssim(1+|x|)^{-\beta}$ with some $\beta$ satisfying the condition \eqref{eq-beta}. Then the following statements hold:

(i)~zero is a regular point of $H$ if and only if $S_1L^2(\mathbf{R})=\{0\};$

(ii)~zero is a first kind resonance of $H$ if and only if $S_1L^2(\mathbf{R})\neq\{0\}$ and $S_2L^2(\mathbf{R})=\{0\};$

(iii)~zero is a second kind resonance of $H$ if and only if $S_2L^2(\mathbf{R})\neq\{0\}.$

\end{theorem}

By the characterizations of zero resonances in Theorem \ref{thm-equivalence} and using the orthogonal relationships in {\bf{Table 1}},  we can derive the asymptotic expansions of $\Big(M^{\pm}(\mu)\Big)^{-1}$, which are fundamental to study dispersive estimates and other problems related to $H$.

\begin{theorem}\label{thm-M-inverse}Assume that $|V(x)|\lesssim (1+|x|)^{-\beta}$ with some $\beta>0$.  Then for  $0<|\mu|\ll1$, we have the following expansions of $\Big(M^\pm(\mu)\Big)^{-1}$ in $L^2(\mathbf{R})$:

(i)~If zero is a regular point of  $H$ and $\beta>13,$ then
\begin{equation}\label{eq-M-inverse-Regular Case}
\begin{split}
\Big(M^\pm(\mu)\Big)^{-1}=&S_0A_{01}^0S_0+\mu QA_{11}^0Q+\mu^2(QA_{21}^0Q+S_0A_{22}^0+A_{23}^0S_0)+\Gamma_3^0(\mu),
\end{split}
\end{equation}
where $A_{kj}^0$ are $\mu$-independent bounded operators in $L^2(\mathbf{R})$  and  $\Gamma_3^0(\mu)$ is $\mu$-dependent operator which satisfies
\[\big\|\Gamma_3^0(\mu)\big\|_{L^2\rightarrow L^2}+\mu\big\|\partial_\mu\big(\Gamma_3^0(\mu)\big)\big\|_{L^2\rightarrow L^2}\lesssim \mu^3.\]

(ii)~If zero is a first kind resonance of $H$ and $\beta>17,$ then
\begin{equation}\label{eq-M-inverse-First Kind}
\begin{split}
\Big(M^\pm(\mu)\Big)^{-1}=&\frac{S_1A_{-1,1}^1S_1}{\mu}+S_0A_{01}^1Q+QA_{02}^1S_0+\mu \big(QA_{11}^1Q+S_0A_{12}^1+A_{13}^1S_0\big)\\
&+\mu^2(QA_{21}^1+A_{22}^1Q)+\Gamma_3^1(\mu),
\end{split}
\end{equation}
where $A_{kj}^1$ are $\mu$-independent bounded operators in $L^2(\mathbf{R})$  and  $\Gamma_3^1(\mu)$ is $\mu$-dependent operator which satisfies
\[\big\|\Gamma_3^1(\mu)\big\|_{L^2\rightarrow L^2}+\mu\big\|\partial_\mu\big(\Gamma_3^1(\mu)\big)\big\|_{L^2\rightarrow L^2}\lesssim \mu^3.\]

(iii)~If zero is a second kind  resonance of $H$ and  $\beta>25,$ then
\begin{equation}\label{eq-M-inverse-Second Kind}
\begin{split}
\Big(M^\pm(\mu)\Big)^{-1}=&\frac{S_2A_{-3,1}^2S_2}{\mu^3}+\frac{S_2A_{-2,1}^2S_0+S_0A_{-2,2}^2S_2}{\mu^{2}}+\frac{S_0A_{-1,1}^2S_0+QA_{-1,2}^2S_2+S_2A_{-1,3}^2Q}{\mu}\\
&\ \ +\big(QA_{01}^2S_0+S_0A_{02}^2Q+S_2A_{03}^2+A_{04}^2S_2\big)+\mu\big(QA_{11}^2Q+S_0A_{12}^2+A_{13}^2S_0\big)\\
&\ \ +\mu^2(QA_{21}^2+A_{21}^2Q)+\Gamma_3^2(\mu),
\end{split}
\end{equation}
where $A_{kj}^2$  are $\mu$-independent bounded operators in $L^2(\mathbf{R})$  and  $\Gamma_3^2(\mu)$ is  $\mu$-dependent operator which satisfies
\[\big\|\Gamma_3^2(\mu)\big\|_{L^2\rightarrow L^2}+\mu\big\|\partial_\mu\big(\Gamma_3^2(\mu)\big)\big\|_{L^2\rightarrow L^2}\lesssim \mu^3.\]
\end{theorem}
\begin{remark}\label{remark-M inverse-expansion}
\vskip0.3cm
If we assume the more decay rate of $V(x)$, then we will have more information about the error term $\Gamma_3^i(\mu)$ ( here the  index $i$ respects to the zero energy type). More specifically, one actually can obtain that
\[\Gamma_3^i(\mu)=\frac{\mu^3}{\tilde{a}^\pm}P+\mu^3(QA_{31}^i+A_{32}^iQ)+\Gamma_4^i(\mu),\]
where $A_{31}^i$ are $\mu$-independent bounded operators in $L^2(\mathbf{R})$  and  $\Gamma_4^i(\mu)$ are $\mu$-dependent operators which satisfy
\[\big\|\Gamma_4^i(\mu)\big\|_{L^2\rightarrow L^2}+\mu\big\|\partial_\mu\big(\Gamma_4^i(\mu)\big)\big\|_{L^2\rightarrow L^2}\lesssim \mu^4.\]
\end{remark}
\begin{remark}\label{remark-M inverse-expansions}
In addtion, assuming the less decay rate of $V(x)$, we can also obtain some weak expansions of $\big(M^\pm(\mu)\big)^{-1}$ by checking the proof process of Theorem \ref{thm-M-inverse} (see Section \ref{section-M}) along with Remark \ref{remark-M-expansion}. Specifically, in the \textbf{Regular Case}, it is enough for $\beta>7$ to obtain the following expansions:
\[\Big(M^\pm(\mu)\Big)^{-1}=S_0\Gamma_{01}^0(\mu)S_0+ Q\Gamma_{11}^0(\mu)Q+(Q\Gamma_{21}^0(\mu)Q+S_0\Gamma_{22}^0(\mu)+\Gamma_{23}^0(\mu)S_0)+\Gamma_{31}^0(\mu),\]
where $\Gamma_{kj}^0(\mu) (k=0,1,2,3)$ are $\mu$-dependent operators which satisfy
\[\big\|\Gamma_{kj}^0(\mu)\big\|_{L^2\rightarrow L^2}+\mu\big\|\partial_\mu\big(\Gamma_{kj}^0(\mu)\big)\big\|_{L^2\rightarrow L^2}\lesssim \mu^k.\]
Furthermore, similar weak expansions of $\big(M^\pm(\mu)\big)^{-1}$ can be established for the different resonance cases if $V$ satisfies the following decay rate:
\[\beta>\begin{cases}
9, & \text{zero is the first kind resonance},\\
13, & \text{zero is the second kind resonance}.
\end{cases}\]
According to the proof of Theorem \ref{thm-main results}, we remark that even if we use the above weak form expansion of $\big(M^\pm(\mu)\big)^{-1},$  the $L^1-L^\infty$ decay estimates \eqref{eq-main result} of $e^{-itH} $ also hold. But the expansions of  Theorem \ref{thm-M-inverse} will be helpful to other problems  such as the $L^p$ boundedness of the wave operators.
\end{remark}

From Theorem \ref{thm-M-inverse} above,  it is obvious that all types of zero energy resonances lead to the higher singularity than the regular case in the expansions of $\Big(M^\pm(\mu)\Big)^{-1}$. On the other hand, the zero energy resonances also offer more orthogonal relationships, which in turn provide the more cancellations to balance the singularities at the cost of the decay rate of potentials, as shown in the proof of Theorem \ref{thm-main results} below.

  We also remark that our analysis for the proof of Theorem \ref{thm-M-inverse} is longer and more complicated processes in some ways. First, some singular terms in the expansions of the free resolvent $R_0^\pm(\mu^4)$ will bring more complicated interacting terms in the expansions of $\Big(M^\pm(\mu)\Big)^{-1}$, see Section \ref{section-M}.
   Then these interacting terms will lead to the difficulties of the inversion processes, especially for the second kind resonance case. For instance, in the case of the second kind resonance , we need to analyze the kernel of the non-positive operator $T_2$ given by
\begin{align*}
T_2&=S_2vG_2vS_2-\frac{8\cdot(6!)}{\|V\|_{L^1}}S_2T_0^2S_2+\frac{5}{2}S_2vG_1vD_0vG_1vS_2\\
&+\frac{48\cdot(6!)}{\|V\|_{L^1}^2}\Big(S_2T_0vG_{-1}v-\frac{\|V\|_{L^1}}{6}S_2vG_1vD_0T_0\Big)D_2\Big(v G_{-1}v T_0S_2-\frac{\|V\|_{L^1}}{6}T_0D_0vG_1vS_2\Big),
\end{align*}
  which is actually a sum of the four complicated operators with changing signs. We will identify the $\ker T_2$ as the projection space $S_3L^2(\mathbf{R})$, see Lemma \ref{lemma-projection space} for more details.

Finally, let us comment about the proof of the decay estimates in Theorem \ref{thm-main results}.

Since the derivative of $\mu $ for the free resolvent kernel $R_0^\pm(\mu^4)(x,y)$ can not be dominated by $\mu^{\alpha}$ in $x,y$ uniformly, in this paper we have used Littlewood-Paley decompositions and oscillatory integrals theory to solve the uniformity problem, which seems to be general interesting technical point.
Using some similar methods but with important technical modifications, Li, Soffer and Yao \cite{LSY} have established the $L^1-L^\infty$ estimates for $e^{-itH}$ in dimension two.  We also think that these arguments used here could be extended to general higher-order elliptic operators $P(D)+V$.

The paper is organized as follows. In Section 2, we give the proof of Theorem \ref{thm-main results}.  In Section 3, we give the process how to obtain the asymptotic expansion of $\Big(M^\pm(\mu)\Big)^{-1}$ case by case. In Section 4, we characterize the resonance spaces for each kind of zero resonance.

\section{Proof of Theorem \ref{thm-main results}}
\label{section2}
In this section, we will give the proof of Theorem \ref{thm-main results}. Choose a fixed even function $\varphi\in C_c^\infty(\mathbf{R})$ such that
\begin{equation}\label{eq-varphi}
\varphi(s)=\begin{cases}
1,\ &\ \text{if}\ |s|\le\frac{1}{2},\\
0,\ &\ \text{if}\ |s|\ge 1.
\end{cases}
\end{equation}
For each $N\in\mathbf{Z},$ let $\varphi_N(s)=\varphi(2^{-N}s)-\varphi(2^{-N+1}s),$ then $\text{supp}\ \varphi_N\subset[2^{N-2},2^N]$ and $\varphi_N(s)=\varphi_0(2^{-N}s).$ Furthermore, for any $s\neq 0$ we have
$\displaystyle\sum_{N=-\infty}^{+\infty}\varphi_N(s)=1$.
Using functional calculus and Stone's formula, we have
\begin{align}\label{Stone formula}
e^{-itH}P_{ac}(H)f=&\frac{2}{\pi i}\int_0^\infty e^{-it\mu^4}\mu^3[R_V^+(\mu^4)-R_V^-(\mu^4)]fd\mu\\ \nonumber
=&\frac{2}{\pi i}\int_0^\infty e^{-it\mu^4}\sum_{N=-\infty}^{+\infty}\varphi_N(\mu)\mu^3[R_V^+(\mu^4)-R_V^-(\mu^4)]fd\mu\\ \nonumber
:=&\frac{2}{\pi i}\int_0^\infty e^{-it\mu^4}\chi(\mu)\mu^3[R_V^+(\mu^4)-R_V^-(\mu^4)]fd\mu\\ \nonumber
&\ +\frac{2}{\pi i}\int_0^\infty e^{-it\mu^4}\widetilde{\chi}(\mu)\mu^3[R_V^+(\mu^4)-R_V^-(\mu^4)]fd\mu,
\end{align}
where $\chi(\mu)=\sum\limits_{N=-\infty}^{N_0}\varphi_N(\mu)$ and $\widetilde{\chi}(\mu)=\sum\limits_{N=N_0+1}^{+\infty}\varphi_N(\mu)$ for some $N_0<0$ ( fixed later but depending on specific zero energy types).
Hence we reduce the proof of Theorem \ref{thm-main results} into the estimates of the following two parts: low and high energy parts.
\begin{theorem}\label{thm-low energy}
\textbf{(the low energy dispersive estimate)}\\
Let $|V(x)|\lesssim(1+|x|)^{-\beta}$ with some $\beta>0$ depending on the following zero energy types:
\[\beta>\begin{cases}
13, & \text{zero is the regular point},\\
17, & \text{zero is the first kind resonance},\\
25, & \text{zero is the second kind resonance}.
\end{cases}\] Assume that $H=\Delta^2+V$ has no positive embedded eigenvalue and $P_{ac}(H)$ denotes the projection onto the absolutely continuous spectrum space of $H,$
then  the following estimate holds:
\begin{equation}\label{eq-low result-1}
\|e^{-itH}P_{ac}(H)\chi(H)\|_{L^1(\mathbf{R})\rightarrow L^\infty(\mathbf{R})}\lesssim |t|^{-\frac{1}{4}}.
\end{equation}
Moreover, for any $0\le \alpha<1,$ we have the following estimate:
\begin{equation}\label{eq-regularity estimate-low}
\big\|H^{\frac{\alpha}{4}} e^{-itH}P_{ac}(H)\chi(H)\big\|_{L^1(\mathbf{R})\rightarrow L^\infty(\mathbf{R})}\lesssim |t|^{-\frac{1+\alpha}{4}}.
\end{equation}
\end{theorem}
\begin{theorem}\label{thm-high energy}
\textbf{(the high energy dispersive estimate)}\\
Let $|V(x)|\lesssim(1+|x|)^{-2-}$. Assume that $H=\Delta^2+V$ has no positive embedded eigenvalue and $P_{ac}(H)$ denotes the projection onto the absolutely continuous spectrum space of $H$, then we have
\begin{equation}\label{eq-high result-1}
\|e^{-itH}P_{ac}(H)\widetilde{\chi}(H)\|_{L^1(\mathbf{R})\rightarrow L^\infty(\mathbf{R})}\lesssim |t|^{-\frac{1}{4}}.
\end{equation}
Moreover, for any $0<\alpha<1,$ we have the following estimate:
\begin{equation}\label{eq-regularity estimate-high}
\big\|H^{\frac{\alpha}{4}} e^{-itH}P_{ac}(H)\widetilde{\chi}(H)\big\|_{L^1(\mathbf{R})\rightarrow L^\infty(\mathbf{R})}\lesssim |t|^{-\frac{1+\alpha}{4}}.
\end{equation}
\end{theorem}
The proofs of the low and high energy parts will be given in the following subsections. For the end, let us first prove the decay estimates of the free group $e^{-it\Delta^2}$ in dimension one by the resolvent method, which display some basic technics how to prove the potential case.
\subsection{The estimates for the free case}

First, we give a useful lemma to estimates a kind of oscillatory integral appeared repeatedly in this paper.
\begin{lemma}\label{lemma-low-energy-2}
Let $m\ge 1$, $\Psi(z)$ be a nonnegative real valued function of $\mathbf{R}^m $ and $\Phi(s;z)$ be a function on $\mathbf{R}\times\mathbf{R}^m$, which is smooth for the first variable $s$. Assume that for any $s\in\text{supp}\ \varphi_0(s)$ (here $\varphi_0(s)=\varphi(s)-\varphi(\frac{1}{2}s),$ see \eqref{eq-varphi}) and any $N\in\mathbf{Z},$  $\Phi(s;z)$ satisfies the following uniform bounded condition:
\[\big|\partial_{s}^k \Phi(2^Ns;z)\big|\lesssim 1,\ \ \ k=0,1,\]
then we have
\begin{equation}\label{eq-low energy-2}
\sup_{z\in\mathbf{R}^m}\Big|\int_0^\infty e^{-it2^{4N}s^4}\varphi_0(s)e^{\pm i2^Ns\Psi(z)}\Phi(2^Ns;z)ds\Big|\lesssim (1+|t|\cdot 2^{4N})^{-\frac{1}{2}}.
\end{equation}
\end{lemma}
\begin{proof}
Denote
\begin{align*}
K_N^\pm(t,z):=&\int_0^\infty e^{-it2^{4N}s^4}\varphi_0(s)e^{\pm i2^Ns\Psi(z)}\Phi(2^Ns;z)ds.
\end{align*}
Since for any $s\in\text{supp}\ \varphi_0(s)\subset\big[\frac{1}{4},1],$ $\big|\Phi(2^Ns;z)\big|\lesssim 1$, then we have
\begin{equation}\label{eq-K-2}
\big|K_N^\pm(t,z)\big|\lesssim \int_\frac{1}{4}^1 \big|\varphi_0(s)\big|ds \lesssim 1.
\end{equation}
\textbf{Case~1:} $2^N\Psi(z)\le 1.$  Using integrate by parts once, we have
\begin{equation*}
\begin{split}
K_N^\pm(t,z)=&\frac{e^{-it2^{4N}s^4}\big(s^{-3}\varphi_0(s)e^{\pm i2^Ns\Psi(z)}\Phi(2^Ns;z)\big)}{-4it\cdot2^{4N}}\bigg|_0^\infty\\
&\ \ \ +\frac{1}{4it\cdot2^{4N}}\int_0^\infty e^{-it2^{4N}s^4}\partial_{s}\big(s^{-3}\varphi_0(s)e^{\pm i2^Ns\Psi(z)}\Phi(2^Ns;z)\big)ds.
\end{split}
\end{equation*}
Noting that for any $s\in\text{supp}\ \varphi_0(s)$ and $k=0,1,$
\[\big|\partial_{s}^k\big(e^{\pm i2^Ns\Psi(z)}\big)\big|\lesssim 1,\ \ \ \big|\partial_{s}^k\big(\Phi(2^Ns;z\big)\big|\lesssim 1,\]
thus we can obtain that
\begin{equation}\label{eq-K-3}
\big|K_N^\pm(t,z)\big|\lesssim \big(|t|\cdot 2^{4N}\big)^{-1}.
\end{equation}
Combining \eqref{eq-K-2} with \eqref{eq-K-3}, we obtain that
\[\sup_{z\in\mathbf{R}^m}\big|K_N^\pm(t,z)\big|\lesssim (1+|t|\cdot 2^{4N})^{-\frac{1}{2}}.\]
\textbf{Case~2:} $2^N\Psi(z)\ge 1.$  Set $r=2^N\Psi(z).$ In this case, we rewrite $K_N^\pm(t,z)$ as follows:
\[K_N^\pm(t,z)=\int_0^\infty e^{-it2^{4N}s^4\pm isr}\varphi_0(s)\Phi(2^Ns;z)ds.\]
Assume that $t>0,$ we consider $K_N^-(t,z)$ and $K_N^+(z)$ case by case.

Firstly, we study $K_N^-(t,z)$. Let $u_-(s,r)=t\cdot 2^{4N}s^4+sr,$ then we have
\[K_N^-(t,z)=\int_0^\infty e^{-iu_-(s,r)}\varphi_0(s)\Phi(2^Ns;z)ds.\]
It is easy to know that
\[\big|\partial_su_-(s,r)\big|=\big|4t\cdot 2^{4N}s^3+r\big|\gtrsim(1+|t|\cdot 2^{4N}).\]
Thus $u_-(s,r)$ does not exist critical point, by \eqref{eq-K-2} and integrating by parts once, we have
\[\big|K_N^-(t,z)\big|\lesssim (1+|t|\cdot 2^{4N})^{-\frac{1}{2}}.\]

Next to consider $K_N^+(t,z)$. Let $u_+(s,r)=t\cdot 2^{4N}s^4-sr,$ then we have
\[K_N^+(t,z)=\int_0^\infty e^{-iu_+(s,r)}\varphi_0(s)\Phi(2^Ns;z)ds.\]
It is obvious that
\[\partial_su_+(s,r)=4t\cdot 2^{4N}s^3-r\]
has a critical point $s_0$ such $r=4t\cdot 2^{4N}s_0^3.$

If $r\ge 8t\cdot 2^{4N}$ or $r\le \frac{1}{32}t\cdot 2^{4N},$ then the such critical point $s_0$ does not exist in the range $[1/4,1]$, so by \eqref{eq-K-2} and using integration by parts once, we have
\[\big|K_N^+(t,z)\big|\lesssim (1+|t|\cdot 2^{4N})^{-\frac{1}{2}}.\]

If $\frac{1}{32}t\cdot 2^{4N}\le r\le 8t\cdot 2^{4N},$ then the critical point exists. But since
$\big|\partial_{s}^2u_+(s,r)\big|=\big|12t\cdot 2^{4N}s^2\big|\gtrsim|t|\cdot 2^{4N},$
thus by Van der Corput lemma ( see e.g. \cite{Stein}), we have
\[\big|K_N^+(t,z)\big|\lesssim (1+|t|\cdot 2^{4N})^{-\frac{1}{2}}.\]
Hence, by summing all cases above we have that
\[\sup_{z\in\mathbf{R}^m}\big|K_N^\pm(t,z)\big|\lesssim (1+|t|\cdot 2^{4N})^{-\frac{1}{2}},\]
which concludes the proof of this lemma.
\end{proof}

Using Lemma \ref{lemma-low-energy-2}, we can give the following decay estimates of $e^{-it\Delta^2}$ in one dimensional case by the free resolvent $R_0^\pm(\mu^4).$
Note  that we can write down
\begin{align}\label{free formula}
e^{-it\Delta^2}f(x)=\int_{\mathbf{R}}\Big(\frac{2}{\pi i}\int_0^\infty e^{-it\mu^4}\mu^3[R_0^+(\mu^4)-R_0^-(\mu^4)](x,y)d\mu\Big) f(y)dy,
\end{align}
where
\begin{equation*}\label{eq-free kernel-1}
\begin{split}
R_0^\pm(\mu^4)(x,y)=&\frac{1}{4\mu^3}\Big(\pm ie^{\pm i\mu|x-y|}-e^{-\mu|x-y|}\Big)
=\frac{e^{\pm i\mu|x-y|}}{4\mu^3}\big(\pm i-e^{-\mu|x-y|\mp i\mu|x-y|}\big).
\end{split}
\end{equation*}

\begin{proposition}\label{prop-free estimates}
For any $N\in\mathbf{Z},$ we have
\begin{equation}\label{eq-free estimates-1}
\sup\limits_{x,y\in\mathbf{R}}\Big|\int_0^\infty e^{-it\mu^4}\mu^3\varphi_N(\mu)R_0^\pm(\mu^4)(x,y)d\mu\Big|\lesssim 2^{N}(1+|t|\cdot 2^{4N})^{-\frac{1}{2}}.
\end{equation}
Furthermore, we have
\begin{equation}\label{eq-free estimates-2}
\sup\limits_{x,y\in\mathbf{R}}\Big|\int_0^\infty e^{-it\mu^4}\mu^3R_0^\pm(\mu^4)(x,y)d\mu\Big|\lesssim |t|^{-\frac{1}{4}}.
\end{equation}
As a consequence, it follows that
\begin{equation}\label{free estimates}
\big\|e^{-it\Delta^2}\big\|_{L^1(\mathbf{R})\rightarrow L^\infty(\mathbf{R})}\lesssim |t|^{-\frac{1}{4}}, \ \ t\neq 0.
\end{equation}
\end{proposition}
\begin{proof}
Since $\varphi_N(\mu)=\varphi_0(2^{-N}\mu),$ thus set
\begin{equation}\label{K_0}
\begin{split}
K_{0,N}^\pm(t;x,y):= \int_0^\infty e^{-it\mu^4}\mu^3\varphi_0(2^{-N}\mu)[R_0^\pm(\mu^4)](x,y)d\mu.
\end{split}
\end{equation}
Let $\mu=2^Ns,$ then we have
\begin{align*}
 K_{0,N}^\pm(t;x,y)
 =&2^{N}\int_0^\infty e^{-it2^{4N}s^4}\varphi_0(s)e^{\pm i2^Ns|x-y|}\big(\pm i-e^{-2^Ns|x-y|\mp i2^Ns|x-y|}\big)ds.
\end{align*}
Let $z=(x,y),$ and
\[\Psi(z)=\Psi(x,y)=|x-y|,\ \ \ \Phi^\pm(s;z)=\Phi^\pm(s;x,y)=\pm i-e^{-s|x-y|\mp is|x-y|}.\]
It is obvious that for any $s\in\text{supp}\ \varphi_0(s)$ and $k=0,1,$ ~$\big|\partial_{s}^k\big(\Phi^\pm(2^Ns;x,y)\big)\big|\lesssim 1,$~
then by Lemma \ref{lemma-low-energy-2}, we obtain that
\[\sup_{x,y\in\mathbf{R}}\big|K_{0,N}^\pm(t;x,y)\big|\lesssim 2^N(1+|t|\cdot 2^{4N})^{-\frac{1}{2}}.\]
Denote
\[K_0^\pm(t;x,y)=\int_0^\infty e^{-it\mu^4}\mu^3R_0^\pm(\mu^4)(x,y)d\mu.\]
Thus
\begin{align*}
K_0^\pm(t;x,y)=&\sum_{N=-\infty}^\infty\int_0^\infty e^{-it\mu^4}\mu^3\varphi_N(\mu)[R_0^\pm(\mu^4)](x,y)d\mu\\
=&\displaystyle\sum_{N=-\infty}^{+\infty}K_{0,N}^{\pm}(t;x,y).
\end{align*}
It is obvious that there exists $N'\in\mathbf{Z}$ such that $|t|\cdot 2^{4N'}\sim1,$ so we have
\begin{align*}
\big|K_0^\pm(t;x,y)\big|\lesssim & \sum_{N=-\infty}^{N'}\big|K_{0,N}^{\pm}(t;x,y)\big|+\sum_{N=N'+1}^{+\infty}\big|K_{0,N}^{\pm}(t;x,y)\big|\\
\lesssim& \sum_{N=-\infty}^{N'}2^N(1+|t|\cdot 2^{4N})^{-\frac{1}{2}}+\sum_{N=N'+1}^{+\infty}2^N(1+|t|\cdot 2^{4N})^{-\frac{1}{2}}\\
\lesssim& \sum_{N=-\infty}^{N'}2^N+|t|^{-\frac{1}{2}}\sum_{N=N'+1}^{+\infty}2^{-N}\\
\lesssim& |t|^{-\frac{1}{4}},
\end{align*}
which immediately gives the \eqref{eq-free estimates-2}.
\end{proof}
\subsection{Low Energy Dispersive Estimate with $\alpha=0$} In this subsection, we show the low energy dispersive estimate \eqref{eq-low result-1} case by case for zero is a regular point or zero energy resonance.

The following lemma play an important role to make use of cancellations from projections operator $S_j$  in the zero asymptotical expansions of resolvent $R_V(z)$ ( see Theorem \ref{thm-M-inverse} above ), and will be used repeatedly to obtain low energy dispersive estimates for all cases.

\begin{lemma}\label{lemma-low-energy}
Let $\mu>0$ and ${\rm sgn}(x)$ be the sign function of $x$ on $\mathbf{R}$. Then

(i)~If $F(p)\in C^1(\mathbf{R})$, then for any  $x, y\in\mathbf{R}$, we have
\begin{equation}\label{eq-low energy-1}
\begin{split}
F(\mu|x-y|)
=F(\mu|x|)-\mu y \int_0^1{\rm sgn}(x-\theta y)F'(\mu|x-\theta y|)d\theta.
\end{split}
\end{equation}

(ii)~If~$F(p)\in C^2(\mathbf{R})$ and $F'(0)=0$~, then for any  $x, y\in\mathbf{R}$, we have
\begin{equation}\label{eq-low energy-2}
\begin{split}
F(\mu|x-y|)
=F(\mu|x|)-\mu y\ {\rm sgn}(x)F'(\mu|x|)+\mu^2y^2\int_0^1(1-\theta)F''(\mu|x-\theta y|)d\theta.
\end{split}
\end{equation}

(iii)~If~$F(p)\in C^3(\mathbf{R})$ and $F'(0)=F''(0)=0,$ then for any  $x, y\in\mathbf{R}$, we have
\begin{equation}\label{eq-low energy-3}
\begin{split}
F(\mu|x-y|)
=&F(\mu|x|)-\mu y\ {\rm sgn}(x)F'(\mu|x|)+\frac{\mu^2 y^2}{2!}F''(\mu|x|)\\
&\ -\frac{\mu^3y^3}{2!}\int_0^1(1-\theta)^2\big({\rm sgn}(x-\theta y)\big)^3F^{(3)}(\mu|x-\theta y|)d\theta.
\end{split}
\end{equation}

\end{lemma}
\begin{proof}
Denote $$G_\varepsilon(y)=F\big(\mu\sqrt{\varepsilon^2+(x-y)^2}\big),\ \ \ \varepsilon\neq0. $$
If $F(p)\in C^{n+1}(\mathbf{R}),$ then $G_\varepsilon(y)\in C^{n+1}(\mathbf{R})$ for $\varepsilon\neq0.$ By Taylor expansions, we have
\begin{equation}\label{eq-taylor expansion}
G_\varepsilon(y)=\sum_{k=0}^n\frac{G_\varepsilon^{(k)}(0)}{k!}y^k+\frac{y^{n+1}}{n!}\int_0^1(1-\theta)^nG_\varepsilon^{(n+1)}(\theta y)d\theta.
\end{equation}

(i)~Since $F(p)\in C^1(\mathbf{R}),$ then $\displaystyle\lim_{\varepsilon\rightarrow0}G_\varepsilon(y)=\lim_{\varepsilon\rightarrow0}F(\mu\sqrt{\varepsilon^2+(x-y)^2})=F(\mu|x-y|).$ Note that for $\varepsilon\neq 0,$
\[G'_\varepsilon(y)=\frac{-\mu(x-y)}{\sqrt{(\varepsilon^2+(x-y)^2)}}F'\big(\mu\sqrt{\varepsilon^2+(x-y)^2}\big),\]
then
for each $0\le \theta\le 1$ and $x, y\in \mathbf{R}$,
\begin{align}\label{one-derivative}
\lim_{\varepsilon\rightarrow0}G'_\varepsilon(\theta y)=-\mu{\rm sgn}(x-\theta y)F'\big(\mu|x-\theta y|\big).
\end{align}
Since there exists a constant $C=C(\mu, x, y)>0$ such that $|G'_\varepsilon(\theta y)|\le C$ for all $0\le \theta\le 1$ and $0<\varepsilon\le 1$, then by taking $\varepsilon\downarrow0$ in the two sides of \eqref{eq-taylor expansion} with $n=0$ and Lebesgue's dominated convergence theorem,  it immediately follows that
\[F(\mu|x-y|)
=F(\mu|x|)-\mu y\int_0^1{\rm sgn}(x-\theta y)F'(\mu|x-\theta y|)d\theta.\]

(ii)~If $F(p)\in C^2(\mathbf{R}),$ then for $\varepsilon\neq 0,$
\begin{align*}
G''_\varepsilon(y)=&\mu^2\Big(\frac{x-y}{\sqrt{\varepsilon^2+(x-y)^2}}\Big)^2F''(\mu\sqrt{\varepsilon^2+(x-y)^2})\\
&+\mu\frac{\varepsilon^2}{\big(\sqrt{\varepsilon^2+(x-y)^2}\big)^2}\cdot\frac{F'(\mu\sqrt{\varepsilon^2+(x-y)^2})}{\sqrt{\varepsilon^2+(x-y)^2}}.
\end{align*}
Since $\displaystyle F'(0)=0,$ thus we can obtain that
\begin{align}\label{second-derivative}\lim_{\varepsilon\rightarrow0}G''_\varepsilon(y)=\mu^2 F''\big(\mu|x-y|\big), \ x,y\in\mathbf{R}.\end{align}
Then by using \eqref{eq-taylor expansion} with $n=1 $, \eqref{one-derivative}, \eqref{second-derivative} and  Lebesgue's dominated convergence theorem, we can conclude that
\begin{align*}
F(\mu|x-y|)
=&\sum_{k=0}^1\frac{\displaystyle\lim_{\varepsilon\rightarrow0}G_\varepsilon^{(k)}(0)}{k!}y^k+y^2\cdot\lim_{\varepsilon\rightarrow0}\int_0^1(1-\theta)G_\varepsilon^{''}(\theta y)d\theta.\\
=&F(\mu|x|)-\mu y{\rm sgn}(x)F'(\mu|x|)+\mu^2y^2\int_0^1(1-\theta)F''(\mu|x-\theta y|)d\theta.
\end{align*}

(iii)~If $F(p)\in C^3(\mathbf{R}),$ then for $\varepsilon\neq 0,$
\begin{align*}
G^{(3)}_\varepsilon(y)=&-\mu^3\Big(\frac{x-y}{\sqrt{\varepsilon^2+(x-y)^2}}\Big)^3F^{(3)}(\mu\sqrt{\varepsilon^2+(x-y)^2})\\
&-3\mu^2\frac{(x-y)\varepsilon^2}{\big(\sqrt{\varepsilon^2+(x-y)^2}\big)^3}\cdot\frac{F''(\mu\sqrt{\varepsilon^2+(x-y)^2})}{\sqrt{\varepsilon^2+(x-y)^2}}\\
&+3\mu\frac{(x-y)\varepsilon^2}{\big(\sqrt{\varepsilon^2+(x-y)^2}\big)^3}\cdot\frac{F'(\mu\sqrt{\varepsilon^2+(x-y)^2})}{\big(\sqrt{\varepsilon^2+(x-y)^2}\big)^2}.
\end{align*}
Then for each $x,y\in\mathbf{R},$ we have
\begin{align}\label{third-derivative}
\lim_{\varepsilon\rightarrow0}G^{(3)}_\varepsilon(y)
=-\mu^3({\rm sgn}(x-y))^3F^{(3)}\big(\mu|x-y|\big).
\end{align}
Since $F'(0)=F''(0)=0,$ so by using the expansion \eqref{eq-taylor expansion} with $n=2 $, the equalities \eqref{one-derivative}-\eqref{third-derivative} and Lebesgue's dominated convergence theorem, we can similarly obtain that
\begin{align*}
F(\mu|x-y|)
=&\sum_{k=0}^2\frac{\displaystyle\lim_{\varepsilon\rightarrow0}G_{\varepsilon}^{(k)}(0)}{k!}y^k+
\frac{y^{3}}{2!}\cdot\lim_{\varepsilon\rightarrow0}\int_0^1(1-\theta)^2G_{\varepsilon}^{(3)}(\theta y)d\theta\\
=&F(\mu|x|)-\mu y{\rm sgn}(x)F'(\mu|x|)+\frac{\mu^2 y^2}{2!}F''(\mu|x|)\\
&\ -\frac{\mu^3y^3}{2!}\int_0^1(1-\theta)^2\big({\rm sgn}(x-\theta y)\big)^3F^{(3)}(\mu|x-\theta y|)d\theta.
\end{align*}


\end{proof}

\begin{lemma}\label{lemma-low sum} For any $0\le\alpha<1,$  we have
\begin{equation}\label{eq-low sum}
\sum_{N=-\infty}^{+\infty} 2^{(1+\alpha)N}(1+|t|\cdot 2^{4N})^{-\frac{1}{2}}\lesssim |t|^{-\frac{1+\alpha}{4}}.
\end{equation}

\end{lemma}
\begin{proof} Since for any $t\neq 0,$ there exists a $N_0'\in \mathbf{Z}$ such that $|t|\cdot 2^{4N_0'}\sim 1,$ thus we have
\begin{align*}
\sum_{N=-\infty}^{+\infty} 2^{(1+\alpha)N}(1+|t|\cdot 2^{4N})^{-\frac{1}{2}}=&\sum_{N=-\infty}^{N_0'} 2^{(1+\alpha)N}(1+|t|\cdot 2^{4N})^{-\frac{1}{2}}+\sum_{N=N_0'+1}^{+\infty} 2^{(1+\alpha)N}(1+|t|\cdot 2^{4N})^{-\frac{1}{2}}\\
\lesssim & \sum_{N=-\infty}^{N_0'} 2^{(1+\alpha)N}+|t|^{-\frac{1}{2}}\sum_{N=N_0'+1}^{+\infty}2^{(\alpha-1)N}\\
\lesssim & |t|^{-\frac{1+\alpha}{4}}.
\end{align*}
\end{proof}
\subsubsection{\textbf{Regular Case}}

In order to establish the lower energy estimate \eqref {eq-low result-1} with $\alpha=0$ in regular case, recall that by Stone's formula ( also see the equality \eqref{Stone formula})\begin{align}\label{lower energy Stone formula}
e^{-itH}P_{ac}(H)\chi(H)f=&\frac{2}{\pi i}\int_0^\infty e^{-it\mu^4}\chi(\mu )\mu^3[R_V^+(\mu^4)-R_V^-(\mu^4)]fd\mu\\
=&\sum_{N=-\infty}^{N_0}\sum_{\pm}\frac{\pm1}{2\pi i}\int_0^\infty e^{-it\mu^4}\varphi_N(\mu)\mu^3R_V^{\pm}(\mu^4)fd\mu.
\end{align}
When zero is a regular point of the spectrum of $H$, by \eqref{eq-symmetric identity} and \eqref{eq-M-inverse-Regular Case}, we have
\begin{equation}\label{eq-Rv-expansion-regualr}
\begin{split}
R_V^\pm(\mu^4) =& R_0^\pm(\mu^4)-R_0^\pm(\mu^4)v\Big(S_0A_{01}^0S_0\Big)vR_0^\pm(\mu^4)-R_0^\pm(\mu^4)v\Big(\mu QA_{11}^0Q\Big)vR_0^\pm(\mu^4)\\
&\ \  -R_0^\pm(\mu^4)v\Big(\mu^2 QA_{21}^0Q\Big)vR_0^\pm(\mu^4)-R_0^\pm(\mu^4)v\Big(\mu^2 S_0A_{22}^0+\mu^2 A_{23}^0S_0\Big)vR_0^\pm(\mu^4)\\
&\ \  -R_0^\pm(\mu^4)v\Gamma_3^0(\mu)vR_0^\pm(\mu^4).
\end{split}
\end{equation}
If we substitute the \eqref{eq-Rv-expansion-regualr} into the \eqref{lower energy Stone formula}, and combine Proposition \ref{prop-free estimates} and Lemma \ref{lemma-low sum}, then it suffices to show the following Propositions \ref{pro-low energy-regular case},\ \ref{pro-Q-1},\ \ref{pro-Q-2} and \ref{pro-Gamma_3} corresponding to each different terms of the resolvent formula \eqref{eq-Rv-expansion-regualr}  above.

\begin{proposition}\label{pro-low energy-regular case}
For any $N\in\mathbf{Z}_{-}$ and $N\le N_0,$ we have
\begin{equation}\label{eq-low energy estimates-regular case-1}
\sup\limits_{x,y\in\mathbf{R}}\Big|\int_0^\infty e^{-it\mu^4}\mu^3\varphi_N(\mu)\Big[R_0^\pm(\mu^4)v\big(S_0A_{01}^0S_0\big)vR_0^\pm(\mu^4)\Big](x,y)d\mu\Big|\lesssim 2^{2N}(1+|t|\cdot 2^{4N})^{-\frac{1}{2}}.
\end{equation}
\end{proposition}
\begin{proof}
Since $\varphi_N(\mu)=\varphi_0(2^{-N}\mu),$ thus
\begin{equation*}
\begin{split}
K_{1,N}^{0,\pm}(t;x,y):=&\int_0^\infty e^{-it\mu^4}\mu^3\varphi_N(\mu)\Big[R_0^\pm(\mu^4)v\big(S_0A_{01}^0S_0\big)vR_0^\pm(\mu^4)\Big](x,y)d\mu\\
=&\int_0^\infty e^{-it\mu^4}\mu^3\varphi_0(2^{-N}\mu)\Big[R_0^\pm(\mu^4)v\big(S_0A_{01}^0S_0\big)vR_0^\pm(\mu^4)\Big](x,y)d\mu.
\end{split}
\end{equation*}
Indeed, since
\[R_0^\pm(x,y)=\frac{\pm ie^{\pm i\mu|x-y|}-e^{-\mu|x-y|}}{4\mu^3}:=\frac{1}{4\mu^3}F^\pm(\mu|x-y|),\]
and $(F^\pm)'(0)=0$ , then by Lemma \eqref{lemma-low-energy}(ii) and using the fact that $S_0(v)=S_0(xv(x))=0,$ we have
\begin{equation*}
\begin{split}
\ \big[R_0^\pm(\mu^4)v&\big(S_0A_{01}^0S_0\big)vR_0^\pm(\mu^4)\big](x,y)\\
=&\frac{1}{16\mu^6}\int_{\mathbf{R}^2}F^\pm(\mu|x-y_2|)\big[vS_0A_{01}^0S_0v\big](y_2,y_1)F^\pm(\mu|y-y_1|)dy_1dy_2\\
=&\frac{1}{16\mu^2}\int_{\mathbf{R}^2}\Big(\int_0^1\int_0^1(1-\theta_2)(1-\theta_1)(F^\pm)''(\mu |x-\theta_2y_2|)\times\\
&\ \ \ \ \ \ \ \ \ \ \ \ \ \ \ \ \ \ \
(F^\pm)''(\mu |y-\theta_1y_1|)d\theta_1d\theta_2\Big)y_1^2y_2^2\big[vS_0A_{01}^0S_0v\big](y_2,y_1)dy_1dy_2.
\end{split}
\end{equation*}
Then we obtain that
\begin{equation}\label{eq-K_1^0-1}
\begin{split}
&\ K_{1,N}^{0,\pm}(t;x,y)\\
=&\frac{1}{16}\int_0^\infty e^{-it\mu^4}\mu\varphi_0(2^{-N}\mu)\Bigg[\int_{\mathbf{R}^2}\Big(\int_0^1\int_0^1(1-\theta_2)(1-\theta_1)(F^\pm)''(\mu |x-\theta_2y_2|)\\
&\ \ \ \ \ \ \ \ \ \ \ \ \ \ \ \ \ \ \ \ \ \ \ \ \ \ \
(F^\pm)''(\mu |y-\theta_1y_1|)d\theta_1d\theta_2\Big)y_1^2y_2^2\big[vS_0A_{01}^0S_0v\big](y_2,y_1)dy_1dy_2\Bigg]d\mu\\
=&\frac{1}{16}\int_{\mathbf{R}^2}\Bigg[\int_0^1\int_0^1\bigg(\int_0^\infty e^{-it\mu^4}\mu\varphi_0(2^{-N}\mu)
(F^\pm)''(\mu |x-\theta_2y_2|)(F^\pm)''(\mu |y-\theta_1y_1|)d\mu\bigg)\\
&\ \ \ \ \ \ \ \ \ \ \ \ \ \ \ \ \ \ \ \ \ \ \ \ \ \ \ \ \ \ \ \ \ \
(1-\theta_1)(1-\theta_2)d\theta_1d\theta_2\Bigg]y_1^2y_2^2\big[vS_0A_{01}^0S_0v\big](y_2,y_1)dy_1dy_2.\\
\end{split}
\end{equation}
Denote
\begin{align*}
&\ E_{1,N}^{0,\pm}(t;x,y,\theta_1,\theta_2,y_1,y_2)\\
=&\int_0^\infty e^{-it\mu^4}\mu\varphi_0(2^{-N}\mu)(F^\pm)''(\mu |x-\theta_2y_2|)
(F^\pm)''(\mu |y-\theta_1y_1|)d\mu.
\end{align*}
It is obvious that
\[\Big|K_{1,N}^{0,\pm}(t;x,y)\Big|\lesssim\int_{\mathbf{R}^2}\Big(\int_0^1\int_0^1 \big|E_{1,N}^{0,\pm}(t;x,y,\theta_1,\theta_2,y_1,y_2)\big|d\theta_1d\theta_2\Big)\big|[vS_0A_0^0S_0v](y_2,y_1)\big|y_1^2y_2^2dy_1dy_2.\]

Now we estimate $E_{1,N}^{0,\pm}(t;x,y,\theta_1,\theta_2,y_1,y_2).$ Since
\[(F^{\pm})''(\mu|x-\theta_2y_2|)=\mp ie^{\pm i\mu|x-\theta_2y_2|}-e^{-\mu|x-\theta_2y_2|}:=e^{\pm i\mu|x-\theta_2y_2|}\mathcal{F}^\pm(\mu|x-\theta_2y_2|),\]
\[(F^{\pm})''(\mu|y-\theta_1y_1|)=\mp ie^{\pm i\mu|y-\theta_1y_1|}-e^{-\mu|y-\theta_1y_1|}:=e^{\pm i\mu|y-\theta_1y_1|}\mathcal{F}^\pm(\mu|y-\theta_1y_1|),\]
and let $\mu=2^Ns,$ we can rewrite $E_{1,N}^{0,\pm}(t;x,y,\theta_1,\theta_2,y_1,y_2)$ as follows:
\begin{equation*}
\begin{split}
E_{1,N}^{0,\pm}(t;x,y,\theta_1,\theta_2,y_1,y_2)=&2^{2N}\int_0^\infty e^{-it2^{4N}s^4}s\varphi_0(s)e^{\pm i2^Ns|x-\theta_2y_2|}\mathcal{F}^\pm(2^Ns|x-\theta_2y_2|)\\
&\ \ \ \ \ \ \ \ \ \ \ \ \ \ \ \ \ \ \ \ \ \ \ \ \ \ \ \ \ \times e^{\pm i2^Ns|y-\theta_1y_1|}\mathcal{F}^\pm(2^Ns|y-\theta_1y_1|)ds.
\end{split}
\end{equation*}
Noting that for any $s\in\text{supp}\ \varphi_0(s)$ and $k=0,1,$ we have
\[\Big|\partial_{s}^k\mathcal{F}^\pm(2^Ns|x-\theta_2y_2|)\Big|\lesssim 1,\ \ \ \ \ \ \Big|\partial_{s}^k\mathcal{F}^\pm(2^Ns|y-\theta_1y_1|)\Big|\lesssim 1,\]
Using Lemma \ref{lemma-low-energy-2} with $z=(x,y,y_1,y_2,\theta_1,\theta_2)$ and
\[\Psi(z)=|x-\theta_2 y_2|+|y-\theta_1 y_1|,\ \ \ \ \Phi^\pm(2^Ns;z)=\mathcal{F}^\pm(2^Ns|x-\theta_2y_2|)\cdot\mathcal{F}^\pm(2^Ns|y-\theta_1y_1|),\]
we obtain that
\[\sup_{x,y\in\mathbf{R}}\big|E_{1,N}^{0,\pm}(t;x,y,\theta_1,\theta_2,y_1,y_2)\big|\lesssim 2^{2N}(1+|t|\cdot 2^{4N})^{-\frac{1}{2}}.\]
Noting that $v(x)\lesssim(1+|x|)^{-\frac{13}{2}-}$ in this case, thus we have
\begin{align*}
\big|K_{1,N}^{0,\pm}(t;x,y)\big|\lesssim & 2^{2N}(1+|t|\cdot 2^{4N})^{-\frac{1}{2}}\int_{\mathbf{R}^2}y_2^2|v(y_2)|\big|[S_0A_0^0S_0](y_2,y_1)\big||v(y_1)|y_1^2dy_1dy_2\\
\lesssim& 2^{2N}(1+|t|\cdot 2^{4N})^{-\frac{1}{2}}\|y_2^2v(y_2)\|_{L^2_{y_2}}\cdot\|S_0A_{01}^0S_0\|_{L^2\rightarrow L_2}\cdot\|y_1^2v(y_1)\|_{L^2_{y_1}}\\
\lesssim& 2^{2N}(1+|t|\cdot 2^{4N})^{-\frac{1}{2}}
\end{align*}
uniformly in $x,y.$
\end{proof}

\begin{proposition}\label{pro-Q-1}
For any $N\in\mathbf{Z}_-$ and $N\le N_0$, we have
\begin{equation}\label{eq-low Q-1-1}
\sup\limits_{x,y\in\mathbf{R}}\Big|\int_0^\infty e^{-it\mu^4}\mu^3\varphi_N(\mu)\Big[R_0^\pm(\mu^4)v\big(\mu QA_{11}^0Q\big)vR_0^\pm(\mu^4)\Big](x,y)d\mu\Big|\lesssim 2^{N}(1+|t|\cdot 2^{4N})^{-\frac{1}{2}}.
\end{equation}
\end{proposition}

\begin{proof}
Since $\varphi_N(\mu)=\varphi_0(2^{-N}\mu)$, thus
\begin{equation*}
\begin{split}
K_{2,N}^{0,\pm}(t;x,y):=&\int_0^\infty e^{-it\mu^4}\mu^3\varphi_0(2^{-N}\mu)\Big[R_0^\pm(\mu^4)v\big(\mu QA_{11}^0Q\big)vR_0^\pm(\mu^4)\Big](x,y)d\mu.
\end{split}
\end{equation*}
It suffices to prove that
\begin{equation*}
\sup_{x,y}\Big|K_{2,N}^{0,\pm}(t;x,y)\Big|\lesssim 2^{N}(1+|t|\cdot 2^{4N})^{-\frac{1}{2}}.
\end{equation*}
Indeed, since
\[R_0^\pm(x,y)=\frac{\pm ie^{\pm i\mu|x-y|}-e^{-\mu|x-y|}}{4\mu^3}:=\frac{1}{4\mu^3}F^\pm(\mu|x-y|),\]
by Lemma \eqref{lemma-low-energy}(i) and using the fact that $Q(v)=0,$ we have
\begin{equation*}
\begin{split}
&\ \Big[R_0^\pm(\mu^4)v\big(\mu QA_{11}^0Q\big)vR_0^\pm(\mu^4)\Big](x,y)\\
=&\frac{1}{16\mu^5}\int_{\mathbf{R}^2}F(\mu|x-y_2|)\big[vQA_{11}^0Qv\big](y_2,y_1)F(\mu|y-y_1|)dy_1dy_2\\
=&\frac{1}{16\mu^3}\int_{\mathbf{R}^2}\Big(\int_0^1\int_0^1(F^\pm)'(\mu |x-\theta_2y_2|){\rm sgn}(x-\theta_2y_2)\big[vQA_{11}^0Qv\big](y_2,y_1)\\
&\ \ \ \ \ \ \ \ \ \ \ \ \ \ \ \ \ \ \ \ \ \ \ \ \ \ \ \ \ \ \ \ \ \ \ \ \ \ \ \ \
(F^\pm)'(\mu |y-\theta_1y_1|){\rm sgn}(y-\theta_1y_1)d\theta_1d\theta_2\Big)y_1y_2dy_1dy_2.
\end{split}
\end{equation*}
Then we obtain that
\begin{equation}\label{eq-K_2^0-1}
\begin{split}
&\ K_{2,N}^{0,\pm}(t;x,y)\\
=&\frac{1}{16}\int_{\mathbf{R}^2}\Bigg[\int_0^1\int_0^1{\rm sgn}(y-\theta_1y_1){\rm sgn}(x-\theta_2y_2)\Big(\int_0^\infty e^{-it\mu^4}\varphi_0(2^{-N}\mu)(F^\pm)'(\mu |x-\theta_2y_2|)\\
&\ \ \ \ \ \ \ \ \ \ \ \ \ \ \ \ \ \ \ \ \ \ \ \ \ \ \ \ \
(F^\pm)'(\mu |y-\theta_1y_1|)d\mu\Big)d\theta_1d\theta_2\Bigg]\big[vQA_{11}^0Qv\big](y_2,y_1) y_1y_2dy_1dy_2.\\
\end{split}
\end{equation}
Denote
\begin{align*}
\ E_{2,N}^{0,\pm}(t;x,y,\theta_1,\theta_2,y_1,y_2)
=\int_0^\infty e^{-it\mu^4}\varphi_0(2^{-N}\mu)(F^\pm)'(\mu |x-\theta_2y_2|)
(F^\pm)'(\mu |y-\theta_1y_1|)d\mu.
\end{align*}
Then
\[\Big|K_{2,N}^{0,\pm}(t;x,y)\Big|\lesssim\int_{\mathbf{R}^2}\Big(\int_0^1\int_0^1\big|E_{2,N}^{0,\pm}(t;x,y,\theta_1,\theta_2,y_1,y_2)\big|d\theta_1d\theta_2 \Big)\big|[vQA_{11}^0Qv]\big|(y_2,y_1)\ |y_1y_2|dy_1dy_2.\]
Next to estimate $E_{2,N}^{0,\pm}(t;x,y,\theta_1,\theta_2,y_1,y_2).$ Since
\[(F^{\pm})'(\mu|x-\theta_2y_2|)=-e^{\pm i\mu|x-\theta_2y_2|}+e^{-\mu|x-\theta_2y_2|}:=e^{\pm i\mu|x-\theta_2y_2|}\mathcal{F}^\pm(\mu|x-\theta_2y_2|),\]
\[(F^{\pm})'(\mu|y-\theta_1y_1|)=-e^{\pm i\mu|y-\theta_1y_1|}+e^{-\mu|y-\theta_1y_1|}:=e^{\pm i\mu|y-\theta_1y_1|}\mathcal{F}^\pm(\mu|y-\theta_1y_1|),\]
and let $\mu=2^Ns,$ we can rewrite $E_{2,N}^{0,\pm}(t;x,y,\theta_1,\theta_2,y_1,y_2)$ as follows:
\begin{equation*}
\begin{split}
E_{2,N}^{0,\pm}(t;x,y,\theta_1,\theta_2,y_1,y_2)=&2^{N}\int_0^\infty e^{-it2^{4N}s^4}\varphi_0(s)e^{\pm i2^Ns|x-\theta_2y_2|}\mathcal{F}^\pm(2^Ns|x-\theta_2y_2|)\\
&\ \ \ \ \ \ \ \ \ \ \ \ \ \ \ \ \ \ \ \ \ \ \ \ \ \ \ \ \ \ \ \ \ \ \ \ \   e^{\pm i2^Ns|y-\theta_1y_1|}\mathcal{F}^\pm(2^Ns|y-\theta_1y_1|)ds.
\end{split}
\end{equation*}
Noting that for any $s\in\text{supp}\ \varphi_0(s)$ and $k=0,1,$ we have
\[\big|\partial_{s}^k\mathcal{F}^\pm(2^Ns|x-\theta_2y_2|)\big|\lesssim 1,\ \ \ \ \ \ \big|\partial_{s}^k\mathcal{F}^\pm(2^Ns|y-\theta_1y_1|)\big|\lesssim 1,\]
then using Lemma \ref{lemma-low-energy-2} with $z=(x,y,y_1,y_2,\theta_1,\theta_2)$ and
\[\Psi(z)=|x-\theta_2 y_2|+|y-\theta_1 y_1|,\ \ \ \ \Phi^\pm(2^Ns;z)=\mathcal{F}^\pm(2^Ns|x-\theta_2y_2|)\cdot\mathcal{F}^\pm(2^Ns|y-\theta_1y_1|),\]
we have
\[\sup_{x,y\in\mathbf{R}}\big|E_{2,N}^{0,\pm}(t;x,y,\theta_1,\theta_2,y_1,y_2)\big|\lesssim 2^{N}(1+|t|\cdot 2^{4N})^{-\frac{1}{2}}.\]
Furthermore, since $v(x)\lesssim(1+|x|)^{-\frac{13}{2}-},$ we have
\begin{align*}
\big|K_{2,N}^{0,\pm}(t;x,y)\big|\lesssim & 2^{N}(1+|t|\cdot 2^{4N})^{-\frac{1}{2}}\int_{\mathbf{R}^2}|y_2 v(y_2)|\big|[QA_{11}^0Q](y_2,y_1)\big||y_1v(y_1)|dy_1dy_2\\
\lesssim& 2^{N}(1+|t|\cdot 2^{4N})^{-\frac{1}{2}}\|y_2v(y_2)\|_{L^2_{y_2}}\cdot\|QA_{11}^0Q\|_{L^2\rightarrow L_2}\cdot\|y_1v(y_1)\|_{L^2_{y_1}}\\
\lesssim& 2^{N}(1+|t|\cdot 2^{4N})^{-\frac{1}{2}}
\end{align*}
uniformly in $x,y.$
\end{proof}

By using the method in the proof of Proposition \ref{pro-low energy-regular case} and Proposition \ref{pro-Q-1}, it is easy similarly to obtain the following estimates ( we omit the details of calculations here ).
\begin{proposition}\label{pro-Q-2}
For any $N\in\mathbf{Z}_-$ and $N\le N_0$, we have
\begin{equation*}\label{eq-Q-2-1}
\sup\limits_{x,y}\Big|\int_0^\infty e^{-it\mu^4}\mu^3\varphi_N(\mu)\Big[R_0^\pm(\mu^4)v\big(\mu^2QA_{21}^0Q\big)vR_0^\pm(\mu^4)\Big](x,y)d\mu\Big|\lesssim 2^{2N}(1+|t|\cdot 2^{4N})^{-\frac{1}{2}}.
\end{equation*}
\begin{equation*}\label{eq-Q-2-2}
\sup\limits_{x,y}\Big|\int_0^\infty e^{-it\mu^4}\mu^3\varphi_N(\mu)\Big[R_0^\pm(\mu^4)v\big(\mu^2S_0A_{22}^0\big)vR_0^\pm(\mu^4)\Big](x,y)d\mu\Big|\lesssim 2^{2N}(1+|t|\cdot 2^{4N})^{-\frac{1}{2}}.
\end{equation*}
\begin{equation*}\label{eq-Q-2-3}
\sup\limits_{x,y}\Big|\int_0^\infty e^{-it\mu^4}\mu^3\varphi_N(\mu)\Big[R_0^\pm(\mu^4)v\big(\mu^2A_{23}^0S_0\big)vR_0^\pm(\mu^4)\Big](x,y)d\mu\Big|\lesssim 2^{2N}(1+|t|\cdot 2^{4N})^{-\frac{1}{2}}.
\end{equation*}
\end{proposition}

Finally, for the last remained term $R_0^\pm(\mu^4)v\Gamma_3^0(\mu)vR_0^\pm(\mu^4) $ in \eqref{eq-Rv-expansion-regualr}, we have the following estimates.
\begin{proposition}\label{pro-Gamma_3}
For any $N\in\mathbf{Z}_-$ and $N\le N_0$, we have
\begin{equation}\label{eq-low Gamma_3}
\sup\limits_{x,y}\Big|\int_0^\infty e^{-it\mu^4}\mu^3\varphi_N(\mu)\Big[R_0^\pm(\mu^4)v\Gamma_3^0(\mu)vR_0^\pm(\mu^4)\Big](x,y)d\mu\Big|\lesssim 2^{N}(1+|t|\cdot 2^{4N})^{-\frac{1}{2}}.
\end{equation}
\end{proposition}
\begin{proof}
Indeed, it suffices to prove that for any $f,g\in L^1(\mathbf{R}),$
\[\Big|\int_0^\infty e^{-it\mu^4}\mu^3\varphi_N(\mu)\big\langle [v\Gamma_3^0(\mu)v]R_0^\pm(\mu^4)f, (R_0^\pm)^*(\mu^4)g\big\rangle d\mu\Big|\lesssim 2^N(1+|t|\cdot 2^{4N})^{-\frac{1}{2}}\|f\|_{L^1}\cdot \|g\|_{L^1}.\]
Equivalently, it is enough to  bound the following kernel uniformly by  $ O(2^N(1+|t|\cdot 2^{4N})^{-\frac{1}{2}})$ :
\begin{align*}
K_{4,N}^{0,\pm}(t;x,y)=&\int_0^\infty e^{-it\mu^4}\mu^3\varphi_N(\mu)\Big\langle [v\Gamma_3^0(\mu)v]\big(R_0^\pm(\mu^4)(\ast,y)\big)(\cdot),R_0^\mp(\mu^4) (x,\cdot)\Big\rangle d\mu
\end{align*}
Noting that
\begin{align*}
R_0^\pm(\mu^4)(x,y)=&\frac{1}{4\mu^3}\big(\pm ie^{\pm i\mu|x-y|}-e^{-\mu|x-y|}\big)
:=\frac{e^{\pm i\mu|x-y|}}{4\mu^3}\mathcal{R}_0^\pm(\mu)(x,y),
\end{align*}
thus, we have
\begin{equation*}
\begin{split}
&\ \ \Big\langle [v\Gamma_3^0(\mu)v]\big(R_0^\pm(\mu^4)(\ast,y)\big)(\cdot),R_0^\mp(\mu^4) (x,\cdot)\Big\rangle\\
=&\frac{1}{16\mu^6}\Big\langle [v\Gamma_3^0(\mu)v]\big(e^{\pm i\mu|\ast-y|}\mathcal{R}_0^\pm(\mu)(\ast,y)\big)(\cdot),\big(e^{\mp i\mu|x-\cdot|}\mathcal{R}_0^\mp(\mu)(x,\cdot)\big)\Big\rangle\\
=&\frac{1}{16\mu^6}e^{\pm i\mu|x|}e^{\pm i\mu|y|}\Big\langle[v\Gamma_3^0(\mu)v]\big(e^{\pm i\mu(|\ast-y|- |y|)}\mathcal{R}_0^\pm(\mu)(\ast,y)\big)(\cdot),\big(e^{\mp i\mu(|x-\cdot|-|x|)}\mathcal{R}_0^\mp(\mu)(x,\cdot)\big)\Big\rangle\\
:=&\frac{1}{16\mu^6}e^{\pm i\mu(|x|+|y|)}E_{3,N}^{0,\pm}(\mu;x,y).
\end{split}
\end{equation*}
Let $\mu=2^Ns$, we have
\begin{equation*}\label{eq-low K_4^0-1}
\begin{split}
K_{4,N}^{0,\pm}(t;x,y)
=&2^{-2N}\int_0^\infty e^{-it2^{4N}s^4}\varphi_0(s)s^{-3}e^{\pm i2^Ns(|x|+|y|)}E_{3,N}^{0,\pm}(2^Ns;x,y)ds.
\end{split}
\end{equation*}
Noting that£¬ for any $s\in\text{supp}\ \varphi_0(s)$ and any $N\in\mathbf{Z}_-$, we have
\[\Big|\partial_{s}^k\Big(e^{\pm i2^N s(\ast-y|- |y|)}\mathcal{R}_0^\pm(2^Ns)(\ast,y)\Big)\Big|\lesssim \langle \ast\rangle^{k},\ \ k=0,1,\]
\[\Big|\partial_{s}^k\Big(e^{\pm i2^N s(|x-\cdot|- |x|)}\mathcal{R}_0^\pm(2^Ns)(x,\cdot)\Big)\Big|\lesssim \langle \cdot\rangle^{k},\ \ k=0,1,\]
and since $\|\Gamma_3^0(\mu)\|_{L^2\rightarrow L^2}=O_1(\mu^3)$ for $\mu<<1$  from Theorem \ref{thm-M-inverse} ( e.g. choosing $\mu<2^{N_0}$ for some large $N_0\in \mathbf{Z_-}$ ), then we have
\[\|\partial_{s}^k(\Gamma_3^0(2^Ns))\|_{L^2\rightarrow L^2}\lesssim 2^{3N}s^{3-k},\ \ \ k=0,1. \]
Furthermore, since
\begin{align*}
&\ \partial_{s}E_{3,N}^{0,\pm}(2^Ns;x,y)\\
=&\Big\langle[v\partial_s\Gamma_3^0(2^Ns)v]\big(e^{\pm i2^Ns(|\ast-y|- |y|)}\mathcal{R}_0^\pm(2^Ns)(\ast,y)\big)(\cdot),\big(e^{\mp i2^Ns(|x-\cdot|-|x|)}\mathcal{R}_0^\mp(2^Ns)(x,\cdot)\big)\Big\rangle\\
&\ +\Big\langle[v\Gamma_3^0(2^Ns)v]\partial_s\big(e^{\pm i2^Ns(|\ast-y|- |y|)}\mathcal{R}_0^\pm(2^Ns)(\ast,y)\big)(\cdot),\big(e^{\mp i2^Ns(|x-\cdot|-|x|)}\mathcal{R}_0^\mp(2^Ns)(x,\cdot)\big)\Big\rangle\\
&\ +\Big\langle[v\Gamma_3^0(2^Ns)v]\big(e^{\pm i2^Ns(|\ast-y|- |y|)}\mathcal{R}_0^\pm(2^Ns)(\ast,y)\big)(\cdot),\partial_s\big(e^{\mp i2^Ns(|x-\cdot|-|x|)}\mathcal{R}_0^\mp(2^Ns)(x,\cdot)\big)\Big\rangle,
\end{align*}
and $v(x)\lesssim(1+|x|)^{-\frac{13}{2}-},$ then by H\"{o}lder inequality, for any $s\in\text{supp}\ \varphi_0(s),$ we have
\begin{equation*}
\big|\partial_{s}^kE_{3,N}^{0,\pm}(2^Ns;x,y)\big|\lesssim 2^{3N}s^{3-k},\ \ \ \ k=0,1.
\end{equation*}
Using again Lemma \ref{lemma-low-energy-2} with $z=(x,y)$, and
\[\Psi(z)=|x|+|y|,\ \ \ \Phi^\pm(2^Ns;z)=2^{-3N}E_{3,N}^{0,\pm}(2^{N}s;x,y),\]
we have
\[\sup_{x,y\in\mathbf{R}}\big|K_{4,N}^{0,\pm}(t;x,y)\big|\lesssim 2^N(1+|t|\cdot 2^{4N})^{-\frac{1}{2}}.\]
\end{proof}

\subsubsection{\textbf{First Kind of Resonance}}~~If there is first kind of resonance at zero, using \eqref{eq-symmetric identity} and \eqref{eq-M-inverse-First Kind},  we have
\begin{equation}\label{eq-Rv-expansion-first}
\begin{split}
R_V^\pm(\mu^4) =& R_0^\pm(\mu^4)-R_0^\pm(\mu^4)v\Big(\mu^{-1}S_1A_{-1,1}^1S_1\Big)vR_0^\pm(\mu^4)-R_0^\pm(\mu^4)v\Big(S_0A_{01}^1Q+QA_{02}^1S_0\Big)vR_0^\pm(\mu^4)\\
&\ \  -R_0^\pm(\mu^4)v\Big(\mu QA_{11}^1Q\Big)vR_0^\pm(\mu^4)-R_0^\pm(\mu^4)v\Big(\mu S_0A_{12}^1+\mu A_{13}^1S_0\Big)vR_0^\pm(\mu^4)\\
&\ \  -R_0^\pm(\mu^4)v\Big(\mu^2 QA_{21}^1+\mu^2A_{22}^1Q\Big)vR_0^\pm(\mu^4)-R_0^\pm(\mu^4)v\Gamma_3^1(\mu)vR_0^\pm(\mu^4).
\end{split}
\end{equation}
By checking the analysis of \textbf{Regular Case}, we only need to establish the following two propositions.
\begin{proposition}\label{pro-low energy-first kind case}
For any $N\in\mathbf{Z}_-$ and $N\le N_0,$ we have
\begin{equation}\label{eq-low energy-first kind case-1}
\sup\limits_{x,y}\Big|\int_0^\infty e^{-it\mu^4}\mu^3\varphi_N(\mu)\Big[R_0^\pm(\mu^4)v\big(\mu^{-1}S_1A_{-1}^1S_1\big)vR_0^\pm(\mu^4)\Big](x,y)d\mu\Big|\lesssim 2^N(1+|t|\cdot 2^{4N})^{-\frac{1}{2}}.
\end{equation}
\end{proposition}
\begin{proof}
Since $S_1$ has the same orthogonality relationships with $S_0,$ thus the proof is the same as the proof of Proposition \ref{pro-low energy-regular case}.
\end{proof}
According to the orthogonality of $S_1$ and $Q,$ using Lemma \ref{lemma-low-energy}, it is easy to prove the following estimates.
\begin{proposition}\label{pro-low energy-first kind case-2}
For any $N\in\mathbf{Z}_-$ and $N\le N_0$, we have
\begin{equation}\label{eq-low energy-first kind case-2}
\sup\limits_{x,y}\Big|\int_0^\infty e^{-it\mu^4}\mu^3\varphi_N(\mu)\Big[R_0^\pm(\mu^4)v\big(S_0A_{01}^1Q\big)vR_0^\pm(\mu^4)\Big](x,y)d\mu\Big|\lesssim 2^N(1+|t|\cdot 2^{4N})^{-\frac{1}{2}}.
\end{equation}
\begin{equation}\label{eq-low energy-first kind case-3}
\sup\limits_{x,y}\Big|\int_0^\infty e^{-it\mu^4}\mu^3\varphi_N(\mu)\Big[R_0^\pm(\mu^4)v\big(QA_{01}^1S_0\big)vR_0^\pm(\mu^4)\Big](x,y)d\mu\Big|\lesssim 2^N(1+|t|\cdot 2^{4N})^{-\frac{1}{2}}.
\end{equation}
\end{proposition}

\subsubsection{\textbf{Second Kind of Resonance}}If there is second kind of resonance at zero,  using \eqref{eq-symmetric identity} and \eqref{eq-M-inverse-Second Kind}, we have
\begin{equation}\label{eq-Rv-expansion-second}
\begin{split}
& R_V^\pm(\mu^4)= \\
& R_0^\pm(\mu^4)-R_0^\pm(\mu^4)v\Big(\mu^{-3}S_2A_{-3,1}^2S_2\Big)vR_0^\pm(\mu^4)-R_0^\pm(\mu^4)v\mu^{-2}\Big( S_2A_{-2,1}^2S_0+S_0A_{-2,2}^2S_2\Big)vR_0^\pm(\mu^4)\\
&\ \ -R_0^\pm(\mu^4)v\Big(\mu^{-1}S_0A_{-1,1}^2S_0\Big)vR_0^\pm(\mu^4)-R_0^\pm(\mu^4)v\Big(\mu^{-1}S_2A_{-1,2}^2Q+\mu^{-1}QA_{-1,2}^2S_2\Big)vR_0^\pm(\mu^4)\\
&\ \ -R_0^\pm(\mu^4)v\Big(QA_{01}^2S_0+S_0A_{02}^2Q\Big)vR_0^\pm(\mu^4)-R_0^\pm(\mu^4)v\Big(S_2A_{03}^2+A_{04}^2S_2\Big)vR_0^\pm(\mu^4)\\
&\ \ -R_0^\pm(\mu^4)v(\mu QA_{11}^2Q)vR_0^\pm(\mu^4)-R_0^\pm(\mu^4)v(\mu S_0A_{12}^2+\mu A_{13}^2S_0)vR_0^\pm(\mu^4)\\
&\ \ -R_0^\pm(\mu^4)v(\mu^2 QA_{21}^2+\mu^2 A_{22}^2Q)vR_0^\pm(\mu^4)-R_0^\pm(\mu^4)v\Gamma_3^2(\mu)vR_0^\pm(\mu^4).
\end{split}
\end{equation}
By combing the analysis of \textbf{Regular Case} and \textbf{First Kind of Resonance}, it is suffices to establish the following estimates.
\begin{proposition}\label{pro-low energy-second kind case}
For any $N\in\mathbf{Z}_{-}$ and $N\le N_0,$ we have
\begin{equation}\label{eq-low energy-second kind case-1}
\sup\limits_{x,y}\Big|\int_0^\infty e^{-it\mu^4}\varphi_N(\mu)\mu^3\Big[R_0^\pm(\mu^4)v\big(\mu^{-3}S_2A_{-3,1}^2S_2\big)vR_0^\pm(\mu^4)\Big](x,y)d\mu\Big|\lesssim 2^N(1+|t|\cdot 2^{4N})^{-\frac{1}{2}},
\end{equation}
\begin{equation}\label{eq-low energy-second kind case-2}
\sup\limits_{x,y}\Big|\int_0^\infty e^{-it\mu^4}\varphi_N(\mu)\mu^3\Big[R_0^\pm(\mu^4)v\big(\mu^{-2}S_2A_{-2,1}^2S_0\big)vR_0^\pm(\mu^4)\Big](x,y)d\mu\Big|\lesssim 2^{N}(1+|t|\cdot 2^{4N})^{-\frac{1}{2}}.
\end{equation}
\begin{equation}\label{eq-low energy-second kind case-3}
\sup\limits_{x,y}\Big|\int_0^\infty e^{-it\mu^4}\varphi_N(\mu)\mu^3\Big[R_0^\pm(\mu^4)v\big(\mu^{-2}S_0A_{-2,2}^2S_2\big)vR_0^\pm(\mu^4)\Big](x,y)d\mu\Big|\lesssim 2^{N}(1+|t|\cdot 2^{4N})^{-\frac{1}{2}}.
\end{equation}
\end{proposition}
\begin{proof}We only prove \eqref{eq-low energy-second kind case-1} here since the other two estimates can be proved by the similar way. Note that
\begin{equation*}
\begin{split}
K_{1,N}^{2,\pm}(t;x,y):=&\int_0^\infty e^{-it\mu^4}\varphi_N(\mu)\mu^3\Big[R_0^\pm(\mu^4)v\big(\mu^{-3}S_2A_{-3,1}^2S_2\big)vR_0^\pm(\mu^4)\Big](x,y)d\mu.
\end{split}
\end{equation*}
It suffices to prove that
\[\sup_{x,y}\Big|K_{1,N}^{2,\pm}(t;x,y)\Big|\lesssim 2^N(1+|t|\cdot 2^{4N})^{-\frac{1}{2}}.\]
Noting that
\[R_0^\pm(x,y)=\frac{\pm ie^{\pm i\mu|x-y|}-e^{-\mu|x-y|}}{4\mu^3},\]
by Lemma \ref{lemma-low-energy}(iii) and  using the fact that $S_2(v)=S_2(xv(x))=S_2(x^2v(x))=0,$ we have
\begin{equation*}
\begin{split}
&  \Big[R_0^\pm(\mu^4) v\big(\mu^{-3}S_2A_{-3,1}^2S_2\big)vR_0^\pm(\mu^4)\Big](x,y)\\
=&\frac{1}{16\mu^9}\int_{\mathbf{R}^2}\Big(\pm ie^{\pm i\mu|x-y_2|}-e^{-\mu|x-y_2|}+\frac{1\pm i}{2}\mu^2|x-y_2|^2\Big)[vS_2A_{-3,1}^2S_2v](y_2,y_1)\\
&\ \ \ \ \ \ \ \ \ \ \ \ \ \ \ \ \ \ \ \ \ \ \ \ \ \ \ \ \  \Big(\pm ie^{\pm i\mu|y-y_1|}-e^{-\mu|y-y_1|}+\frac{1\pm i}{2}\mu^2|y-y_1|^2\Big)dy_1dy_2\\
:=&\frac{1}{16\mu^9}\int_{\mathbf{R}^2}F^\pm(\mu|x-y_2|)[vS_2A_{-3,1}^2S_2v](y_2,y_1)F^\pm(\mu|y-y_1|)dy_1dy_2,
\end{split}
\end{equation*}
where $\displaystyle F^\pm(p)=\pm ie^{\pm i p}-e^{-p}+\frac{1\pm i}{2}p^2.$

It is obvious that $F^{\pm}(p)\in C^\infty(\mathbf{R})$ and $\big(F^{\pm}\big)^{(k)}(0)=0\ (k=1,2),$ then by \eqref{eq-low energy-3}
and using the fact that $S_2(v)=S_2(xv(x))=S_2(x^2v(x))=0,$ we obtain that
\begin{align*}
&  \Big[R_0^\pm(\mu^4) v\big(\mu^{-3,1}S_2A_{-3}^2S_2\big)vR_0^\pm(\mu^4)\Big](x,y)\\
=&\frac{1}{16\mu^9}\int_{\mathbf{R}^2}F^\pm(\mu|x-y_2|)[vS_2A_{-3,1}^2S_2v](y_2,y_1)F^\pm(\mu|y-y_1|)dy_1dy_2\\
=&\frac{1}{16(2!)^2\mu^3}\int_{\mathbf{R}^2}\Big(\int_0^1\int_0^1(1-\theta_1)^{2}(1-\theta_2)^{2}(F^\pm)^{(3)}(\mu|x-\theta_2y_2|)\big({\rm sgn}(x-\theta_2y_2)\big)^3
\\
&\ \ \ \ \ \ \ \ \ \ \ \ \ \ \
(F^\pm)^{(3)}(\mu|y-\theta_1y_1|)\big({\rm sgn}(y-\theta_1y_1)\big)^3d\theta_1d\theta_2\Big)y_1^3y_2^3[vS_2A_{-3,1}^2S_2v](y_2,y_1)dy_1dy_2.
\end{align*}
Then we obtain that
\begin{equation}\label{eq-K_1^2-1}
\begin{split}
&\ \ K_{1,N}^{2,\pm}(t;x,y)\\
=&\frac{1}{16(2!)^2}\int_{\mathbf{R}^2}\Bigg[\int_0^1\int_0^1(1-\theta_1)^{2}(1-\theta_2)^{2}\big({\rm sgn}(y-\theta_1y_1)\big)^3\big({\rm  sgn}(x-\theta_2y_2)\big)^3
\Big(\int_0^\infty e^{-it\mu^4}\varphi_0(2^{-N}\mu)\\
&\ \ \ \ \ \ \ \ \ \  \ \ \ \ \
(F^\pm)^{(3)}(\mu|x-\theta_2y_2|)(F^\pm)^{(3)}(\mu|y-\theta_1y_1|)d\mu\Big)d\theta_1d\theta_2\Bigg]y_1^3y_2^3[vS_2A_{-3,1}^2S_2v](y_2,y_1)dy_1dy_2.\\
\end{split}
\end{equation}
Denote
\begin{align*}
\ E_{1,N}^{2,\pm}(t;x,y,\theta_1,\theta_2,y_1,y_2)=\int_0^\infty e^{-it\mu^4}\varphi_0(2^{-N}\mu)
(F^\pm)^{(3)}(\mu|x-\theta_2y_2|)(F^\pm)^{(3)}(\mu|y-\theta_1y_1|)d\mu.
\end{align*}
Thus
\[\Big|K_{1,N}^{2,\pm}(t;x,y)\Big|\lesssim\int_{\mathbf{R}^2}\Big(\int_0^1\int_0^1
\Big|E_{1,N}^{2,\pm}(t;x,y,\theta_1,\theta_2,y_1,y_2)\Big|d\theta_1d\theta_2\Big)\big|[vS_2A_{-3,1}^2S_2v](y_2,y_1)\big||y_1^3y_2^3|dy_1dy_2.\]
Now we consider $E_{1,N}^{2,\pm}(t;x,y,\theta_1,\theta_2,y_1,y_2)$.  Since
\[(F^{\pm})^{(3)}(\mu|x-\theta_2y_2|)=e^{\pm i\mu|x-\theta_2 y_2|}+e^{-\mu|x-\theta_2 y_2|}:=e^{\pm i\mu|x-\theta_2 y_2|}\mathcal{F}^{\pm}(\mu|x-\theta_2 y_2|),\]
\[(F^{\pm})^{(3)}(\mu|y-\theta_1y_1|)=e^{\pm i\mu|y-\theta_1 y_1|}+e^{-\mu|y-\theta_1 y_1|}:=e^{\pm i\mu|y-\theta_1 y_1|}\mathcal{F}^{\pm}(\mu|y-\theta_1 y_1|),\]
let $\mu=2^Ns,$ we can rewrite $E_{1,N}^{2,\pm}(t;x,y,\theta_1,\theta_2,y_1,y_2)$ as follows:
\begin{equation*}
\begin{split}
E_{1,N}^{2,\pm}(t;x,y,\theta_1,\theta_2,y_1,y_2)=&2^N\int_0^\infty e^{-it2^{4N}s^4}\varphi_0(s)e^{\pm i2^Ns|x-\theta_2y_2|}\mathcal{F}^\pm(2^Ns|x-\theta_2y_2|)\\
&\ \ \ \ \ \ \ \ \ \ \ \ \ \ \ \ \ \ \ \ \ \ \ \ \ \ \ \ \ \ \ \ \  e^{\pm i2^Ns|y-\theta_1y_1|}\mathcal{F}^\pm(2^Ns|y-\theta_1y_1|)ds.
\end{split}
\end{equation*}
Noting that for any $s\in\text{supp}\ \varphi_0(s)$ and $k=0,1,$ we have
\[\Big|\partial_{s}^k\mathcal{F}^\pm(2^Ns|x-\theta_2y_2|)\Big|\lesssim 1,\ \ \ \ \ \ \Big|\partial_{s}^k\mathcal{F}^\pm(2^Ns|y-\theta_1y_1|)\Big|\lesssim 1,\]
then using Lemma \ref{lemma-low-energy-2} with $z=(x,y,y_1,y_2,\theta_1,\theta_2)$ and
\[\Psi(z)=|x-\theta_2 y_2|+|y-\theta_1 y_1|,\ \ \ \ \Phi^\pm(2^Ns;z)=\mathcal{F}^\pm(2^Ns|x-\theta_2y_2|)\cdot\mathcal{F}^\pm(2^Ns|y-\theta_1y_1|),\]
we obtain that
\[\sup_{x,y\in\mathbf{R}}\big|E_{1,N}^{2,\pm}(t;x,y,\theta_1,\theta_2,y_1,y_2)\big|\lesssim 2^{N}(1+|t|\cdot 2^{4N})^{-\frac{1}{2}}.\]
Then, since $v(x)\lesssim(1+|x|)^{-\frac{25}{2}-},$ we have
\begin{align*}
\big|K_{1,N}^{2,\pm}(t;x,y)\big|\lesssim & 2^{N}(1+|t|\cdot 2^{4N})^{-\frac{1}{2}}\int_{\mathbf{R}^2}|y_2^3 v(y_2)|\big|[S_2A_{-3,1}^2S_2](y_2,y_1)\big||y_1^3v(y_1)|dy_1dy_2\\
\lesssim& 2^{N}(1+|t|\cdot 2^{N})^{-\frac{1}{2}}\|y_2^3v(y_2)\|_{L^2_{y_2}}\|S_2\ {-3}^2S_2\|_{L^2\rightarrow L_2}\ \|y_1^3v(y_1)\|_{L^2_{y_1}}\\
\lesssim& 2^{N}(1+|t|\cdot 2^{N})^{-\frac{1}{2}}
\end{align*}
uniformly in $x,y.$
\end{proof}

\subsection{High Energy Dispersive Estimates with $\alpha=0$}
In this subsection, we will give the proof of the dispersive bound \eqref{eq-high result-1} with $\alpha=0$ for the large energy part. To complete the proof,  we use the formula
\begin{align}\label{higher energy Stone formula}
e^{-itH}P_{ac}(H)\widetilde{\chi}(H)f(x)=\sum^{+\infty}_{N=N_0+1}\sum_{\pm}\frac{\pm2}{\pi i}\int_\mathbf{R}\Big(\int_0^\infty e^{-it\mu^4}\varphi_N(\mu)\mu^3R_V^{\pm}(\mu^4)(x,y)d\mu\Big)f(y)dy,
\end{align}
and the following resolvent identity,
\begin{equation}\label{eq-resolvent identity-high}
R_V^\pm(\mu^4)=R_0^\pm(\mu^4)-R_0^\pm(\mu^4)VR_0^\pm(\mu^4)+R_0^\pm(\mu^4)VR_V^\pm(\mu^4)VR_0^\pm(\mu^4).
\end{equation}
%
By combing Proposition \ref{prop-free estimates} and Lemma \ref{lemma-low sum}, it suffices to establish the following bound is valid for the last two terms in \eqref{eq-resolvent identity-high}.
\begin{proposition}\label{prop-high energy perturbated-1}
Assume that $|V(x)|\lesssim(1+|x|)^{-1-},$ then for any $N>N_0,$ we have
\begin{equation}\label{eq-high energy perturbated-2}
\sup\limits_{x,y\in\mathbf{R}}\Big|\int_0^\infty e^{-it\mu^4}\mu^3\varphi_N(\mu)\Big[R_0^\pm(\mu^4)VR_0^\pm(\mu^4)\Big](x,y)d\mu\Big|\lesssim 2^{N}(1+|t|\cdot 2^{4N})^{-\frac{1}{2}}.
\end{equation}
\end{proposition}
\begin{proof}
Let
\begin{equation}\label{eq- high L_1}
\begin{split}
L_{1}^{\pm,N}(t;x,y):=&\int_0^\infty e^{-it\mu^4}\mu^3\varphi_N(\mu)\Big[R_0^\pm(\mu^4)VR_0^\pm(\mu^4)\Big](x,y)d\mu\\
\end{split}
\end{equation}
and
\begin{equation*}\label{eq-free kernel-1}
\begin{split}
R_0^\pm(\mu^4)(x,y)=&\frac{1}{4\mu^3}\Big(\pm ie^{\pm i\mu|x-y|}-e^{-\mu|x-y|}\Big)
:=\frac{e^{\pm i\mu|x-y|}}{4\mu^3}\mathcal{R}_0^\pm(\mu|x-y|).
\end{split}
\end{equation*}
Let $\mu=2^Ns$ we can rewrite $L_{1}^{\pm,N}(x,y)$ as follows:
\begin{equation*}\label{eq- high L_1-1}
\begin{split}
&\ L_1^{\pm,N}(t;x,y)\\
=&\frac{2^{-2N}}{16}\int_0^\infty e^{-it2^{4N}s^4}s^{-3}\varphi_0(s)\Big(\int_{\mathbf{R}}e^{\pm i2^Ns|x-y_1|}\mathcal{R}_0^\pm(2^Ns|x-y_1|)V(y_1)e^{\pm i2^Ns|y-y_1|}\mathcal{R}_0^\pm(2^Ns|y-y_1|)dy_1\Big)ds\\
=&\frac{2^{-2N}}{16}\int_{\mathbf{R}}\Big(\int_0^\infty e^{-it2^{4N}s^4}s^{-3}\varphi_0(s)e^{\pm i2^Ns|x-y_1|}e^{\pm i2^Ns|y-y_1|}\mathcal{R}_0^\pm(2^Ns|x-y_1|)V(y_1)\mathcal{R}_0^\pm(2^Ns|y-y_1|)ds\Big)dy_1.
\end{split}
\end{equation*}
Since for any $s\in\text{supp}\ \varphi_0(s)$ and $k=0,1,$ we have
\[\sup_{x,y\in\mathbf{R}}\Big|\partial_{s}^k\big(\mathcal{R}_0^\pm(2^Ns|x-y_1|)V(y_1)\mathcal{R}_0^\pm(2^Ns|y-y_1|)\big)\Big|\lesssim |V(y_1)|,\]
then using Lemma \ref{lemma-low-energy-2} with $z=(x,y,y_1)$ and
\[\Psi(z)=|x-y_1|+|y-y_1|,\ \ \ \ \Phi(2^Ns;z)=\mathcal{R}_0^\pm(2^Ns|x-y_1|)\cdot\mathcal{R}_0^\pm(2^Ns|y-y_1|),\]
 we can obtain that
\[\sup_{x,y\in\mathbf{R}}\big|L_{1}^{\pm,N}(t;x,y)\big|\lesssim 2^{-2N}(1+|t|\cdot 2^{4N})^{-\frac{1}{2}}\int_\mathbf{R}|V(y_1)|dy_1\lesssim 2^N(1+|t|\cdot 2^{4N})^{-\frac{1}{2}}.\]
\end{proof}

We next consider the last term $R_0^\pm(\mu^4)VR_V^\pm(\mu^4)VR_0^\pm(\mu^4)$ in \eqref{eq-resolvent identity-high}. To control this term, we utilize the following estimates.
\begin{lemma}\label{lemma-FSY}(\cite{FSY},\text{Theorem} 2.23) Let $k\ge 0$ and   $|V(x)|\lesssim \big(1+|x|\big)^{-k-1-}$ such that $H=(-\Delta)^2+V$ has no embedded positive eigenvalues.  Then for any $\sigma >k+\frac{1}{2},\  R^\pm_V(\lambda)\in \mathcal{B}\big(L_\sigma^2(\mathbf{R}),L_{-\sigma}^2(\mathbf{R})\big)$ are $C^k$-continuous for all $\lambda>0$.  Furthermore,
\[\Big\|\partial^k_\lambda R^\pm_V(\lambda)\Big\|_{L_\sigma^2(\mathbf{R})\rightarrow L_{-\sigma}^2(\mathbf{R})}=O\big( |\lambda|^{-\frac{3(k+1)}{4}}\big),\]
 as $\lambda\rightarrow +\infty.$ \end{lemma}

\begin{proposition}\label{prop-high energy perturbated-3}
Assume that $|V(x)|\lesssim(1+|x|)^{-2-},$ then for any $N>N_0,$ we have
\begin{equation}\label{eq-high energy perturbated-3}
\sup\limits_{x,y}\Big|\int_0^\infty e^{-it\mu^4}\mu^3\varphi_N(\mu)\Big[R_0^\pm(\mu^4)VR_V^\pm(\mu^4)VR_0^\pm(\mu^4)\Big](x,y)d\mu\Big|\lesssim 2^{N}(1+|t|\cdot 2^{4N})^{-\frac{1}{2}}.
\end{equation}
\end{proposition}

\begin{proof}
For the \eqref{eq-high energy perturbated-3}, it suffices to prove that for any $f,g\in L^1(\mathbf{R}),$
\begin{equation}\label{eq-high L_2-1}
\Big|\int_0^\infty e^{-it\mu^4}\mu^3\varphi_0(2^{-N}\mu)\big\langle VR_V^\pm(\mu^4) VR_0^\pm(\mu^4)f,\big(R_0^\pm(\mu^4)\big)^* g\big\rangle d\mu\Big|\lesssim 2^{N}(1+|t|\cdot 2^{4N})^{-\frac{1}{2}}\|f\|_{L^1}\cdot\|g\|_{L^1}.
\end{equation}
Equivalently,  it is enough to  bound the following kernel uniformly by  $ O(2^N(1+|t|\cdot 2^{4N})^{-\frac{1}{2}})$ :
\[\widetilde{L}_{2}^{\pm,N}(t;x,y)=\int_0^\infty e^{-it\mu^4}\mu^3\varphi_0(2^{-N}\mu)\Big\langle VR_V^\pm(\mu^4) V\big(R_0^\pm(\mu^4)(\ast,y)\big)(\cdot),(R_0^\pm(\mu^4))^*(x,\cdot)\Big\rangle d\mu,\]
where
\begin{align*}
R_0^\pm(\mu^4)(x,y)=&\frac{1}{4\mu^3}\big(\pm ie^{\pm i\mu|x-y|}-e^{-\mu|x-y|}\big)
:=\frac{e^{\pm i\mu|x-y|}}{4\mu^3}\mathcal{R}_0^\pm(\mu)(x,y).
\end{align*}
Rewrite that
\begin{align*}
&\ \ \Big\langle VR_V^\pm(\mu^4) V\big(R_0^\pm(\mu^4)(\ast,y)\big)(\cdot),R_0^\mp(\mu^4)(x,\cdot)\Big\rangle\\
=&\frac{1}{16\mu^6}\Big\langle VR_V^\pm V\big(e^{\pm i\mu|\ast-y|}\mathcal{R}_0^\pm(\mu)(\ast,y)\big)(\cdot),\big(e^{\mp i\mu|x-\cdot|}\mathcal{R}_0^\mp(\mu)(x,\cdot)\big)\Big\rangle\\
=&\frac{1}{16\mu^6}e^{\pm i\mu|x|}e^{\pm i\mu|y|}\Big\langle VR_V^\pm V\big(e^{\pm i\mu(|\ast-y|-|y|)}\mathcal{R}_0^\pm(\mu)(\ast,y)\big)(\cdot),\big(e^{\mp i\mu(|x-\cdot|-|x|)}\mathcal{R}_0^\mp(\mu)(x,\cdot)\big)\Big\rangle\\
:=&\frac{1}{16\mu^6}e^{\pm i\mu(|x|+|y|)}E_{2,N}^{L,\pm}(\mu;x,y).
\end{align*}
Let $\mu=2^Ns,$ then we have
\begin{equation*}
\begin{split}
\widetilde{L}_{2}^{\pm,N}(t;x,y)=&\frac{2^{-2N}}{16}\int_0^\infty e^{-it2^{4N}s^4}e^{\pm i2^Ns|x|}e^{\pm i2^Ns|y|}s^{-3}\varphi_0(s)E_{2,N}^{L,\pm}(2^Ns;x,y)ds.
\end{split}
\end{equation*}
Noting that£¬ for any $s\in\text{supp}\ \varphi_0(s)$ and any $N>N_0$, we have
\[\Big|\partial_{s}^k\Big(e^{\pm i2^N s(\ast-y|- |y|)}\mathcal{R}_0^\pm(2^Ns)(\ast,y)\Big)\Big|\lesssim 2^{kN}\langle \ast\rangle^{k},\ \ k=0,1,\]
\[\Big|\partial_{s}^k\Big(e^{\pm i2^N s(|x-\cdot|- |x|)}\mathcal{R}_0^\pm(2^Ns)(x,\cdot)\Big)\Big|\lesssim 2^{kN}\langle \cdot\rangle^{k},\ \ k=0,1.\]
And by Lemma \ref{lemma-FSY}, for $\sigma>k+\frac{1}{2}$ with $k=0,1,$ we have
\[\big\|\partial_{s}^kR_V^\pm(2^{4N}s^4)\big\|_{L^2_{\sigma}\rightarrow L^2_{-\sigma}}\lesssim 2^{kN}(2^Ns)^{-3}. \]
Furthermore, since $|V(x)|\lesssim (1+|x|)^{-2-},$ then for $\sigma>k+\frac{1}{2}$ with $k=0,1,$ we can obtain that
\begin{align*}\big|\partial_{s}^kE_{2,N}^{L,\pm}(2^Ns;x,y)\big|\lesssim & \sum_{k=0}^1\|V(\cdot)\langle\cdot\rangle^{\sigma+1-k}\|_{L^2}^2\cdot
\big\|\partial_{s}^kR_V^\pm(2^{4N}s^4)\big\|_{L^2_{\sigma}\rightarrow L^2_{-\sigma}}\\
\lesssim&2^{kN}(2^Ns)^{-3}.
\end{align*}
Hence,  using Lemma \ref{lemma-low-energy-2} with $z=(x,y)$ and
\[\Psi(z)=|x|+|y|,\ \ \ \Phi(2^Ns;z)=E_{2,N}^{L,\pm}(2^Ns;x,y),\]
we have
\[\big|\widetilde{L}_{2}^{\pm,N}(t;x,y)\big|\lesssim 2^{(k-5)N}(1+|t|\cdot 2^{4N})^{-\frac{1}{2}}\lesssim 2^N(1+|t|\cdot 2^{4N})^{-\frac{1}{2}}\]
uniform in $x,y$ for any $N>N_0$.
\end{proof}

\subsection{Proofs of the decay estimates with regular term $H^{\frac{\alpha}{4}}$ for $0<\alpha<1$}
\vskip0.2cm
In the final subsection, we will give the proof of the estimates \eqref{eq-regularity estimate-low} and \eqref{eq-regularity estimate-high}. By using functional calculus and the Stone's formula, we have
\begin{equation}\label{eq-regularity stone}
\begin{split}
H^\frac{\alpha}{4} e^{-itH}P_{ac}(H)f(x)=&\frac{2}{\pi i}\int_0^\infty e^{-it\lambda}\lambda^\frac{\alpha}{4}[R_V^+(\lambda)-R_V^-(\lambda)]fd\lambda\\
=&\int_{\mathbf{R}}\Big(\frac{2}{\pi i}\int_0^\infty e^{-it\mu^4}\mu^{3+\alpha}[R_V^+(\mu^4)-R_V^-(\mu^4)](x,y)d\mu\Big) f(y)dy.
\end{split}
\end{equation}
Similar to the proof of Theorem \ref{thm-main results} for $\alpha=0$ , by using dyadic partition of unity, we can obtain for any $N\in\mathbf{Z}$ and $\alpha\ge0$,
\[\sup_{x,y\in\mathbf{R}}\Big|\int_0^\infty e^{-it\mu^4}\mu^{3+\alpha}\varphi_N(\mu)[R_V^+(\mu^4)-R_V^-(\mu^4)](x,y)d\mu\Big|\lesssim 2^{(\alpha+1)N}(1+|t|\cdot 2^{4N})^{-\frac{1}{2}}.\]
Then, by Lemma \ref{lemma-low sum}, for any $0< \alpha<1,$ we have
\begin{align*}
&\ \sup_{x,y\in\mathbf{R}}\Big|\int_0^\infty e^{-it\mu^4}\mu^{3+\alpha}[R_V^+(\mu^4)-R_V^-(\mu^4)](x,y)d\mu\Big|\\
\lesssim
&\sum_{N=-\infty}^{+\infty}\sup_{x,y\in\mathbf{R}}\Big|\int_0^\infty e^{-it\mu^4}\varphi_N(\mu)\mu^{3+\alpha}[R_V^+(\mu^4)-R_V^-(\mu^4)](x,y)d\mu\Big|\lesssim |t|^{-\frac{1+\alpha}{4}}.
\end{align*}
Hence, for any $0<\alpha<1,$ we have
\[\big\|H^{\frac{\alpha}{4}}e^{-itH}P_{ac}(H)\big\|_{L^1(\mathbf{R})\rightarrow L^\infty(\mathbf{R})}\lesssim |t|^{-\frac{1+\alpha}{4}},\ t\neq0.\]

\section{Proof of Theorem \ref{thm-M-inverse}}\label{section-M}
In this section, we show the processes  how to derive the asymptotic expansions of $\Big(M^\pm(\mu)\Big)^{-1}.$
Recall that the integral kernel of $R_0^\pm(\mu^4)$ are given by
\[R_0^\pm(\mu^4)(x,y)=\frac{1}{4\mu^3}\Big(\pm ie^{\pm i\mu|x-y|}-e^{-\mu|x-y|}\Big).\]
When $\mu|x-y|<1,$ we have the following representation for the $R_0^\pm(\mu^4)(x,y):$
\begin{equation}\label{eq-expansion for free resolvent-1}
\begin{split}
	  R_0^\pm(\mu^4)(x,y)=& \frac{a^\pm}{\mu^3}I(x,y)+\frac{a_{-1}^\pm}{\mu}G_{-1}(x,y) + G_0(x,y)
                 + a_1^\pm\mu G_1(x,y)+a_3^\pm\mu^3 G_3(x,y)\\
              &+\mu^4G_4(x,y)+\sum_{k=5}^Na_{k}^\pm\mu^{k}G_{k}(x,y)+O(\mu^{N+1}|x-y|^{N+4}).
		\end{split}
\end{equation}
Here $G_k(x,y)$ are integral kernels which defined as follows:
\begin{equation}\label{eq-kernel}
\begin{split}
G_{-1}(x,y)=&|x-y|^2,\ \ G_0(x,y)=\frac{|x-y|^3}{12},\ \ G_1(x,y)=|x-y|^4,\\
G_3(x,y)=&|x-y|^6,\ \ G_4(x,y)=\frac{|x-y|^7}{2\times(7!)},\ \ G_k(x,y)=|x-y|^{3+k}\ (k\ge 5),
\end{split}
\end{equation}
and the coefficients $\displaystyle a^\pm=\frac{-1\pm i}{4}$,$\displaystyle a_k^\pm=\frac{(\pm i)^{k+4}+(-1)^{k+4}} {4\times(k+3)!}.$ Notice that, $\displaystyle G_0(x)=\frac{|x|^3}{12}$ is the fundamental solution of $\Delta^2$ in $\mathbf{R}$. In fact,
when $\mu|x-y|\gtrsim1,$ the expansions remains valid. It is obvious that the first two terms in \eqref{eq-expansion for free resolvent-1} are singular, which is different from at most one singular term on $\mu$ for $n\ge3$, see e.g. \cite{Erdogan-Green-Toprak,FWY,Green-Toprak}.

For convenience, we introduce two notations here. For operators $\varepsilon(\mu)$ with some parameter $\mu$, we write $\varepsilon(\mu)=O_1(\mu^{-\alpha}g(x,y)))$ if it's kernel $\varepsilon(\mu)(x,y)$ has the property
  \begin{equation}\label{eq-notation-1}
  \mu^\alpha|\varepsilon(\mu)(x,y)|+\mu^{\alpha+1}|\partial_\mu\varepsilon(\mu)(x,y)|\lesssim g(x,y).
  \end{equation}
And we denote $\Gamma_k(\mu)$ as a $\mu$ dependent operator which satisfies
  \begin{equation}\label{eq-notation-2}
  \big\||\Gamma_k(\mu)|\big\|_{L^2\rightarrow L^2}+\mu\big\|\partial_\mu|\Gamma_k(\mu)|\big\|_{L^2\rightarrow L^2}\lesssim \mu^k,\ \ \mu>0,
  \end{equation}
Noting that the operator $\Gamma_k(\mu)$ will vary from line to line.

Let $T_0:= U+vG_0v$ and $P=\|V\|_{L^1(\mathbf{R})}^{-1}v\langle v,\cdot \rangle$ denote the orthogonal projection onto the span of $v.$ Since $M^\pm(\mu)=U+vR_0^\pm(\mu^4)v$ ( see Theorem \ref{thm-M-inverse} ), then we obtain the following expansions of $M^\pm(\mu).$
\begin{lemma}\label{lemma-M expansion}
Let $ \widetilde{a}^\pm= a^\pm\|V\|_{L^1(\mathbf{R})}$. Assume that $|V(x)|\lesssim(1+|x|)^{-\beta}$ with some $\beta>0$, the following expansions for $M^\pm(\mu)$ in $L^2(\mathbf{R})$ hold:

 (i)~If $\beta>13,$ then  we have
\begin{equation}\label{eq-M-expansions-1}
   \begin{split}
M^\pm(\mu) = & \frac{\widetilde{a}^\pm}{\mu^3}P+\frac{a_{-1}^\pm}{\mu}vG_{-1}v+T_0+a_1^\pm \mu vG_1v+\Gamma_{3}(\mu).
   \end{split}
\end{equation}

(ii)~If $\beta>17,$ then we have
\begin{equation}\label{eq-M-expansions-1-2}
   \begin{split}
M^\pm(\mu) = & \frac{\widetilde{a}^\pm}{\mu^3}P+\frac{a_{-1}^\pm}{\mu}vG_{-1}v+T_0+a_1^\pm \mu vG_1v+a_3^\pm \mu^3vG_3v+\mu^4vG_4v+\Gamma_{5}(\mu).
   \end{split}
\end{equation}

(iii)~If $\beta>25,$  we have
\begin{equation}\label{eq-M-expansions-1-3}
   \begin{split}
M^\pm(\mu) = & \frac{\widetilde{a}^\pm}{\mu^3}P+\frac{a_{-1}^\pm}{\mu}vG_{-1}v+T_0+a_1^\pm \mu vG_1v+a_3^\pm \mu^3vG_3v +\mu^4vG_4v\\
& +\sum_{k=5}^{8}a_{k}^\pm\mu^k vG_kv+\Gamma_{9}(\mu).
   \end{split}
\end{equation}
\end{lemma}
\begin{proof}
We only give a proof of \eqref{eq-M-expansions-1}, the other cases are similar. If $\mu|x-y|<1,$ using \eqref{eq-expansion for free resolvent-1}, we have
\begin{equation}\label{eq-expansion for free resolvent-2}
R_0^\pm(\mu^4)(x,y)=\frac{a^\pm}{\mu^3}+\frac{a_{-1}^\pm}{\mu}G_{-1}+G_0+a_1^\pm\mu G_1+O_1(\mu^{3} |x-y|^{6}).
\end{equation}
If $\mu|x-y|>1,$ using \eqref{eq-free resolvent-LAP-3}, we have
\begin{equation}\label{eq-expansion for free resolvent-3}
\begin{split}
&\ R_0^\pm(\mu^4)(x,y)\\
=&\frac{a^\pm}{\mu^3}+\frac{a_{-1}^\pm}{\mu}G_{-1}+G_0+a_1^\pm\mu G_1+\Big(\frac{\pm ie^{\pm i\mu|x-y|}-e^{-\mu|x-y|}}{4\mu^3}-\frac{a^\pm}{\mu^3}-\frac{a_{-1}^\pm}{\mu}G_{-1}-G_0-a_1^\pm\mu G_1\Big)\\
=&\frac{a^\pm}{\mu^3}+\frac{a_{-1}^\pm}{\mu}G_{-1}+G_0+a_1^\pm\mu G_1+O_1(\mu^{3} |x-y|^{6}).
\end{split}
\end{equation}
Using \eqref{eq-expansion for free resolvent-2} and \eqref{eq-expansion for free resolvent-3} in \eqref{eq-M-expansions-1}, we have
\[\Big|\Big(M^\pm(\mu)-\frac{\widetilde{a}^\pm}{\mu^3}P-\frac{a_{-1}^\pm}{\mu}vG_{-1}v-T_0-a_1^\pm\mu vG_1v\Big)\Big|\lesssim v(x)v(y)|x-y|^{6}\mu^{3},\]
and
\[\Big|\partial_\mu\Big(M^\pm(\mu)-\frac{\widetilde{a}^\pm}{\mu^3}P-\frac{a_{-1}^\pm}{\mu}vG_{-1}v-T_0-a_1^\pm\mu vG_1v\Big)\Big|\lesssim v(x)v(y)|x-y|^{6}\mu^{2}.\]
Since the error term is a Hilbert-Schmidt operator if $v(x)\lesssim \big(1+|x|\big)^{-\frac{13}{2}-},$ this implies \eqref{eq-M-expansions-1} .
\end{proof}

\begin{remark}\label{remark-M-expansion}
 If we just want to get a weak expansions of $\big(M^\pm(\mu)\big)^{-1}$ (see Remark \ref{remark-M inverse-expansions} of Introduction), then we only need the following expansions of $M^\pm(\mu)$ in the following \textbf{Table 2}.
\begin{table}[h]\label{table-M-expansions}
\renewcommand\arraystretch{2}
\centering
\begin{tabular}{|p{2cm}<{\centering}|p{14cm}<{\centering}|}
  \specialrule{0.03em}{0pt}{0pt}
  \makecell{$\beta>7+2l$}      &  \makecell{$\displaystyle M^\pm(\mu) =\frac{\widetilde{a}^\pm}{\mu^3}P+\frac{a_{-1}^\pm}{\mu}vG_{-1}v+T_0+\Gamma_l(\mu),\ l\in[0,1]$} \\
  \specialrule{0.03em}{0pt}{0pt}
 \makecell{$\beta>9+2l$}     &  \makecell{$\displaystyle M^\pm(\mu) =\frac{\widetilde{a}^\pm}{\mu^3}P+\frac{a_{-1}^\pm}{\mu}vG_{-1}v+T_0+a_1^\pm \mu vG_1v+\Gamma_{1+l}(\mu),\ l\in[0,2]$   }\\
  \specialrule{0.03em}{0pt}{0pt}
  \makecell{$\beta>13+2l$}      &  \makecell{$\displaystyle M^\pm(\mu) =\frac{\widetilde{a}^\pm}{\mu^3}P+\frac{a_{-1}^\pm}{\mu}vG_{-1}v+T_0+a_1^\pm \mu vG_1v+a_3^\pm \mu^3 vG_3v+\Gamma_{3+l}(\mu),\ l\in[0,1]$      } \\
  \specialrule{0.03em}{0pt}{0pt}
\end{tabular}
\vspace{0.2cm}
\caption{$M^\pm(\mu)$ expansions}
\end{table}
\end{remark}

Now, we give the following lemma  used repeatedly to determine the inverse of $\Big(M^\pm(\mu)\Big)^{-1}$.
\begin{lemma}\label{Feshbach-formula}(\cite[Lemma 2.1]{JN})
Let $A$ be a closed operator and $S$ be a projection. Suppose $A+S$ has a bounded inverse. Then $A$ has
a bounded inverse if and only if
\begin{equation}		
B:= S-S(A+S)^{-1}S
\end{equation}
has a bounded inverse in $S\mathcal{H}$, and in this case
\begin{equation}		
A^{-1}= (A+S)^{-1} + (A+S)^{-1}S B^{-1} S(A+S)^{-1}.
\end{equation}	
\end{lemma}

Recall from  \eqref{eq-M-expansions-1}, let $ \displaystyle M^\pm(\mu)= \frac{\widetilde{a}^\pm}{\mu^3} \widetilde{M^\pm}(\mu)$, we have
\begin{equation}\label{eq-M-expansions-2}
   \begin{split}
\widetilde{M^\pm}(\mu) = & P+\frac{a_{-1}^\pm}{\widetilde{a}^\pm}\mu^2 vG_{-1}v+ \frac{\mu^3}{\widetilde{a}^\pm}T_0+
 \frac{a_1^\pm}{\widetilde{a}^\pm}\mu^4 v G_1 v
   +\frac{ a_3^\pm}{\widetilde{a}^\pm}\mu^6 v G_3v+\frac{\mu^7}{\widetilde{a}^\pm} v G_4v\\
   &+\sum_{k=5}^{10}\frac{a_{k}^\pm}{\widetilde{a}^\pm} \mu^{k+3}vG_kv+\Gamma_{14}(\mu).
   \end{split}
\end{equation}
Thus it suffices to derive the asymptotic expansion of $\Big(\widetilde{M^\pm}(\mu)\Big)^{-1}$ as $\mu\rightarrow 0.$~Denote $Q=I-P,$
it is obvious that
\begin{equation}\label{eq-M-expansions-2}
   \begin{split}
\widetilde{M^\pm}(\mu)+Q= &I+\frac{a_{-1}^\pm}{\widetilde{a}^\pm}\mu^2 vG_{-1}v+ \frac{\mu^3}{\widetilde{a}^\pm}T_0+
 \frac{a_1^\pm}{\widetilde{a}^\pm}\mu^4 v G_1 v
   +\frac{ a_3^\pm}{\widetilde{a}^\pm}\mu^6 v G_3v+\frac{\mu^7}{\widetilde{a}^\pm} v G_4v\\
   &+\sum_{k=5}^{10}\frac{a_{k}^\pm}{\widetilde{a}^\pm} \mu^{k+3}vG_kv+\Gamma_{14}(\mu).
   \end{split}
\end{equation}
Since~$vG_{-1}v$~is a bounded operator from $L^2(\mathbf{R})$ to $L^2(\mathbf{R}),$ then using Neumann series, we can obtain that
\begin{equation}\label{eq-M-expansions-3}
   \begin{split}
\Big(\widetilde{M^\pm}(\mu)+Q\Big)^{-1} =& I-\sum_{k=1}^{12}\mu^{k+1}B_{k}^\pm+\Gamma_{14}(\mu),
   \end{split}
\end{equation}
where ~$B_k^\pm\ (1\le k\le12)$~are bounded operators in $L^2(\mathbf{R})$ and it is easy to know that
\begin{equation*}\label{defi-B}
\begin{split}
B_1^\pm=& \frac{a_{-1}^\pm}{\widetilde{a}^\pm} vG_{-1}v,\ \ \ \ \ \  B_2^\pm= \frac{1}{\widetilde{a}^\pm}T_0,\ \ \ \ \ \ B_3^\pm= \frac{a_1^\pm}{\widetilde{a}^\pm}vG_1v-\frac{(a_{-1}^\pm)^2}{(\widetilde{a}^\pm)^2}(vG_{-1}v)^2,\\
B_4^\pm=& -\frac{a_{-1}^\pm}{(\widetilde{a}^\pm)^2}\big( vG_{-1}vT_0 + T_0vG_{-1}v \big),\\
B_5^\pm=& \frac{a_3^\pm}{\widetilde{a}^\pm}vG_3v + \frac{(a_{-1}^\pm)^3}{(\widetilde{a}^\pm)^3}(vG_{-1}v)^3-\frac{1}{(\widetilde{a}^\pm)^2}T_0^2
- \frac{a_{-1}^\pm a_1^\pm}{(\widetilde{a}^\pm)^2}\big( vG_{-1}v \cdot vG_1v + vG_1v\cdot vG_{-1}v\big),\\
B_6^\pm=& \frac{1}{\widetilde{a}^\pm}vG_4v+ \frac{(a^\pm_{-1})^2}{(\widetilde{a}^\pm)^3}\Big((vG_{-1}v)^2T_0 +
vG_{-1}v T_0 vG_{-1}v+T_0(vG_{-1}v)^2  \Big)- \frac{a_1^\pm}{(\widetilde{a}^\pm)^2}\big( vG_1v T_0 +T_0 vG_1v\big).
\end{split}
\end{equation*}

For the convenience, we denote some inverse operators as follows:
\begin{equation}\label{eq-inverse operator}
\begin{split}
&D_0=(QvG_{-1}vQ+S_0)^{-1}\ \ \text{on}\ \ QL^2(\mathbf{R});\ \ D_1=(S_0T_0S_0+S_1)^{-1}\ \ \text{on}\ \ S_0L^2(\mathbf{R});\\
&D_2=(T_1+S_2)^{-1}\ \ \text{on}\ \ S_1L^2(\mathbf{R}); \ \ \ \ \ \ \ \ \ \ \  D_3=(T_2+S_3)^{-1}\ \ \text{on}\ \  S_2L^2(\mathbf{R}).
\end{split}
\end{equation}
Next, we list the orthogonality relations of various operators and projections which we need in the proof of Theorem \ref{thm-M-inverse}.
\begin{align}\label{eq-orthogonality-1}
& QD_0=D_0Q=D_0;\\ \label{eq-orthogonality-2}
& S_iD_j=D_jS_i=D_j,\  i<j;\ \ S_iD_j=D_jS_i=S_i,\ i\ge j;\ \ D_iD_j=D_jD_i=D_j, i<j;\\ \label{eq-orthogonality-3}
& QB_1^\pm S_0=S_0B_1^\pm Q=S_2B_1^\pm=B_1^\pm S_2=S_3B_1^\pm=B_1^\pm S_3=0;\\ \label{eq-orthogonality-4}
& S_0B_2^\pm S_1=S_1 B_2^\pm S_0=QB_2^\pm S_2=S_2 B_2^\pm Q=S_3B_2^\pm=B_2^\pm S_3=0;\\ \label{eq-orthogonality-5}
& S_0B_3^\pm S_2=S_2 B_3^\pm S_0=QB_3^\pm S_3=S_3 B_3^\pm Q=0;\\ \label{eq-orthogonality-6}
& S_3B_4^\pm=B_4^\pm S_3=S_2B_5^\pm S_3=S_3 B_5^\pm S_2=0.
\end{align}

Now, we turn to the proof of Theorem \ref{thm-M-inverse}. The following lemma which will be used in the proof is an important classification of the projection spaces $S_jL^2(\mathbf{R})\ (j=0,1,2,3)$.
\begin{lemma} \label{lemma-projection space} We have

(1)~$f\in S_0L^2(\mathbf{R})$  if and only if $f\in \ker QvG_{-1}vQ$ on $QL^2(\mathbf{R});$

(2)~$f\in S_1L^2(\mathbf{R})$ if and only if $f\in \ker S_0T_0S_0$ on $S_0L^2(\mathbf{R});$

(3)~$f\in S_2L^2(\mathbf{R})$ if and only if $f\in \ker T_1$ on $S_1L^2(\mathbf{R})$ with
\begin{equation}\label{eq-T_1}
T_1:=S_1vG_1vS_1+\frac{6}{\|V\|_{L^1}}S_1 vG_{-1}v P vG_{-1}vS_1
 -384S_1T_0D_0T_0S_1.
\end{equation}

(4)~$f\in S_3L^2(\mathbf{R})$ if and only if $f\in \ker T_2$ on $S_2L^2(\mathbf{R})$ with
\begin{equation}\label{eq-T_2}
\begin{split}
T_2:=&S_2vG_3vS_2-\frac{8\cdot(6!)}{\|V\|_{L^1}}S_2T_0^2S_2+\frac{5}{2}S_2vG_1vD_0vG_1vS_2
+\frac{48\cdot(6!)}{\|V\|_{L^1}^2}\Big(S_2T_0vG_{-1}v\\
&-\frac{\|V\|_{L^1}}{6}S_2vG_1vD_0T_0\Big)D_2\Big(v G_{-1}v T_0S_2-\frac{\|V\|_{L^1}}{6}T_0D_0vG_1vS_2\Big).
\end{split}
\end{equation}
\end{lemma}
\begin{proof}
(1) Let $f\in QL^2(\mathbf{R})$ and $f\in\ker QvG_{-1}vQ,$ that is, $ QvG_{-1}vQ f=0$. Then we have
$$0=\langle QvG_{-1}vQ f, f\rangle= \langle vG_{-1}vQ f, Qf\rangle = \langle vG_{-1}vf, f\rangle.$$
Notice that
$$ G_{-1}vf(x)= \int_{\mathbf{R}} |x-y|^2 v(y)f(y)dy ,$$
and $Q(v)=0,$ i.e. $\displaystyle\langle f,v\rangle=\int_\mathbf{R}v(x)f(x)dx=0.$  So we have
\begin{equation*}
	\begin{split}
0=\langle v G_{-1}vf, f\rangle= &\int_{\mathbf{R}}(vG_{-1}vf)(x) \overline{f(x)}\, dx\\
=& \int_{\mathbf{R}}\overline{v(x)f(x)}\Big(\int_{\mathbf{R}}|x-y|^2v(y)f(y)\,dy\Big)\,dx\\
=&\int_{\mathbf{R}}\overline{v(x)f(x)}\Big(\int_{\mathbf{R}}(x^2-2xy+y^2)v(y)f(y)\,dy\Big)\,dx\\
=&-2\Big|\int_\mathbf{R}xv(x)f(x)dx\Big|^2=-2\big|\langle f,xv\rangle\big|^2,
\end{split}
\end{equation*}
which implies that $\langle f,xv\rangle=0 $. Hence $f\in S_0L^2(\mathbf{R}).$

On the other hand, if $f\in S_0L^2(\mathbf{R}),$ that is,
\[\langle f,v\rangle=\int_\mathbf{R} v(x)f(x)dx=0,\ \ \ \langle f,xv\rangle=\int_\mathbf{R}xv(x)f(x)dx=0.\]
Since $S_0L^2(\mathbf{R})\subseteq QL^2(\mathbf{R}),$ thus we have
\[QvG_{-1}vQ f=Qv(x)\int_{\mathbf{R}}(x^2-2xy+y^2)v(y)f(y)dy=\langle f,x^2v\rangle Q(v(x))=0,\]
which means that $f\in\ker QvG_{-1}vQ.$

(2)~It is obvious by the definition of $S_1,$ see {\bf{Table 1}}.

(3)~First, we prove the \textbf{Claim:} Denote $\displaystyle v_1:=Q(xv(x))=xv(x)-\frac{\langle v,xv\rangle}{\|V\|_{L^1}} v(x),$ then we have
\begin{equation}\label{eq-D_0}
\langle D_0v_1,v_1\rangle=-\frac{1}{2}.
\end{equation}
In fact, using $Q(v(x))=0,S_0(v(x))=S_0(xv(x))=0$ and $\langle v ,v_1\rangle=0,$ we have
\begin{align*}
(QvG_{-1}vQ+S_0)v_1=&QvG_{-1}vQv_1\\
=&Qv(x)\int_{\mathbf{R}}(x^2-2xy+y^2)v(y)v_1(y)dy\\
=&-2Q\big(xv(x)\big)\int_\mathbf{R}yv(y)v_1(y)dy\\
=&-2v_1\langle v_1,v_1\rangle.
\end{align*}
Hence, we obtain that
\begin{equation}\label{eq-D_0-1}
D_0v_1=(QvG_{-1}vQ+S_0)^{-1}v_1=\displaystyle-\frac{v_1}{2\langle v_1,v_1\rangle},
\end{equation}
furthermore, we have $\langle D_0v_1,v_1\rangle=\displaystyle-\frac{1}{2}.$

Let $f\in S_1L^2(\mathbf{R}) $ and $f\in \ker T_1$.
Recall \eqref{eq-T_1},  we have
\begin{equation}\label{eq-kernel T_1}
\begin{split}
0=\langle T_1f,f\rangle=& \big\langle S_1vG_1vS_1f,f\big\rangle+\frac{6}{\|V\|_{L^1}}\big\langle S_1vG_{-1}vPvG_{-1}vS_1f,f
\big\rangle\\
&-384\big\langle S_1T_0D_0T_0S_1f,f\big\rangle.
\end{split}
\end{equation}
Firstly, for the first part of \eqref{eq-kernel T_1}, using $S_1(v(x))=S_1(xv(x))=0,$ we have
\begin{align*}S_1vG_1vS_1f=&S_1v(x)\int_{\mathbf{R}}\big(x^4-4x^3y-
      4xy^3 +6x^2y^2+y^4\big)v(y)f(y)dy\\
      =& 6\langle f,x^2v\rangle S_1(x^2v).
\end{align*}
Thus, we have
\begin{align}\label{eq-kernel T_1-1}
\big\langle S_1vG_1vS_1 f, f\big\rangle
=6\big|\langle f,x^2v\rangle\big|^2.
\end{align}
Next, for the second part of \eqref{eq-kernel T_1}, since $Pf=\displaystyle\frac{\langle f,v\rangle}{\|V\|_{L^1}}v$ and $S_1(v(x))=S_1(xv(x))=0,$   we have
\[PvG_{-1}vS_1f=Pv(x)\int_{\mathbf{R}^2}|x-y|^2v(y)f(y)dy=v(x)\int_{\mathbf{R}}x^2v(x)f(x)dx=\langle f,x^2v\rangle v(x),\]
and
\[S_1vG_{-1}v P vG_{-1}vS_1 f=\langle f,x^2v\rangle\|V\|_{L^1}S_1(x^2v(x)).\]
Thus
\begin{align}\label{eq-kernel T_1-2}
\big\langle S_1 vG_{-1}v P vG_{-1}vS_1 f, f\big\rangle
=&\|V\|_{L^1}\cdot\big|\langle f,x^2v\rangle\big|^2.
\end{align}
Finally,  for the third part of \eqref{eq-kernel T_1}, using $QD_0=D_0Q=D_0$ and by the definition of $v_1,$ we know that
~$v_1\in (Q-S_0)L^2(\mathbf{R}).$ Indeed, we have $(Q-S_0)L^2(\mathbf{R})=\text{span}\{v_1\},$ then we can obtain that
\[(Q-S_0)T_0f=\frac{\langle T_0f,v_1\rangle}{\langle v_1,v_1\rangle }.\]
Thus, using \eqref{eq-D_0-1} and $D_0P=PD_0=0,S_1T_0S_0=S_0T_0S_1=0$, for all $f\in S_1L^2(\mathbf{R}),$ we have
\begin{align*}
S_1T_0D_0T_0S_1f=&S_1T_0[P+(Q-S_0)+S_0]D_0[P+(Q-S_0)+S_0]T_0f\\
=&S_1T_0(Q-S_0)D_0(Q-S_0)T_0f\\
=&-\frac{\langle T_0f,v_1\rangle}{2\langle v_1,v_1\rangle ^2}S_1(T_0v_1).
\end{align*}
Thus,
\begin{equation}\label{eq-kernel T_1-3}
\langle S_1T_0D_0T_0S_1f, f\rangle= -\frac{\big|\langle T_0f,v_1\rangle\big|^2}{2\langle v_1,v_1\rangle ^2}.
\end{equation}
Hence, substituting \eqref{eq-kernel T_1-1}, \eqref{eq-kernel T_1-2} and \eqref{eq-kernel T_1-3} into \eqref{eq-kernel T_1}, we have
\[0=\big\langle T_1 f,f\big\rangle=12\big|\langle f,x^2v\rangle\big|^2+192\frac{\big|\langle T_0f,v_1\rangle\big|^2}{\langle v_1,v_1\rangle ^2},\]
which implies that
\[\langle f,x^2v\rangle=0,\ \ \ \langle T_0f,v_1\rangle=0.\]
Furthermore, since $f\in S_1L^2(\mathbf{R}),$ thus $S_0T_0f=0,$ combined with $\langle T_0f,v_1\rangle=0,$ we have $QT_0f=0.$  Hence we have $f\in S_2L^2(\mathbf{R}).$

On the other hand, if $f\in S_2L^2(\mathbf{R}),$ then it is easy to know that
\[T_1f=12\langle f,x^2v\rangle S_1(x^2v)+192\frac{\langle T_0f,v_1\rangle}{\langle v_1,v_1\rangle ^2}S_1(T_0v_1)=0,\]
which means that $f\in\ker T_1.$

(4)~Let $f\in S_2L^2(\mathbf{R})$ and $f\in\ker T_2.$ It follows that from \eqref{eq-T_2}
\begin{equation}\label{eq-kernel T_2}
\begin{split}
0=\langle T_2f,f\rangle=&\big\langle S_2vG_3vS_2f,f\big\rangle-\frac{8\cdot(6!)}{\|V\|_{L^1}}\big\langle S_2T_0^2S_2f,f\big\rangle+\frac{5}{2}\big\langle S_2vG_1vD_0vG_1vS_2f,f\big\rangle\\
+\frac{48\cdot(6!)}{\|V\|_{L^1}^2}&\Big\langle\Big(S_2T_0vG_{-1}v-\frac{\|V\|_{L^1}}{6}S_2vG_1vD_0T_0\Big)D_2\Big(v G_{-1}v T_0S_2-\frac{\|V\|_{L^1}}{6}T_0D_0vG_1vS_2\Big)f,f\Big\rangle.
\end{split}
\end{equation}
For the first part of \eqref{eq-kernel T_2}, using $S_2(v(x))=S_2(xv(x))=S_2(x^2v(x))=0$ and $QT_0f=0,$ we have
\[
S_2vG_3vS_2f=S_2v(x)\int_\mathbf{R}\big|x-y\big|^6v(y)f(y)dy=-20\langle f,x^3v \rangle S_2(x^3v(x)),
\]
thus
\begin{equation}\label{eq-kernel T_2-1}
\langle S_2vG_3vf,f\rangle=-20\big|\langle f,x^3v \rangle\big|^2.
\end{equation}
Secondly, for the second part of \eqref{eq-kernel T_2}, using $QT_0f=0,$ we have
\[S_2T_0^2S_2f=S_2T_0PT_0S_2f=\frac{\langle T_0f,v\rangle}{\|V\|_{L^1}}S_2(T_0v),\]
so
\begin{equation}\label{eq-kernel T_2-2}
\langle S_2T_0^2S_2f,f\rangle=\frac{\langle T_0f,v\rangle^2}{\|V\|_{L^1}}.
\end{equation}
Then, for the third part of \eqref{eq-kernel T_2}, using \eqref{eq-D_0-1} and
\[QD_0Q=D_0,\ \ Q(xv(x))=v_1,\ \ S_2(v(x))=S_2(xv(x))=S_2(x^2v(x))=0,\]
we have
\[S_2vG_1vD_0vG_1vS_2f=S_2vG_1vQD_0QvG_1vf=-8\langle f,x^3v\rangle S_2(x^3v),\]
and
\begin{equation}\label{eq-kernel T_2-3}
\langle S_2vG_1vD_0vG_1vS_2f ,f\rangle =-8\big|\langle f,x^3v\rangle\big|^2.
\end{equation}
Finally, for the last part of \eqref{eq-kernel T_2}, using \eqref{eq-D_0-1} and
\[S_1D_2S_1=D_2,\ \ S_1(v)=S_1(xv)=0,\ \ S_2(v(x))=S_2(xv(x))=S_2(x^2v(x))=0,\]
we have
\[
S_2T_0vG_{-1}vD_2vG_{-1}vT_0S_2f=\langle T_0f,v\rangle\cdot \langle D_2S_1 (x^2v),S_1(x^2v)\rangle S_2(T_0v),
\]
\[
S_2T_0vG_{-1}vD_2T_0D_0vG_1vS_2f=2\frac{\langle f,x^3v\rangle}{\langle v_1,v_1\rangle}\cdot \langle D_2S_1 (T_0v_1),S_1(x^2v)\rangle S_2(T_0v),
\]
\[
S_2vG_1vD_0T_0D_2vG_{-1}vT_0S_2f=2\frac{\langle T_0f,v\rangle}{\langle v_1,v_1\rangle}\cdot\langle D_2S_1 (T_0v_1),S_1(x^2v)\rangle S_2(x^3v),
\]
\[
S_2vG_1vD_0T_0D_2T_0D_0vG_1vS_2f=4\frac{\langle f,x^3v\rangle}{\langle v_1,v_1\rangle^2}\cdot\langle D_2S_1 (T_0v_1),S_1 (T_0v_1)\rangle S_2(x^3v).
\]
Denote $X=\langle f, x^3v\rangle, Y=\langle T_0f,v\rangle,$ we obtain that
\begin{equation}\label{eq-kernel T_2-4}
\begin{split}
&\ \Big\langle \Big(S_2T_0vG_{-1}v-\frac{\|V\|_{L^1}}{6}S_2vG_1vD_0T_0\Big)D_2\Big(v G_{-1}v T_0S_2-\frac{\|V\|_{L^1}}{6}T_0D_0vG_1vS_2\Big)f,f\Big\rangle\\
=&\frac{\|V\|_{L^1}^2\cdot\langle D_2S_1 (T_0v_1),S_1 (T_0v_1)\rangle}{9\langle v_1,v_1\rangle^2}\cdot\big|X\big|^2+\langle D_2S_1 (x^2v),S_1(x^2v)\rangle\cdot\big|Y\big|^2\\
&\ -\frac{2\|V\|_{L^1}\cdot\langle D_2S_1 (T_0v_1),S_1(x^2v)\rangle}{3\langle v_1,v_1\rangle}\cdot{\rm Re}(X\overline{Y}).
\end{split}
\end{equation}
Then, substituting \eqref{eq-kernel T_2-1},\eqref{eq-kernel T_2-2},\eqref{eq-kernel T_2-3} and \eqref{eq-kernel T_2-4} into \eqref{eq-kernel T_2}, we have
\begin{align*}
0=\langle T_2f,f\rangle&=\Big(\frac{32\cdot(5!)\cdot\langle D_2S_1 (T_0v_1),S_1 (T_0v_1)\rangle}{\langle v_1,v_1\rangle^2}-40\Big)\cdot\big|X\big|^2\\
&+\frac{8\cdot(6!)}{\|V\|^2_{L^1}}\Big(6\langle D_2S_1 (x^2v),S_1(x^2v)\rangle-1\Big)\cdot\big|Y\big|^2
 -\frac{32\cdot(6!)\langle D_2S_1 (T_0v_1),S_1(x^2v)\rangle}{\langle v_1,v_1\rangle\cdot\|V\|_{L^1}}\cdot{\rm Re}(X\overline{Y})\\
&:= a\big|X\big|^2+b\big|Y\big|^2+c~{\rm Re}(X\overline{Y}).
\end{align*}
Let $X=X_1+iX_2$ and $Y=Y_1+i Y_2$ with $X_1,X_2,Y_1,Y_2$ are real. Then we have
\begin{align*}
0=\langle T_2f,f\rangle=&a(X_1^2+X_2^2)+b(Y_1^2+Y_2^2)+c(X_1Y_1+X_2Y_2)=\sum_{k=1}^2\Big(aX_k^2+bY_k^2+cX_kY_k\Big).
\end{align*}
Hence, without loss of generality, we only need to consider the case that $X$ and $Y$ are real, that is,
\[0=\langle T_2f,f\rangle= aX^2+bY^2+cXY.\]
Let $g_1=S_1(x^2v), h_1=D_2g_1\in S_1L^2(\mathbf{R}),g_2=S_1(T_0v_1),h_2=D_2g_2,$ since $D_2=(T_1+S_2)^{-1},$ using $S_1(v(x))=S_1(xv(x))=0,$ we have
\[g_1=(T_1+S_2)h_1=12\langle g_1,h_1\rangle g_1+192\frac{\langle g_2, h_1\rangle}{\langle v_1,v_1\rangle ^2}g_2+S_2h_1, \]
and
\[g_2=12\langle g_1,h_2\rangle g_1+192\frac{\langle g_2, h_2\rangle}{\langle v_1,v_1\rangle ^2}g_2+S_2h_2.\]
Then we have
\begin{align*}
\langle D_2S_1 (x^2v),S_1(x^2v)\rangle=&\langle g_1,h_1\rangle
=12\langle g_1,h_1\rangle^2+192\frac{\langle g_2, h_1\rangle^2}{\langle v_1,v_1\rangle ^2}+\langle S_2h_1,h_1\rangle,
\end{align*}
since $\langle S_2h_1,h_1\rangle\ge 0$ and $\displaystyle\frac{\langle g_2, h_1\rangle^2}{\langle v_1,v_1\rangle ^2}\ge0,$ we have
\[
\langle g_1,h_1\rangle
\ge12\langle g_1,h_1\rangle^2,
\]
thus, we obtain that $\displaystyle\langle g_1,h_1\rangle\le \frac{1}{12}.$ Furthermore, if $\langle g_2, h_1\rangle\neq0,$ then $\displaystyle\langle g_1,h_1\rangle< \frac{1}{12}.$

Similarly, we have
\begin{align*}
\langle D_2S_1 (T_0v_1),S_1 (T_0v_1)\rangle=&\langle g_2,h_2\rangle
=12\langle g_1,h_2\rangle^2+192\frac{\langle g_2, h_2\rangle^2}{\langle v_1,v_1\rangle ^2}+\langle S_2h_2,h_2\rangle,
\end{align*}
thus we also have
\[\langle g_2,h_2\rangle\le \frac{\langle v_1,v_1\rangle ^2}{192},\]
and if $\langle g_1,h_2\rangle\neq0,$ we have $\displaystyle\langle g_2,h_2\rangle< \frac{\langle v_1,v_1\rangle ^2}{192}.$\\
Hence, we obtain that
\begin{align*}
a=&\Big(\frac{32\cdot(5!)\cdot\langle D_2S_1 (T_0v_1),S_1 (T_0v_1)\rangle}{\langle v_1,v_1\rangle^2}-40\Big)\le \frac{32\cdot(5!)}{192}-40=-20<0,\\
b=&\Big(6\langle D_2S_1 (x^2v),S_1(x^2v)\rangle-1\Big)\le \frac{6}{12}-1=-\frac{1}{2}<0.
\end{align*}
Furthermore, by Cauchy-Schwarz inequality, we obtain that
\[\langle g_1,h_2\rangle^2=\langle g_1, D_2g_2\rangle^2\le\langle g_1, D_2g_1\rangle \langle g_2, D_2g_2\rangle=\langle g_1, h_1\rangle\langle g_2, h_2\rangle.\]
Thus
\begin{align*}
c^2-4ab=&\ \frac{32^2\cdot(6!)^2\langle g_1,h_2\rangle^2}{\langle v_1,v_1\rangle^2\cdot\|V\|_{L^1}^2}-\Big(\frac{32\cdot(5!)\cdot\langle g_2,h_2\rangle}{\langle v_1,v_1\rangle^2}-40\Big)\cdot\frac{32\cdot(6!)}{\|V\|^2_{L^1}}\Big(6\langle g_1,h_1\rangle-1\Big)\\
\le &\ \frac{32\cdot(6!)}{\|V\|^2_{L^1}}\Big(\frac{32\cdot(5!)\cdot\langle g_2,h_2\rangle}{\langle v_1,v_1\rangle^2}+240\langle g_1,h_1\rangle-40\Big)\\
\le &\ \frac{32\cdot(6!)}{\|V\|^2_{L^1}}\Big(\frac{32\cdot(5!)}{192}+\frac{240}{12}-40\Big)=0.
\end{align*}
Since ``=" holds if and only if $\langle g_1,h_2\rangle=\langle g_2,h_1\rangle=0,$ and in this case, we have
\[0=\langle T_2f,f\rangle\le aX^2+b Y^2\le 0,\]
thus it is necessary that
\[X=\langle f,x^3v\rangle=0,\ \ \  Y=\langle T_0f,v\rangle=0.\]
Then $f\in S_3L^2(\mathbf{R}).$ On the other hand, it is obvious that if $f\in S_3L^2(\mathbf{R}),$ we can obtain that $T_2f=0,$ which means that $f\in \ker T_2.$
\end{proof}

\begin{lemma}\label{lemma-S_3}
$S_3L^2(\mathbf{R})=\{0\}.$
\end{lemma}
\begin{proof}
The proof can be divided into the following two steps.

\textbf{Step~\rm{I}}:
 We first show that $f(x)\in S_3L^2(\mathbf{R})$ if and only if $f(x)=Uv\phi$  with some $\phi\in L^2(\mathbf{R})$ such that $H\phi=0$ in the distributional sense. Furthermore,
$\phi= -G_0vf.$

Since $ S_3 \leq S_2 $, by Theorem \ref{thm-resonance space}(ii) below, we have known that  if $f\in S_2L^2(\mathbf{R})$, then $f=Uv\phi$ with $ \phi \in W_{\frac{1}{2}}(\mathbf{R})$  such  that $H\phi =0$ in the distributional sense and
$$ \phi = -G_0v f + \frac{ \langle v , T_0f \rangle}{\|V\|_{L^1}}.$$
Thus, we just need to  prove that $\displaystyle\frac{ \langle T_0f,v \rangle}{\|V\|_{L^1(\mathbf{R})}}=0$ and $\phi \in L^2(\mathbf{R}) $. Indeed, if $f\in S_3L^2(\mathbf{R})$, then
\[\int_{\mathbf{R}}v(y)f(y)dy=\int_{\mathbf{R}}yv(y)f(y)dy=\int_{\mathbf{R}}y^2v(y)f(y)dy=\int_{\mathbf{R}}y^3v(y)f(y)dy=QT_0f=PT_0f=0.\]
Thus, we have $ \displaystyle\frac{ \langle v,T_0f\rangle}{\|V\|_{L^1}}=0$.

Next we show that $ \phi \in L^2(\mathbf{R})$. Since $ \phi = -G_0vf $ , and using
\[S_3(v)=S_3(xv)=S_3(x^2v)=S_3(x^3v)=0,\] we have
\begin{equation*}
\begin{split}
\phi (x)
=& -\frac{1}{12}\int_{\mathbf{R}}\big( |x-y|^3-|x|^3 + 3|x|x\cdot y - 3|x|y^2 +
\frac{x}{|x|} y^3\big)v(y)f(y)dy\\
:=& - \frac{1}{12}\int_{\mathbf{R}}K_3(x,y)v(y)f(y)dy,
\end{split}
\end{equation*}
where $$K_3(x,y)= |x-y|^3-|x|^3 + 3|x|x\cdot y - 3|x|y^2 + \displaystyle\frac{x}{|x|} y^3. $$

Indeed, we can divided $K_3(x,y)$ into five parts as follows
\begin{equation*}
\begin{split}
K_3(x,y)=& \Big[y^2\big(  |x-y|-|x| \big)+ \frac{x }{|x|}y^3\Big]
            + \Big[\frac{x }{|x|}y^3 - \frac{2xy^3}{|x-y|+|x|}\Big]\\
            &+\frac{5}{2} \Big[  \frac{|x|y^2(|x|-|x-y|)}{|x-y|+|x|}- \frac{1}{2}\frac{x}{|x|}y^3 \Big]
            + \Big[ \frac{x^2y^2(|x-y|-|x|)}{(|x-y|+|x|)^2} +\frac{1}{4}\frac{x}{|x|}y^3  \Big]\\
            &+\frac{|x|y^4}{2(|x-y|+|x|)^2}\\
       :=& K_{31}(x,y)+ K_{32}(x,y)+K_{33}(x,y)+K_{34}(x,y)+K_{35}(x,y).
\end{split}
\end{equation*}
For the first part $K_{31}(x,y),$ we have
\begin{equation*}
\begin{split}
|K_{31}(x,y)|=& \Big|\frac{ |x|y^2(y^2-2x\cdot y )}{ (|x-y|+|x|)|x|}
               + \frac{ xy^3(|x-y|+|x| )}{ (|x-y|+|x|)|x|} \Big|  \\
        \lesssim & \frac{y^4}{|x-y|+|x|} + \Big|\frac{ xy^3(|x-y|-|x| )}{ (|x-y|+|x|)|x|} \Big|  \\
 \lesssim & \frac{y^4}{|x-y|+|x|} .
\end{split}
\end{equation*}
Moreover,
\begin{equation*}
	|K_{31}(x,y)| \lesssim  \frac{y^4}{|x-y|+|x|} \lesssim
  \begin{cases}
	|y|^3,\,\,\,  & 0\leq |x| \leq 1;\\
	\displaystyle\frac{y^4}{|x|},\,\,\,  & |x| > 1.
   \end{cases}
\end{equation*}
Thus
$$ \Big|\int_{\mathbf{R}} K_{31}(x,y)v(y)f(y)dy \Big|
 \lesssim  \int_{\mathbf{R}} \frac{y^4}{|x-y|+|x|}|v(y)f(y)|dy \in L^2(\mathbf{R}). $$
For the second part $K_{32}(x,y),$ we have
\begin{equation*}
\begin{split}
|K_{32}(x,y)| =& \Big|\frac{ x|x|y^3 + xy^3|x-y|- 2|x|xy^3}{|x|(|x-y|+|x)|}\Big|\\
              =&\Big|\frac{ xy^3( |x-y|-|x|)}{|x|(|x-y|+|x)|}\Big|
              \lesssim \frac{y^4}{|x-y|+|x|}.
\end{split}
\end{equation*}
Thus
$$ \Big|\int_{\mathbf{R}} K_{32}(x,y)v(y)f(y)dy \Big|
 \lesssim  \int_{\mathbf{R}} \frac{y^4}{|x-y|+|x|} |v(y)f(y)|dy \in L^2(\mathbf{R}). $$
For the third part $K_{33}(x,y),$ we have
\begin{equation*}
\begin{split}
|K_{33}(x,y)| =& \frac{5}{2}\Big|\frac{ x^2y^2(2x\cdot y -y^2)}{|x|(|x-y|+|x|)^2}
                - \frac{1}{2}\frac{ xy^3(2|x-y|\cdot |x| +2x^2-2x\cdot y+y^2 )}{|x|(|x-y|+|x|)^2} \Big|\\
   \lesssim &  \frac{|x||y|^5}{|x|(|x-y|+|x|)^2} + \Big|  \frac{ xy^3(|x-y|-|x|)}{(|x-y|+|x|)^2}  \Big| \\
                \lesssim &   \frac{|x|y^4}{(|x-y|+|x|)^2} .
\end{split}
\end{equation*}
Noting that
\begin{equation}\label{eq-K_33}
	|K_{33}(x,y)| \lesssim  \frac{|x|y^4}{(|x-y|+|x|)^2} \lesssim
  \begin{cases}
	|y|^3,\,\,\,  & 0\leq |x| \leq 1;\\
	\displaystyle\frac{y^4}{|x|},\,\,\,  & |x| > 1.
   \end{cases}
\end{equation}
Thus
$$ \Big|\int_{\mathbf{R}} K_{33}(x,y)v(y)f(y)dy \Big|
 \lesssim  \int_{\mathbf{R}} \frac{|x|y^4}{(|x-y|+|x|)^2}|v(y)f(y)|dy \in L^2(\mathbf{R}). $$
Then, for the fourth part $K_{34}(x,y),$ we have
\begin{equation*}
\begin{split}
|K_{34}(x,y)| =& \Big|\frac{ |x|x^2y^2(y^2-2x\cdot y )}{|x|(|x-y|+|x|)^3}
                + \frac{ xy^3(|x-y|+|x|)^3}{4|x|(|x-y|+|x|)^3}   \Big|\\
  \lesssim &  \frac{ x^2y^4}{(|x-y|+|x|)^3} +  \frac{|x|\cdot |y|^5 |x-y|}{(|x-y|+|x|)^3\cdot |x|}\\
               &+  \frac{|x|\cdot |y|^5}{(|x-y|+|x|)^3}
               + \Big| \frac{x^3y^3(|x-y|-|x|)}{(|x-y|+|x|)^3|x|}\Big| \\
   \lesssim & \frac{x^2y^4}{(|x-y|+|x|)^3}+ \frac{|y|^5}{(|x-y|+|x|)^2}
    + \frac{|x|\cdot |y|^5}{(|x-y|+|x|)^3},
\end{split}
\end{equation*}
and it is obvious that
\begin{equation*}
	\frac{x^2y^4}{(|x-y|+|x|)^3}+ \frac{|y|^5}{(|x-y|+|x|)^2}
    + \frac{|x|\cdot |y|^5}{(|x-y|+|x|)^3}\lesssim
  \begin{cases}
	|y|^3,\,\,\,  & 0\leq |x| \leq 1;\\
	\displaystyle\frac{y^4}{|x|},\,\,\,  & |x| > 1.
   \end{cases}
\end{equation*}
Thus again
\begin{equation*}
\begin{split}
 &\Big|\int_{\mathbf{R}} K_{34}(x,y)v(y)f(y)dy \Big|\\
 &\lesssim  \int_{\mathbf{R}}\Big[ \frac{x^2y^4}{(|x-y|+|x|)^3}
  + \frac{|y|^5}{(|x-y|+|x|)^2}+ \frac{|x|\cdot |y|^5}{(|x-y|+|x|)^3}\Big] |v(y)f(y)|dy
  \in L^2(\mathbf{R}).
 \end{split}
\end{equation*}
Finally, for the fifth part $K_{35}(x,y),$ using \eqref{eq-K_33}, it is easy to know that
$$ \Big|\int_{\mathbf{R}} K_{35}(x,y)v(y)f(y)dy \Big|
 =\int_{\mathbf{R}} \frac{|x|y^4}{2(|x-y|+|x|)^2}|v(y)f(y)|dy \in L^2(\mathbf{R}). $$
Thus, we obtain that  $\phi= -G_0vf \in L^2(\mathbf{R})$,  and
$\int_{\mathbf{R}} K_3(x,y)v(y)f(y)dy \in L^2(\mathbf{R}) $.

On the other hand, if $\phi\in L^2(\mathbf{R})$ such that $H\phi=0.$ Denote $f=Uv\phi,$ using the similar method of \eqref{eq-orthogonality proof} and \eqref{eq-S_0}, we have
\begin{equation*}
\begin{split}
\int_{\mathbf{R}}v(y)f(y)dy&=0,\
    \int_{\mathbf{R}}yv(y)f(y)dy=0,\
     \int_{\mathbf{R}}y^2v(y)f(y)dy=0,\\
   &\int_{\mathbf{R}}y^3v(y)f(y)dy=0, S_0T_0S_0f=0,
   \end{split}
 \end{equation*}
then we can obtain that $f\in S_1L^2(\mathbf{R})$ and $G_0vf\in  L^2(\mathbf{R}).$ Moreover, by Theorem \ref{thm-resonance space}, we know that $\phi=-G_0vf+c_1x+c_2$ since $L^2(\mathbf{R})\subset W_{\frac{3}{2}}(\mathbf{R}).$ Thus we need that $c_1x+c_2=\phi+G_0vf\in W_{\frac{1}{2}}(\mathbf{R})$
this necessitates that $c_1=c_2=0,$ which implies that $T_0f=0$ and $\phi=-G_0vf$. Hence, we have $f\in S_3L^2(\mathbf{R})$.

\footnote{The proof was suggested to us by M. Goldberg \cite{Goldberg}}\textbf{Step~\rm{II}}: Now we turn to prove that zero is not an eigenvalue of $H$ for $V(x)$ satisfies the condition in Theorem \ref{thm-main results}.

Let $\phi(x)\in L^2(\mathbf{R})$ such that $H\phi(x)=0.$ Denote $f=Uv\phi,$ by the \textbf{Step~\rm{I}}, then it follows that $f\in S_3L^2(\mathbf{R})$ and
\begin{align*}
L^2(\mathbf{R})\ni\phi(x)=-G_0vf=&-\frac{1}{12}\int_{\mathbf{R}}|y-x|^3v(y)f(y)dy\\
=&-\frac{1}{12}\Big(\int_x^\infty(y-x)^3v(y)f(y)dy+\int_{-\infty}^x(x-y)^3v(y)f(y)dy\Big).
\end{align*}
Then using the following orthogonality relationships from $S_3L^2(\mathbf{R})$:
$$ \int_{\mathbf{R}}v(y)f(y)dy=0,\
    \int_{\mathbf{R}}yv(y)f(y)dy=0,\
     \int_{\mathbf{R}}y^2v(y)f(y)dy=0,
   \int_{\mathbf{R}}y^3v(y)f(y)dy=0,$$
and  $v(y)f(y)=V(y)\phi(y),$ we have
$$\phi(x)=-\frac{1}{6}\int_x^\infty(y-x)^3v(y)f(y)dy=-\frac{1}{6}\int_x^\infty(y-x)^3V(y)\phi(y)dy.$$
 Based on the decay condition of Theorem \ref{thm-main results}, that is, $|V(x)|\lesssim\langle x\rangle^{-\beta}$ for $\beta>25$, and $\langle x \rangle=(1+|x|^2)^{1/2}$, then  $\displaystyle\lim_{x\rightarrow+\infty}|\phi(x)|=0$ and for any $x_0>x,$
\begin{align*}
|\phi(x)|\le {1\over 6}\int_x^{x_0}(y-x)^3|V(y)|\cdot|\phi(y)|dy+\big(\sup_{y\in({x_0},+\infty)}|\phi(y)|\big)\int_{x_0}^{+\infty}{1\over 6}(y-x)^3|V(y)|dy.
\end{align*}
Hence, we have
\begin{align*}
\langle x\rangle^{-3}|\phi(x)|\lesssim \int_x^{x_0}\langle y\rangle^{6}|V(y)|\cdot \langle y\rangle^{-3}|\phi(y)|dy+\Big(\sup_{y\in({x_0},+\infty)}|\phi(y)|\Big)\int_{x_0}^{+\infty}\langle y\rangle^{3}|V(y)|dy.
\end{align*}
Let $$\displaystyle u(x)=\langle x\rangle^{-3}|\phi(x)|,\ \ B(y)=\langle y\rangle^{6}|V(y)|$$ and $$A(x_0)=\displaystyle\big(\sup_{y\in({x_0},+\infty)}|\phi(y)|\big)\int_{x_0}^{+\infty}\langle y\rangle^{3}|V(y)|dy.$$ Then we have
\[u(x)\lesssim \int_x^{x_0}B(y)u(y)dy+A(x_0).\]
Hence by Gronwall's inequality ( see e.g. \cite[Theorem 1.10]{TT2} ), we conclude that
\[u(x)\lesssim A(x_0)e^{C\int_x^{x_0} B(y)dy}.\]
Note that $\langle y\rangle^{3}|V(y)|\in L^1(\mathbf{R})$,  hence combining with $\displaystyle\lim_{y\rightarrow+\infty}\phi(y)=0,$ we can obtain that~$\displaystyle\lim_{{x_0}\rightarrow+\infty}A({x_0})=0.$~ Moreover,  $\int_x^{x_0} B(y)dy$ can be bounded by a fixed constant. Hence for any $x\in\mathbf{R},$
\[u(x)\lesssim \lim_{x_0\rightarrow+\infty}A(x_0)e^{C\int_x^{x_0} B(y)dy}=0,\]
which implies $\phi(x)\equiv 0.$  That is,  zero can not be an eigenvalue of $H$.

Combined \textbf{Step~\rm{I}} with \textbf{Step~\rm{II}}, we can conclude that $S_3L^2(\mathbf{R})=\{0\}.$
\end{proof}

\begin{remark}\label{remark-eigenvalue-2}
Since the proof of \textbf{Step~\rm{II}} depends on the orthogonality relationships of $S_3L^2(\mathbf{R})$, so we need the decay rate $\beta$ of $V$ satisfies \eqref{eq-beta}.  However, we remark that the required decay rate $\beta$ in the proof above may not be optimal to ensure the absence of zero eigenvalue of $H$. In fact, in view of Remark \ref{remark-eigenvalue}, it is reasonable to guess that $\beta>4$ might be the best one.
\end{remark}

 In the following, we come to show the proof of Theorem \ref{thm-M-inverse} case by case. For the convenience, we will use the notation  $A_{kj}^i$ ($i$ means the kind of zero resonance, $k$ respects to the power of $\mu$) as some bounded operators in $L^2(\mathbf{R})$, and these operators may vary from line to line.

\noindent\textbf{Proof of Theorem \ref{thm-M-inverse}:} We prove the assertion for $+$ sign. By  Lemma \ref{Feshbach-formula}, we know that $\widetilde{M}^+(\mu)$ is invertible on $L^2(\mathbf{R})$ if and
only if
\begin{equation}\label{eq-M_1-1}
M_1^+(\mu):= Q-Q\Big( \widetilde{M^+}(\mu)+Q\Big)^{-1}Q
\end{equation}
is invertible on $QL^2(\mathbf{R})$. Moreover, using \eqref{eq-M-expansions-3} and by the definition of $B_1^+,$ we get
\begin{equation}\label{eq-M_1-2}
   \begin{split}
M_1^+(\mu)
         =&\sum_{k=1}^{12}\mu^{k+1}QB_k^+Q+Q\Gamma_{14}(\mu)Q\\
         =&\frac{a_{-1}^+}{\widetilde{a}^+} \mu^2 \Big(QvG_{-1}vQ + \frac{\widetilde{a}^+}{a_1^+}\sum_{k=2}^{12}\mu^{k-1}QB_k^+Q+Q\Gamma_{12}(\mu)Q\Big) \\
         :=&\frac{a_{-1}^+}{\widetilde{a}^+}\mu^2 \widetilde{M_1}^+(\mu).
   \end{split}
\end{equation}
Thus, in order to compute the expansions of $\Big(M_1^+(\mu)\Big)^{-1}$, we need to obtain the expansions of
$\Big(\widetilde{M_1^+}(\mu)\Big)^{-1}$ as $\mu \rightarrow 0$ on $QL^2(\mathbf{R})$.

 By Lemma \ref{lemma-projection space}(1), we know that $QvG_{-1}vQ$ is not invertible on $QL^2(\mathbf{R})$. Since $S_0$ be the Riesz projection onto
the kernel of $QvG_{-1}vQ$, then $QvG_{-1}vQ +S_0$ is invertible on $QL^2(\mathbf{R})$.  By using \eqref{eq-M_1-2}, we have
\begin{equation}\label{eq-M_1-3}
   \begin{split}
\widetilde{M_1^+}(\mu)+S_0
         =& \big(QvG_{-1}vQ+ S_0\big)+\frac{\widetilde{a}^+}{a_1^+}\sum_{k=2}^{12}\mu^{k-1}QB_k^+Q+Q\Gamma_{12}(\mu)Q.
   \end{split}
\end{equation}
Then, using the Neumann series, we obtain that
\begin{equation}\label{eq-inverse-M_1-1}
   \begin{split}
\Big(\widetilde{M_1^+}(\mu)+S_0\Big)^{-1}
         =&D_0-\sum_{k=1}^{11}\mu^kB_k^0+D_0\Gamma_{12}(\mu)D_0,
   \end{split}
\end{equation}
where $B_k^0\ (1\le k\le11)$ are bounded operators in $QL^2(\mathbf{R})$ and we have
\begin{equation*}\label{eq-defi-B^0}
\begin{split}
B_1^0=&\frac{\widetilde{a}^+}{a_{-1}^+}D_0 B_2^+ D_0,\ \ \ \ \ \ B_2^0= \frac{\widetilde{a}^+}{a_{-1}^+}D_0 B_3^+D_0 -\frac{(\widetilde{a}^+)^2}{(a_{-1}^+)^2}D_0(B_2^+D_0)^2,\\
B_3^0=& \frac{\widetilde{a}^+}{a_{-1}^+}D_0B_4^+D_0 + \frac{(\widetilde{a}^+)^3}{(a_{-1}^+)^3}D_0(B_2^+D_0)^3-
\frac{(\widetilde{a}^+)^2}{(a_{-1}^+)^2}\big(D_0B_2^+D_0B_3^+D_0+ D_0B_3^+D_0B_2^+D_0  \big),\\
B_4^0=&\frac{\widetilde{a}^+}{a_{-1}^+}D_0B_5^+D_0-\frac{(\widetilde{a}^+)^4}{(a_{-1}^+)^4}D_0(B_2^+D_0)^4
 +\frac{(\widetilde{a}^+)^3}{(a^+_{-1})^3}\Big(D_0(B_2^+D_0)^2B_3^+D_0+ D_0B_2^+D_0B_5^+D_0B_2^+D_0\\
 &+D_0B_3^+D_0(B_2^+D_2)^2 \Big)
     -\frac{(\widetilde{a}^+)^2}{(a_{-1}^+)^2}\Big(D_0B_2^+D_0B_4^+D_0+D_0B_4^+D_0B_2^+D_0+ D_0(B_3^+D_0)^2  \Big),\\
\end{split}
\end{equation*}
\begin{equation*}
\begin{split}
B_5^{0}=& \frac{\widetilde{a}^+}{a_{-1}^+}D_0B_6^+D_0+\frac{(\widetilde{a}^+)^5}{(a_{-1}^+)^5}D_0(B_2^+D_0)^5\\
&-\frac{(\widetilde{a}^+)^4}{(a_{-1}^+)^4}\Big(D_0(B_2^+D_0)^3B_3^+D_0
         + D_0(B_2^+D_0)^2B_3^+D_0B_2^+D_0+D_0B_2^+D_0B_3^+D_0(B_2^+D_0)^2\\
&+D_0B_3^+D_0(B_2^+D_0)^3  \Big)
+\frac{(\widetilde{a}^+)^3}{(a_{-1}^+)^3} \Big(D_0B_2^+D_0(B_3^+D_0)^2 + D_0B_3^+D_0B_2^+D_0B_3^+D_0 + D_0(B_3^+D_0)^2B_2^+D_0\\
&+D_0(B_2^+D_0)^2B_4^+D_0+ D_0B_2^+D_0B_4^+D_0B_2^+D_0+ D_0B_4^+D_0(B_2^+D_0)^2  \Big)\\
&-\frac{(\widetilde{a}^+)^2}{(a_{-1}^+)^2}\Big(D_0B_3^+D_0B_4^+D_0 +D_0B_4^+D_0B_3^+D_0+ D_0B_2^+D_0B_5^+D_0+ D_0B_5^+D_0B_2^+D_0   \Big).
\end{split}
\end{equation*}
Using Lemma \ref{Feshbach-formula} again, we know that $\widetilde{M_1^+}(\mu)$ invertible on $QL^2(\mathbf{R})$
if and only if
\begin{equation}\label{eq-M_2-1}
M_2^+(\mu):= S_0 -S_0\big( \widetilde{M_1^+}(\mu)+S_0 \big)^{-1}S_0
\end{equation}
 invertible on $S_0L^2(\mathbf{R})$. Notice that $S_0 D_0 = D_0S_0 =S_0$, and using \eqref{eq-inverse-M_1-1}, we obtain that
\begin{equation}\label{eq-M_2-2}
   \begin{split}
M_2^+(\mu)=&\mu S_0B_1^0S_0 +\sum_{k=2}^{11}\mu^kB_k^0+S_0\Gamma_{12}(\mu)S_0\\
   =& \frac{\mu}{a_{-1}^+}\Big( S_0T_0S_0 + a_{-1}^+\sum_{k=2}^{11}\mu^{k-1}B_k^0+S_0\Gamma_{11}(\mu)S_0\Big)\\
   :=&\frac{\mu}{a_{-1}^+}\widetilde{M_2^+}(\mu).
	\end{split}
\end{equation}

(1)~\textbf{Regular Case.}~In this case, we using $\eqref{eq-M-expansions-1}$ which need $|V(x)|\lesssim (1+|x|)^{-13-},$ see Lemma \ref{lemma-M expansion}. Then by the same process, we can rewrite $M_2^+(\mu)$ as follows:
\begin{equation*}\label{eq-M_2-2-2}
   \begin{split}
M_2^+(\mu)=&\mu S_0B_1^0S_0 +\mu^2B_2^0+\mu^3B_3^0+S_0\Gamma_{4}(\mu)S_0\\
   =& \frac{\mu}{a_{-1}^+}\Big( S_0T_0S_0 + a_{-1}^+\mu B_2^0+a_{-1}^+\mu^2B_3^0+S_0\Gamma_{3}(\mu)S_0\Big)\\
   :=&\frac{\mu}{a_{-1}^+}\widetilde{M_2^+}(\mu).
	\end{split}
\end{equation*}
If zero is a regular point of the spectrum of $H,$ then $S_0T_0S_0$ is invertible on $S_0L^2(\mathbf{R}).$ Denote $\widetilde{D_0}=\big(S_0T_0S_0\big)^{-1},$ by \eqref{eq-M_2-2}, using Neumann series and $S_0\widetilde{D}_0=\widetilde{D}_0S_0=\widetilde{D}_0,$ we obtain that
\begin{equation*}\label{eq-inverse-M_2-1}
\begin{split}
\Big(\widetilde{M_2^+}(\mu)\Big)^{-1}=&\widetilde{D_0}-a_{-1}^+\mu \widetilde{D}_0B_2^0\widetilde{D}_0
+\mu^2\big((a_{-1}^+)^2\widetilde{D}_0(B_2^0\widetilde{D}_0)^2-a_{-1}^+\widetilde{D}_0B_3^0\widetilde{D}_0\big)+\Gamma_{3}(\mu).
\end{split}
\end{equation*}
Then, using \eqref{eq-M_2-2}, we have
\begin{equation*}\label{eq-inverse-M_2-2}
\begin{split}
\Big(M_2^+(\mu)\Big)^{-1}
=&\frac{a_{-1}^+}{\mu}\widetilde{D_0}-(a_{-1}^+)^2 \widetilde{D}_0B_2^0\widetilde{D}_0
+\mu\big((a_{-1}^+)^3\widetilde{D}_0(B_2^0\widetilde{D}_0)^2-(a_{-1}^+)^2\widetilde{D}_0B_3^0\widetilde{D}_0\big)+\Gamma_{2}(\mu).
\end{split}
\end{equation*}
By Lemma \ref{Feshbach-formula} with $A=\widetilde{M_1^+}(\mu)$, $S=S_0$ and $B=M_2^+(\mu)$, using the orthogonality properties ( see \eqref{eq-orthogonality-1}, \eqref{eq-orthogonality-2}), we obtain that
\begin{equation*}\label{eq-inverse-M_1-2}
\begin{split}
\Big(\widetilde{M_1^+}(\mu)\Big)^{-1}
=\ &\Big(\widetilde{M_1^+}(\mu)+S_0\Big)^{-1}+\Big(\widetilde{M_1^+}(\mu)+S_0\Big)^{-1}S_0\Big(M_2^+(\mu)\Big)^{-1}S_0\Big(\widetilde{M_1^+}(\mu)+S_0\Big)^{-1}\\
=\ &\frac{S_0A_{-1,1}^0S_0}{\mu}+QA_{01}^0Q+\mu QA_{11}^0Q+\Gamma_{2}(\mu).
\end{split}
\end{equation*}
Then, using \eqref{eq-M_1-2}, we have
\begin{equation*}\label{eq-inverse-M_1-3}
\begin{split}
\Big(M_1^+(\mu)\Big)^{-1}
=&\frac{S_0A_{-3,1}^0S_0}{\mu^3}+\frac{QA_{-2,1}^0Q}{\mu^2}+\frac{QA_{-1,1}^0Q}{\mu}+\Gamma_{0}(\mu).
\end{split}
\end{equation*}
Then, we use Lemma \ref{Feshbach-formula} with $A=\widetilde{M^+}(\mu),S=Q$ and $B=M_1^+(\mu)$, and using the orthogonality properties ( see \eqref{eq-orthogonality-1}, \eqref{eq-orthogonality-2}), we have
\begin{equation*}\label{eq-inverse-M-1}
\begin{split}
\Big(\widetilde{M^+}(\mu)\Big)^{-1}=\ & \Big(\widetilde{M^+}(\mu)+Q\Big)^{-1} +\Big(\widetilde{M^+}(\mu)+Q\Big)^{-1} Q\Big(M_1^+(\mu)\Big)^{-1}Q\Big(\widetilde{M^+}(\mu)+Q\Big)^{-1} \\
=& \frac{S_0A_{-3,1}^0S_0}{\mu^3}+\frac{QA_{-2,1}^0Q}{\mu^2}+\frac{QA_{-1,1}^0Q+S_0{A_{-1,2}^0}+{A_{-1,3}^0}S_0}{\mu}+\Gamma_{0}(\mu).
\end{split}
\end{equation*}
Finally, since~$\displaystyle M^+(\mu)= \frac{\widetilde{a}^+}{\mu^3} \widetilde{M^+}(\mu),$ then we can say that
\begin{equation}\label{eq-inverse-M-Regular Case}
\begin{split}
\Big(M^+(\mu)\Big)^{-1}=&S_0A_{01}^0S_0+\mu QA_{11}^0Q+\mu^2(QA_{21}^0Q+S_0A_{22}^0+A_{23}^0S_0)+\Gamma_3^0(\mu).
\end{split}
\end{equation}

(2)~\textbf{First Kind of Resonance.}~In this case, we using $\eqref{eq-M-expansions-1-2}$ with $|V(x)|\lesssim(1+|x|)^{-17-},$ see Lemma \ref{lemma-M expansion}. Then by the same process, we can rewrite $M_2^+(\mu)$ as follows:
\begin{equation*}\label{eq-M_2-2-3}
   \begin{split}
M_2^+(\mu)=&\mu S_0B_1^0S_0 +\sum_{k=2}^5\mu^kB_k^0+S_0\Gamma_{6}(\mu)S_0\\
   =& \frac{\mu}{a_{-1}^+}\Big( S_0T_0S_0 + a_{-1}^+\sum_{k=2}^5\mu^{k-1}B_k^0+S_0\Gamma_{5}(\mu)S_0\Big)\\
   :=&\frac{\mu}{a_{-1}^+}\widetilde{M_2^+}(\mu).
	\end{split}
\end{equation*}
 For the first kind of resonance, we aim to derive the expansions of $\Big(\widetilde{M_2^+}(\mu)\Big)^{-1}.$ By the definition of the first kind of resonance, we know that the operator $S_0T_0S_0$ is not invertible on $S_0L^2(\mathbf{R}).$ Since $S_1$ is the Riesz projection onto the kernel of $S_0T_0S_0$ on $S_0L^2(\mathbf{R}),$ thus $S_0T_0S_0+S_1$ is invertible on $S_0L^2(\mathbf{R}).$ By Neumann series, we have
\begin{equation}\label{eq-inverse-M_2-3}
\begin{split}
\Big(\widetilde{M_2^+}(\mu)+S_1\Big)^{-1}=&D_1-\sum_{k=1}^4\mu^kB_k^1+\Gamma_5(\mu),
\end{split}
\end{equation}
where~$B_k^1\ (1\le k\le 4)$~are bounded operators in $S_0L^2(\mathbf{R})$ with the following definition
\begin{equation}\label{eq-defi-B^1}
\begin{split}
B_1^1=& a_{-1}^+D_1B_2^0D_1,\ \ \ \ \ \ B_2^1= a_{-1}^+D_1B_3^0D_1 -(a_{-1}^+)^2 D_1(B_2^0D_1)^2,\\
B_3^1=&a_{-1}^+D_1B_4^0D_1+ (a_{-1}^+)^3D_1(B_1^0D_1)^3-
 (a_{-1}^+)^2\big( D_1B_2^0D_1B_3^0D_1+ D_1B_3^0D_1B_2^0D_1  \big) ,\\
B_4^1=&(a_{-1}^+)D_1B_5^0D_1 - (a_{-1}^+)^4D_1(B_2^0D_1)^4\\
  &+(a_{-1}^+)^3\Big( D_1(B_2^0D_1)^2B_3^0D_1 + D_1B_2^0D_1B_3^0D_1B_2^0D_1 + D_1B_3^0D_1(B_2^0D_1)^2 \Big)\\
  &+(a_{-1}^+)^2\Big(D_1B_2^0D_1B_4^0D_1 + D_1B_4^0D_1B_2^0D_1+ D_1(B_3^0D_1)^2 \Big).
\end{split}
\end{equation}

According to Lemma \ref{Feshbach-formula}, we know that $ \widetilde{M_2^+}(\mu)$ is invertible on
 $S_0L^2(\mathbf{R})$ if and only if
 \begin{equation}\label{eq-M_3-1}
 M_3^+(\mu):= S_1 -S_1 \big(\widetilde{M_2^+}(\mu)+S_1 \big)^{-1}S_1
 \end{equation}
 is invertible on $S_1L^2(\mathbf{R})$. Moreover, using \eqref{eq-inverse-M_2-3} and \eqref{eq-orthogonality-2}, we obtain
\begin{equation}\label{eq-M_3-2}
\begin{split}
M_3^+(\mu)=& \mu S_1B_1^1S_1 +\mu^2S_1B_2^1S_1+ \mu^3S_1B_3^1S_1 +\mu^4S_1B_4^1S_1+
S_1\Gamma_5(\mu)S_1.\\
\end{split}
\end{equation}
Now, we shall compute the $S_1B_1^1S_1 $ in $M_3^+(\mu)$. Noting that $ S_1D_1= D_1S_1=S_1$ , we have
$S_1B_1^1S_1= a_{-1}^+S_1B_2^0S_1$.  Since
\begin{equation*}
   \begin{split}
B_2^0 &= \frac{\widetilde{a}^+}{a_{-1}^+}D_0B_3^+D_0- \frac{(\widetilde{a}^+)^2}{(a_{-1}^+)^2}D_0(B_2^+D_0)^2\\
     &=\frac{a_1^+}{a_{-1}^+}D_0vG_1vD_0-\frac{a_{-1}^+}{\widetilde{a}^+}D_0(vG_{-1}v)^2D_0
       - \frac{1}{(a_{-1}^+)^2}D_0T_0D_0T_0D_0
   \end{split}
\end{equation*}
 and $ S_1D_0= D_0S_1=S_1$,  we have
 \begin{equation}
   \begin{split}
S_1B^1_1S_1= a_{-1}^+\Big(\frac{a_1^+}{a_{-1}^+}S_1vG_1vS_1-\frac{a_{-1}^+}{\widetilde{a}^+} S_1(vG_{-1}v)^2S_1
 -\frac{1}{(a_{-1}^+)^2}S_1T_0D_0T_0S_1 \Big).
   \end{split}
\end{equation}
Notice that $(vG_{-1}v)^2= vG_{-1}v vG_{-1}v= vG_{-1}v P vG_{-1}v $, we have
\begin{equation}\label{eq-M_3-3}
S_1B_1^1S_1= a_1^+T_1,
\end{equation}
 where
\begin{equation}\label{eq-defi-T1}
\begin{split}
 T_1 = &S_1vG_1vS_1- \frac{(a_{-1}^+)^2}{\widetilde{a}^+a_1^+}S_1 vG_{-1}v P vG_{-1}v   S_1
 -\frac{1}{a_1^+a_{-1}^+}S_1T_0D_0T_0S_1\\
 = &S_1vG_1vS_1+\frac{6}{\|V\|_{L^1}}S_1 vG_{-1}v P vG_{-1}v   S_1
 -384S_1T_0D_0T_0S_1.
\end{split}
\end{equation}
Recall that \eqref{eq-M_3-2} and \eqref{eq-M_3-3}, we can rewrite $M_3^+(\mu)$ as follows:
\begin{equation*}\label{eq-M_3-4}
\begin{split}
M_3^+(\mu)=& a_1^+\mu T_1 +\mu^2S_1B_2^1S_1 +\mu^3S_1B_3^1S_1 +\mu^4S_1B_4^1S_1 +
 S_1\Gamma_5(\mu)S_1\\
 =&a_1^+\mu\Big(T_1+\frac{1}{a_1^+}(\mu S_1B_2^1S_1 +\mu^2S_1B_3^1S_1 +\mu^3S_1B_4^1S_1)+S_1\Gamma_4(\mu)S_1\Big)\\
 :=&a_1^+\mu \widetilde{M_3^+}(\mu).
\end{split}
\end{equation*}
Since $T_1$ is invertible on $S_1L^2(\mathbf{R}),$ we have $S_1T_1^{-1}=T_{1}^{-1}S_1=T_1^{-1},$ then using Neumann series, we have the following asymptotic expansions of $\Big(\widetilde{M_3^+}(\mu)\Big)^{-1}$ as $\mu\rightarrow 0:$
\begin{equation*}\label{eq-inverse-M_3-1}
\begin{split}
\Big(\widetilde{M_3^+}(\mu)\Big)^{-1}
=\ & S_1A_{01}^1S_1+\mu S_1A_{11}^1S_1+\mu^2S_1A_{21}^1S_1+\mu^3S_1A_{31}^1S_1+\Gamma_4(\mu).
\end{split}
\end{equation*}
Then, we have
\begin{equation*}\label{eq-inverse-M_3-2}
\begin{split}
\Big(M_3^+(\mu)\Big)^{-1}=&\frac{S_1A_{-1,1}^1S_1}{\mu}+S_1A_{01}^1S_1+\mu S_1A_{11}^1S_1+\mu^2S_1A_{21}^1S_1+\Gamma_3(\mu).
\end{split}
\end{equation*}
Next, by Lemma \ref{Feshbach-formula} and using \eqref{eq-orthogonality-2}, we have
\begin{equation*}\label{eq-inverse-M_2-4}
\begin{split}
\Big(\widetilde{M_2^+}(\mu)\Big)^{-1}
=&\Big(\widetilde{M_2^+}(\mu)+S_1\Big)^{-1}+\Big(\widetilde{M_2^+}(\mu)+S_1\Big)^{-1}S_1\Big(M_3^+(\mu)\Big)^{-1}S_1\Big(\widetilde{M_2^+}(\mu)+S_1\Big)^{-1}\\
=&\frac{S_1A_{-1,1}^1S_1}{\mu}+S_0A_{01}^1S_0+\mu S_0A_{11}^1S_0+\mu^2S_0A_{21}^1S_0+\Gamma_3(\mu).
\end{split}
\end{equation*}
Furthermore, we have
\begin{equation*}\label{eq-inverse-M_2-5}
\Big(M_2^+(\mu)\Big)^{-1}=\frac{S_1A_{-2,1}^1S_1}{\mu^2}+\frac{S_0A_{-1,1}^1S_0}{\mu}+ S_0A_{01}^1S_0+\mu S_0A_{11}^1S_0+\Gamma_2(\mu).
\end{equation*}
Then using the orthogonality properties (see \eqref{eq-orthogonality-1},\eqref{eq-orthogonality-1}), and  following the similar process of \textbf{Regular Case}, we obtain that
\begin{equation}\label{eq-inverse-M-First Kind}
\begin{split}
\Big(M^+(\mu)\Big)^{-1}=&\frac{S_1A_{-1,1}^1S_1}{\mu}+S_0A_{01}^1Q+QA_{02}^1S_0+\mu \big(QA_{11}^1Q+S_0A_{12}^1+A_{13}^1S_0\big)\\
&+\mu^2(QA_{21}^1+A_{22}^1Q)+\Gamma_3^1(\mu).
\end{split}
\end{equation}

(3)~\textbf{Second Kind of Resonance.}~In this case, we using $\eqref{eq-M-expansions-1-3}$ with $|V(x)|\lesssim(1+|x|)^{-25-},$ see Lemma \ref{lemma-M expansion}. Then by the same process, we can rewrite $M_2^+(\mu)$ as follows:
\begin{equation*}\label{eq-M_2-2-3}
   \begin{split}
M_2^+(\mu)=&\mu S_0B_1^0S_0 +\sum_{k=2}^8\mu^kB_k^0+S_0\Gamma_{10}(\mu)S_0\\
   =& \frac{\mu}{a_{-1}^+}\Big( S_0T_0S_0 + a_{-1}^+\sum_{k=2}^5\mu^{k-1}B_k^0+S_0\Gamma_{9}(\mu)S_0\Big)\\
   :=&\frac{\mu}{a_{-1}^+}\widetilde{M_2^+}(\mu).
	\end{split}
\end{equation*}
Furthermore, we obtain a more detail expansion of $M_3^+(\mu):$
\begin{equation*}\label{eq-M_3-4-2}
\begin{split}
M_3^+(\mu)=& a_1^+\mu T_1 +\sum_{k=2}^8\mu^kS_1B_k^1S_1+S_1\Gamma_9(\mu)S_1\\
 =&a_1^+\mu\Big(T_1+\frac{1}{a_1^+}(\mu S_1B_2^1S_1 +\sum_{k=2}^8\mu^{k-1}S_1B_k^1S_1+S_1\Gamma_8(\mu)S_1\Big)\\
 :=&a_1^+\mu \widetilde{M_3^+}(\mu),
\end{split}
\end{equation*}
where $B_2^1,B_3^1,B_4^1$ are given in \eqref{eq-defi-B^1} and $B_k^1 (5\le k\le8)$ are also bounded operators in $L^2(\mathbf{R}).$

For the second kind of resonance, we aim to derive the expansions of $\Big(\widetilde{M_3^+}(\mu)\Big)^{-1}.$ By the definition of the second kind of resonance, we know that the operator $T_1$ is not invertible on $S_1L^2(\mathbf{R}).$ Since $S_2$ is the Riesz projection onto the kernel of $T_1$ on $S_1L^2(\mathbf{R}),$ thus $T_1+S_2$ is invertible on $S_1L^2(\mathbf{R}).$ Using Neumann series, we have
\begin{equation}\label{eq-inverse-M_3-3}
\begin{split}
\Big(\widetilde{M_3^+}(\mu)+S_2\Big)^{-1}=&D_2- \sum_{k=1}^7\mu^kB_k^2
  +\Gamma_8(\mu),
\end{split}
\end{equation}
where~$B_k^2(1\le k\le 7)$~are bounded  operators in $S_1L^2(\mathbf{R}),$ and $B_1^2,B_2^2,B_3^2$ are given as follows:
\begin{equation}\label{eq-defi-B^2}
\begin{split}
B_1^2=&\frac{1}{a_1^+}D_2B_2^1D_2,\ \ \ \ \ \ B_2^2= \frac{1}{a_1^+} D_2B_3^1D_2 -\frac{1}{(a_1^+)^2}D_2B_2^1D_2B_2^1D_2,\\
B_3^2=& \frac{1}{a_1^+} D_2B_4^1D_2 +\frac{1}{(a_1^+)^3}D_2B_2^1D_2B_2^1D_2-
 \frac{1}{(a_1^+)^2}\big( D_2B_2^1D_2B_3^1D_2+ D_2B_3^1D_2B_2^1D_2 \big).
\end{split}
\end{equation}
By Lemma \ref{Feshbach-formula}, we know that $ \widetilde{M_3^+}(\mu)$ is invertible on
 $S_1L^2(\mathbf{R})$ if and only if
 \begin{equation*}\label{eq-M_4-1}
 M_4^+(\mu):= S_2 -S_2 \big(\widetilde{M_3^+}(\mu)+S_2 \big)^{-1}S_2
 \end{equation*}
 is invertible on $S_2L^2(\mathbf{R})$. Moreover, using \eqref{eq-inverse-M_3-3} and noting that
 $S_2D_2=D_2S_2=S_2$, we obtain
\begin{equation}\label{eq-M_4-2}
\begin{split}
M_4^+(\mu)=&\sum_{k=1}^7\mu^kS_2B_k^2S_2+\Gamma_8(\mu).
\end{split}
\end{equation}
Indeed, using the orthogonal properties of $S_2$ (see \eqref{eq-orthogonality-2}-\eqref{eq-orthogonality-5}), we obtain that
~$S_2B_1^2 S_2=0$ and
\begin{equation}\label{eq-B_2^2-2}
\begin{split}
S_2B_2^2S_2=\frac{a_3^+}{a_1^+}&\Big(S_2vG_3vS_2-\frac{1}{a_2^+\widetilde{a}^+}S_2T_0^2S_2-\frac{(a_1^+)^2}{a_{-1}^+a_3^+}S_2vG_1vD_0vG_1vS_2\\
    & \ \ \ \ \ \ \ \ \ \ \ \ \ \ \ \ \ \  - \frac{(a_{-1}^+)^2}{a_1^+a_3^+(\widetilde{a}^+)^2}  S_2T_0vG_{-1}v D_2 v G_{-1}v T_0S_2\Big)\\
    =\frac{a_3^+}{a_1^+}&\Big(S_2vG_3vS_2-\frac{8\cdot(6!)}{\|V\|_{L^1}}S_2T_0^2S_2+\frac{5}{2}S_2vG_1vD_0vG_1vS_2\\
    & +\frac{48\cdot(6!)}{\|V\|_{L^1}^2}\Big(S_2T_0vG_{-1}v-\frac{\|V\|_{L^1}}{6}S_2vG_1vD_0T_0\Big)D_2\Big(v G_{-1}v T_0S_2-\frac{\|V\|_{L^1}}{6}T_0D_0vG_1vS_2\Big)\\
    :=\frac{a_3^+}{a_1^+}& T_2.
\end{split}
\end{equation}
Furthermore, we can rewrite \eqref{eq-M_4-2} as follows:
\begin{equation*}\label{eq-M_4-3}
\begin{split}
M_4^+(\mu)=\ &\frac{a_3^+}{a_1}\mu^2 T_2+\sum_{k=3}^7\mu^k S_2B_k^2S_2+S_2\Gamma_8(\mu)S_2\\
=\ &\frac{a_3^+}{a_1^+}\mu^2\Big(T_2+\frac{a_1^+}{a_2^+}\sum_{k=3}^7\mu^{k-2} S_2B_k^2S_2+S_2\Gamma_6(\mu)S_2\Big)\\
:=\ &\frac{a_3^+}{a_1^+}\mu^2\widetilde{M_4^+}(\mu).
\end{split}
\end{equation*}
According  to  Lemma \ref{lemma-projection space} and Lemma \ref{lemma-S_3}, we know that $\ker T_2=\{0\},$ which means that $T_2$ is invertible on $S_2L^2(\mathbf{R}),$ by Neumann series and  using $T_2^{-1}=S_2T_2^{-1}=T_2^{-1}S_2$, we have the following asymptotic expansions of $\Big(\widetilde{M_4^+}(\mu)\Big)^{-1}$ as $\mu\rightarrow 0:$
\begin{equation*}\label{eq-inverse-M_4-1}
\begin{split}
\Big(\widetilde{M_4^+}(\mu)\Big)^{-1}=&S_2A_{01}^2S_2+\sum_{k=1}^5\mu^kS_2A_{k1}^2S_2+\Gamma_6(\mu).
\end{split}
\end{equation*}
Thus we have
\begin{equation*}\label{eq-inverse-M_4-2}
\Big(M_4^+(\mu)\Big)^{-1}=\frac{S_2A_{-2,1}^2S_2}{\mu^2}+\sum_{k=-1}^3\mu^kS_2A_{k1}^2S_2+\Gamma_4(\mu).
\end{equation*}
Then by Lemma \ref{Feshbach-formula} and using the orthogonal relationships \eqref{eq-orthogonality-2},
we have
\begin{equation*}\label{eq-inverse-M_3-4}
\begin{split}
\Big(\widetilde{M_3^+}(\mu)\Big)^{-1}
=&\Big(\widetilde{M_3^+}(\mu)+S_2\Big)^{-1}+\Big(\widetilde{M_3^+}(\mu)+S_2\Big)^{-1}S_2\Big(M_4^+(\mu)\Big)^{-1}S_2\Big(\widetilde{M_3^+}(\mu)+S_2\Big)^{-1}\\
=&\frac{S_2A_{-2,1}^2S_2}{\mu^2}+\frac{S_2{A}_{-1,1}^2S_1+S_1{A}_{-1,2}^2S_2}{\mu}+\sum_{k=0}^3\mu^kS_1A_{k1}^2S_1+\Gamma_4(\mu).
\end{split}
\end{equation*}
Next, we can obtain that
\begin{equation*}\label{eq-inverse-M_3-5}
\begin{split}
\Big(M_3^+(\mu)\Big)^{-1}
=&\frac{S_2A_{-3,1}^2S_2}{\mu^3}+\frac{S_2{A}_{-2,1}^2S_1+S_1{A}_{-2,2}^2S_2}{\mu^2}+\sum_{k=-1}^2\mu^kS_1A_{k1}^2S_1+\Gamma_3(\mu).
\end{split}
\end{equation*}
Using the orthogonal relationships \eqref{eq-orthogonality-2} and Lemma \ref{Feshbach-formula}, we have
\begin{equation*}\label{eq-inverse-M_2-6}
\begin{split}
\Big(\widetilde{M_2^+}(\mu)\Big)^{-1}
=&\Big(\widetilde{M_2^+}(\mu)+S_1\Big)^{-1}+\Big(\widetilde{M_2^+}(\mu)+S_1\Big)^{-1}S_1\Big(M_3^+(\mu)\Big)^{-1}S_1\Big(\widetilde{M_2^+}(\mu)+S_1\Big)^{-1}\\
=&\frac{S_2A_{-3,1}^2S_2}{\mu^3}+\frac{S_2{A}_{-2,1}^2S_0+S_0{A}_{-2,2}^2S_2}{\mu^2}+\sum_{k=-1}^2\mu^kS_0A_{k1}^2S_0+\Gamma_3(\mu).
\end{split}
\end{equation*}
Furthermore, since $\displaystyle M_2^+(\mu)=\frac{\mu}{a_{-1}^+}\widetilde{M_2^+}(\mu)$, then
\begin{equation*}\label{eq-inverse-M_2-7}
\begin{split}
\Big({M_2^+}(\mu)\Big)^{-1}
=&\frac{S_2A_{-4,1}^2S_2}{\mu^4}+\frac{S_2{A}_{-3,1}^2S_0+S_0{A}_{-3,2}^2S_2}{\mu^3}+\frac{S_0A_{-2,1}^2S_0}{\mu^2}+\frac{S_0A_{-1,1}^2S_0}{\mu}\\
&\ \ +\mu S_0A_{01}^2S_0+\Gamma_2(\mu).
\end{split}
\end{equation*}
Next, using Lemma \ref{Feshbach-formula} and the orthogonal properties of $S_2$ (see \eqref{eq-orthogonality-1}-\eqref{eq-orthogonality-5}), we have
\begin{equation*}\label{eq-inverse-M_1-4}
\begin{split}
\Big(\widetilde{M_1^+}(\mu)\Big)^{-1}
=&\Big(\widetilde{M_1^+}(\mu)+S_0\Big)^{-1}+\Big(\widetilde{M_1^+}(\mu)+S_0\Big)^{-1}S_0\Big(M_2^+(\mu)\Big)^{-1}S_0\Big(\widetilde{M_1^+}(\mu)+S_0\Big)^{-1}\\
=&\frac{S_2A_{-4,1}^2S_2}{\mu^4}+\frac{S_2{A}_{-3,1}^2S_0+S_0{A}_{-3,2}^2S_2}{\mu^3}+\frac{S_0A_{-2,1}^2S_0+QA_{-2,2}^2S_2+S_2A_{-2,3}^2Q}{\mu^2}\\
&\ \ +\frac{QA_{-1,1}^2S_0+S_0A_{-1,2}^2Q}{\mu}+QA_{01}^2Q+ \mu QA_{11}^2Q+\Gamma_2(\mu).
\end{split}
\end{equation*}
Since $\displaystyle M_1^+(\mu)=\frac{a_{-1}^+}{\widetilde{a}^+}\widetilde{M}_1^+(\mu),$ then
\begin{equation*}\label{eq-inverse-M_1-5}
\begin{split}
\Big({M_1^+}(\mu)\Big)^{-1}
=&\frac{S_2A_{-6,1}^2S_2}{\mu^6}+\frac{S_2{A}_{-5,1}^2S_0+S_0{A}_{-5,2}^2S_2}{\mu^5}+\frac{S_0A_{-4,1}^2S_0+QA_{-4,2}^2S_2+S_2A_{-4,3}^2Q}{\mu^4}\\
&\ \ +\frac{QA_{-3,1}^2S_0+S_0A_{-3,2}^2Q}{\mu^3}+\frac{QA_{-2,1}^2Q}{\mu^2}+\frac{QA_{11}^2Q}{\mu}+\Gamma_1(\mu).
\end{split}
\end{equation*}
Finally, using the orthogonality relationships (see \eqref{eq-orthogonality-1}-\eqref{eq-orthogonality-5}), we have
\begin{equation*}\label{eq-inverse-M}
\begin{split}
\Big(\widetilde{M^+}(\mu)\Big)^{-1}
=&\Big(\widetilde{M^+}(\mu)+Q\Big)^{-1} +\Big(\widetilde{M^+}(\mu)+Q\Big)^{-1} Q\Big(M_1^+(\mu)\Big)^{-1}Q\Big(\widetilde{M^+}(\mu)+Q\Big)^{-1} \\
=& \frac{S_2A_{-6,1}^2S_2}{\mu^6}+\frac{S_2{A}_{-5,1}^2S_0+S_0{A}_{-5,2}^2S_2}{\mu^5}+\frac{S_0A_{-4,1}^2S_0+QA_{-4,2}^2S_2+S_2A_{-4,3}^2Q}{\mu^4}\\
&+\frac{QA_{-3,1}^2S_0+S_0A_{-3,2}^2Q+S_2A_{-3,3}^2+A_{-3,4}^2S_2}{\mu^3}+\frac{QA_{-2,1}^2Q+S_0A_{-2,2}^2+A_{-2,3}^2S_0}{\mu^2}\\
&+\frac{QA_{-1,1}^2+A_{-1,2}^2}{\mu}+\Gamma_0(\mu).
\end{split}
\end{equation*}
Furthermore, we obtain the following asymptotic expansions of $\Big(M^+(\mu)\Big)^{-1}$ for zero is the second resonance of $H$,
\begin{equation}\label{eq-inverse-M-Second Kind}
\begin{split}
\Big(M^+(\mu)\Big)^{-1}=&\frac{S_2A_{-3,1}^2S_2}{\mu^3}+\frac{S_2A_{-2,1}^2S_0+S_0A_{-2,2}^2S_2}{\mu^{2}}+\frac{S_0A_{-1,1}^2S_0+QA_{-1,2}^2S_2+S_2A_{-1,3}^2Q}{\mu}\\
&\ \ +\big(QA_{01}^2S_0+S_0A_{02}^2Q+S_2A_{03}^2+A_{04}^2S_2\big)+\mu\big(QA_{11}^2Q+S_0A_{12}^2+A_{13}S_0\big)\\
&\ \ +\mu^2(QA_{21}^2+A_{21}^2Q)+\Gamma_3^2(\mu).
\end{split}
\end{equation}

\bigskip

\section{Proof of Theorem \ref{thm-equivalence}}\label{resonance space}
In this part, we aim to give the proof of Theorem \ref{thm-equivalence}. In fact, it suffices to establish the following characterizations of resonance spaces.
\begin{theorem}\label{thm-resonance space}
Assume that $|V(x)|\lesssim (1+|x|)^{-\beta}$ with some $\beta>17$. Then the following statements hold:

(i)~$f(x)\in S_1L^2(\mathbf{R})$ if and only if $f(x)=Uv\phi$  with $\phi\in W_{\frac{3}{2}}(\mathbf{R})$ such that $H\phi=0$ in the distributional sense. Furthermore,
$$\phi= -G_0vf +c_1x+c_2,$$
where $$ c_1= \frac{ \langle v_1,T_0f \rangle}{\|v_1\|_{L^2(\mathbf{R})}}  \ \ \
 \hbox{and}\ \ \ c_2= \frac{ \langle v,T_0f \rangle}{\|V\|_{L^1(\mathbf{R})}}-
\frac{ \langle xv, v \rangle}{\|V\|_{L^1(\mathbf{R})}}c_1 $$
with
 $ \displaystyle v_1= xv- \frac{ \langle xv, v \rangle}{\|V\|_{L^1(\mathbf{R})}}v \in(Q-S_0)L^2(\mathbf{R}).$

(ii)~$f(x)\in S_2L^2(\mathbf{R})$ if and only if $f(x)=Uv\phi$  with $\phi\in W_{\frac{1}{2}}(\mathbf{R})$ such that $H\phi=0$ in the distributional sense. Furthermore,
$$\phi= -G_0vf +\frac{ \langle v,T_0f \rangle}{\|V\|_{L^1(\mathbf{R})}}.$$

\end{theorem}

\begin{proof}(i)~ Assume that $f\in S_1L^2(\mathbf{R})$. One has
$$ 0= S_0T_0S_0f=S_0(U+vG_0v)f. $$
Note that
$$ 0= S_0(U+vG_0v)f= (I-\widetilde{P})(U+vG_0v) f= Uf+vG_0vf- \widetilde{P}T_0f,  $$
where $ \widetilde{P}$ is the projection onto the span$\{v , xv \}= (S_0L^2(\mathbf{R}))^\perp$.
In fact, we know that
$v$ and $\displaystyle v_1=xv- \frac{\langle xv , v \rangle}{\| V\|_{L^1(\mathbf{R})}}v$ are the orthogonal basis of
$ (S_0L^2(\mathbf{R}))^\perp$ . Hence, we have
\begin{equation*}
\begin{split}
\widetilde{P}T_0f=& \frac{\langle v_1,T_0f \rangle}{\| v_1\|_{L^2(\mathbf{R})}}v_1 +
                       \frac{\langle v, T_0f \rangle}{\| V\|_{L^1(\mathbf{R})}}v \\
 =& xv \frac{\langle v_1,T_0f\rangle}{\| v_1\|_{L^2(\mathbf{R})}} +
     + v\Big(\frac{\langle v, T_0f\rangle}{\| V\|_{L^1(\mathbf{R})}}
     - \frac{\langle xv , v \rangle}{\| V\|_{L^1(\mathbf{R})}} \cdot
           \frac{\langle v_1,T_0f\rangle}{\| v_1\|_{L^2(\mathbf{R})}}  \Big)\\
 :=& c_1xv+c_2v.
\end{split}
\end{equation*}
Furthermore, we have
$$Uf= -vG_0vf + \widetilde{P}T_0f = -vG_0vf +c_1xv+ c_2v.$$
Since $U^2=1,$ then
$$f= U^2f= Uv(-G_0vf)+U(c_1xv+c_2v)= Uv(-G_0vf +c_1x+c_2):= Uv\phi,$$
and $vf= vUv\phi=V\phi$. By $G_0=\big(\Delta^2\big)^{-1}$ and noting that $\Delta^2(c_1x+c_2)=0$,  we have
$$H\phi=\big(\Delta^2+V\big)\phi = \Delta^2\phi+ V\phi = -vf +V\phi=0.$$
Next, we show that $\phi \in W_{\frac{3}{2}}(\mathbf{R})$.
It is obvious that $ c_1x+c_2 \in W_{\frac{3}{2}}(\mathbf{R})$, so it suffices to show that
  $$G_0vf=  \frac{1}{12}\int_{\mathbf{R}}|x-y|^3v(y)f(y)dy \in W_{\frac{3}{2}}(\mathbf{R}).$$
Since $S_1 \leq S_0,$ and we know that if $f\in S_0L^2(\mathbf{R})$, then
$$ \int_{\mathbf{R}} v(y)f(y)dy =0,\ \ \hbox{and} \ \ \int_{\mathbf{R}} yv(y)f(y)dy =0.$$
Hence
$$ G_0vf= \frac{1}{12}\int_{\mathbf{R}}\big( |x-y|^3-|x|^3 +3|x| x\cdot y \big) v(y)f(y)dy
:= \frac{1}{12}\int_{\mathbf{R}}K_1(x, y) v(y)f(y)dy$$
with $ K_1(x,y)= |x-y|^3-|x|^3 +3|x| x\cdot  y $.

Indeed, we can divided $K_1(x,y)$ into three parts as follows:
\begin{equation*}
\begin{split}
K_1(x,y)=& |x-y|^3-|x|^3 +3|x| x\cdot  y\\
=& \Big(x^2(|x-y| -|x|)  +|x|x\cdot y   \Big)  + y^2( |x-y|-|x| )\\
 &\ \  +
     \Big( y^2|x|-2x \cdot y ( |x-y|-|x| )  \Big)\\
:=&K_{11}(x,y)+ K_{12}(x,y)+ K_{13}(x,y).
\end{split}
\end{equation*}
For the first part $K_{11}(x,y),$ we have
\begin{equation*}
\begin{split}
K_{11}(x,y)&=\left|\frac{x^2(y^2-2xy)}{|x-y|+|x|}  + \frac{x^3y +xy|x|\cdot |x-y|}{|x-y|+|x|} \right|\\
   &\leq \frac{x^2y^2}{|x-y|+|x|} + \left|\frac{|x|xy( |x-y|-|x|)}{|x-y|+|x|}  \right|\\
 & \lesssim |x|y^2.
\end{split}
\end{equation*}
Thus, we have
$$ \Big|\int_{\mathbf{R}} K_{11}(x,y)v(y)f(y)dy \Big| \in W_{\frac{3}{2}}(\mathbf{R}). $$
Then, for  the second part $K_{22}(x,y),$ we have
$$|K_{12}(x,y)| = \Big| y^2 (|x-y|-|x| )\Big|\lesssim |y|^3. $$
It is obvious that
$$\Big|\int_{\mathbf{R}}K_{12}(x,y) v(y)f(y)dy \Big| \lesssim \int_{\mathbf{R}}|y|^3 v(y)f(y)dy
 \in  W_{\frac{3}{2}}(\mathbf{R}).$$
Finally, for the last part $K_{13}(x,y ),$ we have
$$| K_{13}(x,y)| = \Big|~~ y^2|x|-2xy(|x-y|-|x|)  \Big| \leq |x|y^2 + 2|x|y^2 \lesssim |x|y^2.$$
It is easy to know that
$$\Big|\int_{\mathbf{R}}K_{13}(x,y) v(y)f(y)dy \Big| \lesssim |x|\int_{\mathbf{R}}|y|^2 v(y)f(y)dy
 \in   W_{\frac{3}{2}}(\mathbf{R}).$$
Hence, we obtain that
\begin{equation}\label{eq-G_0-1}
\begin{split}
|G_0vf|=&\Big| \frac{1}{12} \int_{\mathbf{R}}K_{1}(x,y) v(y)f(y)dy   \Big|\\
    \lesssim & \Big|  \int_{\mathbf{R}}K_{11}(x,y) v(y)f(y)dy \Big|
      + \Big|  \int_{\mathbf{R}}K_{12}(x,y) v(y)f(y)dy   \Big|\\
    &+ \Big|  \int_{\mathbf{R}}K_{13}(x,y) v(y)f(y)dy   \Big|,
\end{split}
\end{equation}
which implies that $G_0vf\in W_{\frac{3}{2}}(\mathbf{R}).$

On the other hand, let $\phi\in W_{\frac{3}{2}}(\mathbf{R})$ such that $H\phi=0.$ We first show that $f=Uv\phi\in S_0L^2(\mathbf{R}),$ that is,
\[ \int_{\mathbf{R}}v(y)f(y)dy=0 \ \  \hbox{and} \ \    \int_{\mathbf{R}}yv(y)f(y)dy=0.\]
Noting that~$vf=V\phi=-\Delta^2\phi$.
Let $\eta(x)\in C_0^\infty(\mathbf{R})$ such that $\eta(x)=1$ for $|x|\le 1$ and $\eta(x)=0$ for $|x|>2.$ Then for any $\delta>0,$ we have
\begin{equation}\label{eq-orthogonality proof}
\begin{split}
\Big|\int_{\mathbf{R}}v(y)f(y)\eta(\delta y)dy\Big|=&\Big|\int_{\mathbf{R}}\big(\Delta^2\phi(y)\big)\eta(\delta y)dy\Big|
=\Big|\int_{\mathbf{R}}\phi(y)\delta^4\eta^{(4)}(\delta y)dy\Big|\\
\lesssim& \delta^4 \big\|\langle y\rangle^{-\frac{3}{2}-}\phi(y)\big\|_{L^2(\mathbf{R})}\cdot\|\langle y\rangle^{\frac{3}{2}+}\eta^{(4)}(\delta y)\|_{L^2(\mathbf{R})}\\
\lesssim& \delta^{\frac{5}{2}-} \big\|\langle y\rangle^{-\frac{3}{2}-}\phi(y)\big\|_{L^2(\mathbf{R})}\cdot\|\eta^{(4)}( y)\|_{L^2(\mathbf{R})}.
\end{split}
\end{equation}
Taking the limit $\delta\rightarrow0$ and using the dominated convergence theorem we obtain $\displaystyle \int_{\mathbf{R}}v(y)f(y)dy=0.$ Similarly, we have $\displaystyle \int_{\mathbf{R}}v(y)f(y)dy=0.$ Thus, $f=Uv\phi\in S_0L^2(\mathbf{R}).$

Now we turn to  prove that $\phi=-G_0vf+c_1x+c_2.$ Let $\widetilde{\phi}=\phi+G_0vf,$ then by assumption and \eqref{eq-G_0-1}, $\widetilde{\phi}\in  W_{\frac{3}{2}}(\mathbf{R})$ and $\Delta^2\widetilde{\phi}=0,$ thus $\widetilde{\phi}=\widetilde{c}_1x+\widetilde{c}_2.$ This implies that $\phi=-G_0vf+\widetilde{c}_1x+\widetilde{c}_2.$

Since
\begin{align*}
0=H\phi=\big[\Delta^2+V\big]\phi=&-\big[\Delta^2+V\big](G_0vf)+\big[\Delta^2+V\big](\widetilde{c}_1x+\widetilde{c}_2)\\
=&-Uv\big(U+vG_0v\big)f+U\big(\widetilde{c}_1xv^2+\widetilde{c}_2v^2\big)\\
=&U\big(-vT_0f+\widetilde{c}_1xv^2+\widetilde{c}_2v^2\big),
\end{align*}
thus we have
\begin{align}\label{eq-c-1}
\widetilde{c}_1xv^2+\widetilde{c}_2v^2=vT_0f\ \ \Longrightarrow&\ \ \widetilde{c}_1\langle xv,v\rangle+\widetilde{c}_2\langle v,v\rangle=\langle v,T_0f\rangle,\\ \label{eq-c-2}
\widetilde{c}_1x^2v^2+\widetilde{c}_2xv^2=xvT_0f\ \ \Longrightarrow&\ \ \widetilde{c}_1\langle xv,xv\rangle+\widetilde{c}_2\langle xv,v\rangle=\langle xv,T_0f\rangle.
\end{align}
By \eqref{eq-c-1}, we have
\begin{equation}\label{eq-c-3}
\widetilde{c_2}=\frac{\langle v,T_0f}{\|V\|_{L^1(\mathbf{R})}}-\widetilde{c}_1\frac{\langle xv,v\rangle}{\|V\|_{L^1(\mathbf{R})}}.
\end{equation}
Then, substituting \eqref{eq-c-3}, $\displaystyle xv=v_1+\frac{ \langle xv, v \rangle}{\|V\|_{L^1(\mathbf{R})}}v$ and $\langle v_1,v\rangle=0$ into \eqref{eq-c-2}, we have
\[ \widetilde{c}_1= \frac{ \langle v_1,T_0f \rangle}{\|v_1\|_{L^2(\mathbf{R})}}  \ \ \
 \hbox{and}\ \ \ \widetilde{c}_2= \frac{ \langle v,T_0f \rangle}{\|V\|_{L^1(\mathbf{R})}}-
\frac{ \langle xv, v \rangle}{\|V\|_{L^1(\mathbf{R})}}\widetilde{c}_1, \]
which means that $ \widetilde{c}_1=c_1, \widetilde{c}_2=c_2.$ Furthermore,  using $f=Uv\phi$ and $H\phi=0$, we have
\begin{equation}\label{eq-S_0}
\begin{split}
 S_0T_0S_0f & = S_0T_0(Uv\phi)=S_0(U+vG_0v)(Uv\phi)= S_0(v\phi+vG_0vUv\phi)\\
&=S_0vG_0\big(\Delta^2 +V\big)\phi= S_0vG_0H\phi = 0.
\end{split}
\end{equation}
Hence, we obtain that $f\in S_1L^2(\mathbf{R})$. The proof of Theorem \ref{thm-resonance space} (i) is completed.

(ii)~ Since $ S_2\leq S_1 $, by Theorem \ref{thm-resonance space}(i), we know that $f\in S_1L^2(\mathbf{R})$ if and
only if $f=Uv\phi$ with $ \phi \in W_{\frac{3}{2}}(\mathbf{R})$  such  that $H\phi =0$ in the distributional sense and $$ \phi = -G_0v f +c_1x +c_2.$$
Thus, we just need to  prove that $c_1=0$ and $\phi  \in W_{\frac{1}{2}}(\mathbf{R}) $. Indeed,
by the proof of Theorem \ref{thm-resonance space}(i), we know that
$$v_1= xv- \frac{ \langle xv, v \rangle}{\| V \|_{L^1(\mathbf{R})}} v \in QL^2(\mathbf{R}).$$
By Table \ref{table-1},  if $ f\in S_2L^2(\mathbf{R})$, then $QT_0f = 0,$  thus
\[c_1=\langle v_1, T_0f \rangle=0.\]

Next we show that $ \phi \in W_{\frac{1}{2}}(\mathbf{R})$. Since
$ \phi = -G_0vf + \displaystyle\frac{ \langle v, T_0f \rangle}{\| V \|_{L^1(\mathbf{R})}},$
thus it suffices to show that
$$ G_0vf= \frac{1}{12}\int_{\mathbf{R}} |x-y|^3 v(y)f(y)dy  \in W_{\frac{1}{2}}(\mathbf{R}).$$
We obtain that if  $f \in S_2L^2(\mathbf{R})$, then
$$ \int_{\mathbf{R}}v(y)f(y)dy= \int_{\mathbf{R}} y v(y)f(y)dy=
 \int_{\mathbf{R}}y^2v(y)f(y)dy=0.  $$
Thus,
\begin{align*}
G_0vf= &\frac{1}{12}\int_{\mathbf{R}}\big( |x-y|^3-|x|^3 + 3|x|x\cdot y - 3|x|y^2 \big)v(y)f(y)dy\\
:= & \frac{1}{12}\int_{\mathbf{R}}K_2(x,y)v(y)f(y)dy
\end{align*}
with $K_2(x,y)= |x-y|^3-|x|^3 + 3|x|x\cdot y - 3|x|y^2 $.

Indeed, we can divided $K_2(x,y)$ into four parts as follows
\begin{equation*}
\begin{split}
K_2(x,y)=& y^2\big( |x-y|-|x| \big)- 2\big[ xy( |x-y|-|x|)+ |x|y^2   \big]
             + \Big[ \frac{x^2y^2}{|x-y|+|x|} -\frac{1}{2}|x|y^2  \Big]\\
             &+\Big[ \frac{|x|x\cdot y (|x-y|-|x|)}{|x-y|+|x|} +\frac{1}{2}|x|y^2  \Big]\\
             :=& K_{21}(x,y)+ K_{22}(x,y)+K_{23}(x,y)+K_{24}(x,y).
\end{split}
\end{equation*}
For the first part $K_{21}(x,y),$ we have
$$ |K_{21}(x,y)|= \big|y^2(|x-y|-|x|) \big| \lesssim |y|^3.$$
Thus
$$ \Big|\int_{\mathbf{R}} K_{21}(x,y)v(y)f(y)dy \Big|
 \lesssim  \int_{\mathbf{R}} |y|^3 |v(y)f(y)|dy \in W_{\frac{1}{2}}(\mathbf{R}). $$
For the second part $K_{22}(x,y),$ we have
\begin{equation*}
\begin{split}
|K_{22}(x,y)| =& 2\Big|\frac{ x\cdot y( y^2 -2x\cdot y)}{|x-y|+|x|}
                + \frac{ |x|y^2( |x-y| +|x|)}{|x-y|+|x|}   \Big|\\
   \lesssim &  \frac{ |x|\cdot |y|^3}{|x-y|+|x|}
                + \Big| \frac{ |x|y^2( |x-y|-|x|)}{|x-y|+|x|}\Big| \\
                \lesssim & |y|^3.
\end{split}
\end{equation*}
Thus
$$ \Big|\int_{\mathbf{R}} K_{22}(x,y)v(y)f(y)dy \Big|
 \lesssim  \int_{\mathbf{R}} |y|^3 |v(y)f(y)|dy \in W_{\frac{1}{2}}(\mathbf{R}). $$
Then, for the third part $K_{23}(x,y)$, we have
\begin{equation*}
\begin{split}
|K_{23}(x,y)| =& \Big|\frac{ 2x^2y^2}{2(|x-y|+|x|)}
                - \frac{ |x|y^2(|x-y|+|x|)}{2(|x-y|+|x|)}   \Big|\\
   = &  \Big|\frac{ |x|y^2 \cdot |(|x-y|-|x|)|}{2(|x-y|+|x|)}\Big|\\
                \lesssim & |y|^3.
\end{split}
\end{equation*}
Thus
$$ \Big|\int_{\mathbf{R}} K_{23}(x,y)v(y)f(y)dy \Big|
 \lesssim  \int_{\mathbf{R}} |y|^3 |v(y)f(y)|dy \in W_{\frac{1}{2}}(\mathbf{R}). $$
Finally, for the fourth part $K_{24}(x,y)$, we have
\begin{equation*}
\begin{split}
|K_{24}(x,y)| =& \Big|\frac{ |x|x\cdot y (y^2-2x\cdot y)}{(|x-y|+|x|)^2}
                + \frac{ |x|y^2(|x-y|+|x|)^2}{2(|x-y|+|x|)^2}   \Big|\\
   \leq &  \frac{ |x|y^4}{2(|x-y|+|x|)^2} +  \Big| \frac{ x^2y^2 (|x-y| -|x|)}{(|x-y|+|x|)^2}\Big|\\
                \lesssim & |y|^3.
\end{split}
\end{equation*}
Thus
$$ \Big|\int_{\mathbf{R}} K_{24}(x,y)v(y)f(y)dy \Big|
 \lesssim  \int_{\mathbf{R}} |y|^3 \cdot |v(y)f(y)|dy
 \in W_{\frac{1}{2}}(\mathbf{R}). $$
Hence, we obtain that  $G_0vf \in W_{\frac{1}{2}}(\mathbf{R})$. Furthermore, we have
$ \phi \in W_{\frac{1}{2}}(\mathbf{R})$.

On the other hand, let $\phi(x)\in W_{\frac{1}{2}}(\mathbf{R})$ such that $H\phi(x)=0.$ Denote $f=Uv\phi,$ with the similar method of \eqref{eq-orthogonality proof} and \eqref{eq-S_0}, we have
\[\int_{\mathbf{R}}v(y)f(y)dy=0,\
    \int_{\mathbf{R}}yv(y)f(y)dy=0,\
     \int_{\mathbf{R}}y^2v(y)f(y)dy=0, S_0T_0S_0f=0,\]
then we can obtain that $f\in S_1L^2(\mathbf{R})$ and $G_0vf\in W_{\frac{1}{2}}(\mathbf{R}).$ Moreover, by Theorem \ref{thm-resonance space}(i), we know that $\phi=-G_0vf+c_1x+c_2$ since $W_{\frac{1}{2}}(\mathbf{R})\subset W_{\frac{3}{2}}(\mathbf{R}).$ Thus we need that $c_1x=\phi+G_0vf-c_2\in W_{\frac{1}{2}}(\mathbf{R})$
this necessitates that $c_1=0,$ which implies that $T_0f\in \text{span}\{v\}$ and $\phi=-G_0vf+\displaystyle\frac{ \langle v,T_0f \rangle}{\|V\|_{L^1(\mathbf{R})}}$. Hence, we have $f\in S_2L^2(\mathbf{R})$. The proof of Theorem \ref{thm-resonance space}(ii) is completed.
\end{proof}

{\bf Acknowledgements:}  A. Soffer is partially supported by the
Simon's Foundation (No.395767). Z. Wu and X. H. Yao are partially supported by  NSFC (No.11771165 and 12171182). The authors would like to express their sincere gratitude to the reviewing referee for many constructive comments which helped us greatly improve the previous
version. We also would like to thank Professor P. Li for his useful discussions, and thank Professor M. Goldberg for his helpful comments and insights for the previous version.




\end{document}